\documentclass[a4paper, reqno, 10pt]{amsart}

\usepackage{verbatim}
\usepackage{kotex}
\usepackage[dvips]{color}
\usepackage[usenames,dvipsnames]{xcolor}
\usepackage{amssymb}
\usepackage[active]{srcltx}
\usepackage[colorlinks,pdfpagelabels,pdfstartview = FitH,bookmarksopen = true,bookmarksnumbered = true,linkcolor = NavyBlue,plainpages = false,hypertexnames = false,citecolor = OliveGreen] {hyperref}
\usepackage[italian, textsize=tiny, color=BlueGreen, backgroundcolor=BlueGreen, linecolor=BlueGreen]{todonotes}
\usepackage{mathtools}
\usepackage{empheq}
\usepackage{a4wide}
\usepackage[toc,page]{appendix}
\usepackage{aligned-overset}

\newcommand\blfootnote[1]{%
\begingroup
\renewcommand\thefootnote{}\footnote{#1}%
\addtocounter{footnote}{-1}%
\endgroup
}

\title[Self-improving properties]{Self-improving properties for the fractional $p$-Laplacian via nonlinear commutators} 

\author[Lee]{Ho-Sik Lee}
\address{Fakult\"at für Mathematik, Universität Bielefeld, 33615 Bielefeld, Germany}
\email{ho-sik.lee@uni-bielefeld.de}

\author[Song]{Kyeong Song}
\address{School of Mathematics, Korea Institute for Advanced Study, Seoul 02455, Republic of Korea}
\email{kyeongsong@kias.re.kr}

\makeatletter
\@namedef{subjclassname@2020}{\textup{2020} Mathematics Subject Classification}
\makeatother

\subjclass[2020]{
35R11; 
47G20; 
35D30;  
35B65; 
35R05. 
}

\keywords{nonlinear nonlocal operator; fractional $p$-Laplacian; regional fractional $p$-Laplacian; self-improving property; regularity}

\newtheorem{theorem}{Theorem}[section]

\newtheorem{lemma}[theorem]{Lemma}

\newtheorem{corollary}[theorem]{Corollary}
\theoremstyle{definition}

\newtheorem{remark}[theorem]{Remark}

\numberwithin{equation}{section}

\def\eqn#1$$#2$${\begin{equation}\label#1#2\end{equation}}
\def\charfn_#1{{\raise1.2pt\hbox{$\chi_{\kern-1pt\lower3pt\hbox{{$\scriptstyle#1$}}}$}}}

\makeatletter
\newcommand{\pushright}[1]{\ifmeasuring@#1\else\omit\hfill$\displaystyle#1$\fi\ignorespaces}
\newcommand{\pushleft}[1]{\ifmeasuring@#1\else\omit$\displaystyle#1$\hfill\fi\ignorespaces}
\makeatother

\def\ep{\varepsilon}

\newcommand{\vp}{\varphi}
\newcommand{\R}{\mathbb{R}}

\newcommand{\tail}{{\rm Tail}}

\def\diam{\operatorname{diam}}
\def\dist{\operatorname{dist}}

\def\er{\mathbb R}
\newcommand{\ern}{\mathbb{R}^n}

\def\loc{{\operatorname{loc}}}

\newcommand{\supp}{{\rm supp}\,}

\def\mean#1{\mathchoice%
          {\mathop{\kern 0.2em\vrule width 0.6em height 0.69678ex depth -0.58065ex
                  \kern -0.8em \intop}\nolimits_{\kern -0.4em#1}}%
          {\mathop{\kern 0.1em\vrule width 0.5em height 0.69678ex depth -0.60387ex
                  \kern -0.6em \intop}\nolimits_{#1}}%
          {\mathop{\kern 0.1em\vrule width 0.5em height 0.69678ex
              depth -0.60387ex
                  \kern -0.6em \intop}\nolimits_{#1}}%
          {\mathop{\kern 0.1em\vrule width 0.5em height 0.69678ex depth -0.60387ex
                  \kern -0.6em \intop}\nolimits_{#1}}}

\newtoks\by
\newtoks\paper
\newtoks\book
\newtoks\jour
\newtoks\yr
\newtoks\pages
\newtoks\vol
\newtoks\publ

\def\ota{{\hbox{\bf ???}}}
\def\cLear{\by=\ota\paper=\ota\book=\ota\jour=\ota\yr=\ota
\pages=\ota\vol=\ota\publ=\ota}
\def\endpaper{\the\by, \textit{\the\paper},
{\the\jour} \textbf{\the\vol} (\the\yr), \the\pages.\cLear}
\def\endbook{\the\by, \textit{\the\book},
\the\publ, \the\yr.\cLear}
\def\endpap{\the\by, \textit{\the\paper}, \the\jour.\cLear}
\def\endproc{\the\by, \textit{\the\paper}, \the\book, \the\publ,
\the\yr, \the\pages.\cLear}

\begin{document}
\begin{abstract}
We investigate a class of nonlocal equations whose leading operator is modeled on either the fractional $p$-Laplacian or the regional fractional $p$-Laplacian, $p \in (1,\infty)$. We prove local self-improving properties of weak solutions to the fractional $p$-Laplacian in the case $p\in(1,\infty)$ with non-integrable right-hand side, as well as to the regional fractional $p$-Laplacian in the subquadratic case $1 < p < 2$, by extending the nonlinear commutator estimates developed by Schikorra (Math. Ann. 366 (1-2):695--720, 2016).
\end{abstract}

\maketitle

\blfootnote{H.-S. Lee was supported by the Deutsche Forschungsgemeinschaft (DFG, German Research Foundation) through SFB1283/22021–317210226 at Bielefeld University. K. Song was supported by a KIAS individual grant (MG091702) at Korea Institute for Advanced Study. }

\section{Introduction}
This paper is concerned with regularity of weak solutions to nonlinear nonlocal equations with measurable kernels. Especially, we consider either 
\begin{align}\label{eq:fracp}
\mathcal{L}_{s}u(x) \coloneqq \text{p.v.}\int_{\ern}K(x,y)\phi(u(x)-u(y))\,dy=f(x) \quad\text{in }\Omega,
\end{align}
modeled on the \emph{fractional $p$-Laplacian}, or 
\begin{align}\label{eq:regional}
\mathcal{L}_{s,\Omega}u(x) \coloneqq \text{p.v.}\int_{\Omega}K(x,y)\phi(u(x)-u(y))\,dy=f(x) \quad\text{in }\Omega
\end{align}
modeled on the \emph{regional fractional $p$-Laplacian}. Here, $\Omega\subset\ern$ is an open set. Moreover, $K:\ern\times\ern\rightarrow\er$ is a measurable, symmetric kernel satisfying
\begin{align}\label{eq:K}
\dfrac{1}{\Lambda\,|x-y|^{n+sp}}\leq K(x,y)\leq\dfrac{\Lambda}{|x-y|^{n+sp}}
\end{align}
for some constant $\Lambda\geq 1$, and $\phi:\er\rightarrow\er$ is a measurable function satisfying
\begin{align}\label{eq:phi}
|\phi(t)|\leq\lambda|t|^{p-1},\quad\phi(t)t\geq \lambda^{-1}|t|^p\quad(t\in\er)
\end{align}
for some constant $\lambda\geq 1$. Throughout this paper, we assume that $0<s<1<p<\infty$.

The regularity theory of fractional $p$-Laplacian type equations has been a subject of intensive recent research. In particular, for equation \eqref{eq:fracp} under the basic structural assumptions \eqref{eq:K} and \eqref{eq:phi}, H\"older regularity and Harnack inequalities within the framework of the De Giorgi--Nash--Moser theory were established in \cite{DKP16} and \cite{DKP14}, respectively. These results were comprehensively extended to possibly non-differentiable nonlocal functionals in \cite{coz} through the introduction of fractional De Giorgi classes. We also refer to \cite{APT23}, where an approach based on continuous iterations was employed instead of the standard De Giorgi iteration. H\"older regularity for viscosity solutions to \eqref{eq:fracp} was studied in \cite{Lind}. Global H\"older regularity was achieved in \cite{IanMosSqu16} via barrier arguments, and in \cite{KKP16} for obstacle problems. Self-improving properties were investigated in \cite{KMS} for the case $p=2$ and in \cite{Sch16,SM} for the case $p\geq 2$. Potential estimates for measure data problems were established in \cite{KMS15,NOS} using the concept of SOLA (Solutions Obtained as Limits of Approximations), and in \cite{KLL23} utilizing $\mathcal{L}_s$-superharmonic functions. In the case of systems, there are few regularity results, e.g., local boundedness was studied in \cite{BDNS} via shortening operators. Sharp and comprehensive partial regularity results were proven in \cite{DMN25,DMN26} for $p=2$.

The regional fractional Laplacian is intimately connected to censored stable processes in probability theory, which are derived from symmetric stable processes by suppressing jumps from $\Omega$ to $\ern\setminus\Omega$. The pioneering paper \cite{BBC03} constructed several such processes, detailing their properties and establishing the boundary Harnack principle. Within the realm of PDE theory, Ishii and Nakamura \cite{IshNak} proved solvability and demonstrated convergence to the local $p$-Laplacian as $s\nearrow 1$. For further reading, we refer to e.g. \cite{Fal,MouYi15} and the references therein.

In this paper, we are interested in self-improving properties of weak solutions in the realm of Gehring's theory. Surprisingly, it was shown that weak solutions to nonlocal equations with measurable kernels enjoy not only higher integrability but also higher differentiability, which is in sharp contrast to the case of corresponding local equations. 
Such ``nonlocal self-improving properties'' were first observed by Kuusi, Mingione, and Sire \cite{KMS}, where they proved a fractional version of Gehring's lemma by employing dual pairs. While \cite{KMS} considered equations with quadratic growth, the approach therein naturally extends to fractional $p$-Laplacian type equation, $p\ge2$, see \cite{BKK,SM}. It is further applied to Calder\'on--Zygmund type estimates for nonlocal equations \cite{ByuKim25,No1,No2}.

Schikorra \cite{Sch16} considered regional fractional $p$-Laplacian type equations, $p\ge2$, and developed a different approach based on nonlinear commutator estimates to prove similar self-improving properties of weak solutions. Note that \cite{Sch16} also contains applications to self-improving properties of very weak solutions and $\ep$-regularity of fractional harmonic maps. We also mention \cite{ABES,MenSal26} which utilize a functional analytic approach for nonlocal self-improving properties.

The approaches in both of \cite{KMS} and \cite{Sch16} are quite general and extended in several directions.  
However, to the best of our knowledge, there is no result in the literature concerned with self-improving properties for nonlocal equations with subquadratic growth. 
In this paper, we extend the approach in \cite{Sch16} to show self-improving properties for the regional fractional $p$-Laplacian type equation \eqref{eq:regional} with $1<p<2$. In particular, we perform a delicate approximation procedure in order to drop the (implicit) a priori higher regularity assumption in \cite{Sch16}. Moreover, we also establish analogous results for the fractional $p$-Laplacian type equation \eqref{eq:fracp} in bounded domains with $1<p<\infty$. Finally, we remark that even in the linear case $p=2$, the small gain in differentiability and integrability cannot be made arbitrarily large under the mere uniform ellipticity  assumptions  \eqref{eq:K} and \eqref{eq:phi}. This limitation was demonstrated in \cite{BDKL25} via a Meyers type counterexample, inspired by the classical work \cite{Mey63}.

\subsection{Regional fractional $p$-Laplacian with $1<p<2$}
With $f\in (W^{s,p}_0(\Omega))^*$, the dual space of $W^{s,p}_0(\Omega)$ defined in Section \ref{sec:pre}, let $u\in W^{s,p}(\Omega)$ be a weak solution to \eqref{eq:regional}; it is also possible that $\Omega=\ern$. This means that for any $\xi\in C^{\infty}_{c}(\Omega)$,
\begin{align*}
\mathcal{L}_{s,\Omega}u[\xi] \coloneqq \int_{\Omega}\int_{\Omega}K(x,y)\phi(u(x)-u(y))(\xi(x)-\xi(y))\,dx\,dy=\langle f,\xi\rangle,
\end{align*}
where $\langle\cdot,\cdot\rangle$ denotes the bilinear pairing between $(W^{s,p}_0(\Omega))^*$ and $W^{s,p}_{0}(\Omega)$.

Then our first main result is the following. Note that such an estimate is proved in \cite{Sch16} when $p\geq 2$.

\begin{theorem}\label{thm:self}
Let $0<s<1<p< 2$ and let $\Omega\subset\ern$ be an open domain.
For $f\in (W^{s,p}_0(\Omega))^*$, let $u\in W^{s,p}(\Omega)$ be a weak solution to \eqref{eq:regional} in $\Omega$ under assumptions \eqref{eq:K} and \eqref{eq:phi}. 
Then there is $\ep_0=\ep_0(n,s,p,\Lambda,\lambda)\in(0,\frac{\min\{s,1-s\}}{2(p+2)}]$ such that for any $\ep\in(0,\ep_0]$ the following holds: 
\begin{align}\label{eq:est}
f\in (W^{s-\ep(p-1),p}_0(\Omega))^*\implies u\in W^{s+\ep,p}_{\loc}(\Omega).
\end{align}
More precisely, for any $\Omega_1\Subset\Omega$ bounded, there is a constant $c=c(n,s,p,\Lambda,\lambda,\Omega_1,\Omega)$ such that
\begin{align*}
[u]_{W^{s+\ep,p}(\Omega_1)}\leq c\,\|f\|^{\frac{1}{p-1}}_{(W^{s-\ep(p-1),p}_0(\Omega))^*}+c\,[u]_{W^{s,p}(\Omega)}.
\end{align*}
Moreover, for any $\delta\in[0,\ep]$ with $\delta< n/p$,
\begin{align}\label{eq:est2}
f\in (W^{s-\ep(p-1),p}_0(\Omega))^*\implies u\in W^{s+\ep-\delta,\frac{np}{n-\delta p}}_{\loc}(\Omega)
\end{align}
and the following estimate holds:
\begin{align*}
[u]_{W^{s+\ep-\delta,\frac{np}{n-\delta p}}(\Omega_1)}\leq c\,\|f\|^{\frac{1}{p-1}}_{(W^{s-\ep(p-1),p}_0(\Omega))^*}+c\,[u]_{W^{s,p}(\Omega)}.
\end{align*}
\end{theorem}

\subsection{Fractional $p$-Laplacian}
Let $u\in W^{s,p}_{\loc}(\Omega)\cap L^{p-1}_{sp}(\ern)$ be a weak solution to \eqref{eq:fracp}, i.e., for any $\xi\in C^{\infty}_{c}(\Omega)$,
\begin{align*}
\mathcal{L}_{s}u[\xi] \coloneqq\int_{\ern}\int_{\ern}K(x,y)\phi(u(x)-u(y))(\xi(x)-\xi(y))\,dx\,dy=\langle f,\xi\rangle.
\end{align*}
Here we may confine ourselves to the case that $\Omega\subsetneq\ern$ is a bounded domain. 
Then, with the definitions of $\tail(\cdot)$ and the tail space $L^{p-1}_{sp}(\ern)$ given in Section \ref{sec:pre}, the following theorem holds.
\begin{theorem}\label{thm:self2}
Let $0<s<1<p<\infty$ and let $\Omega\subsetneq\ern$ be an open and bounded domain. 
For $f\in (W^{s,p}_0(\Omega))^*$, let $u\in W^{s,p}_{\loc}(\Omega)\cap L^{p-1}_{sp}(\ern)$ be a weak solution to \eqref{eq:fracp} in $\Omega$ under assumptions \eqref{eq:K} and \eqref{eq:phi}. 
Then there is $\ep_0=\ep_0(n,s,p,\Lambda,\lambda)\in(0,\frac{\min\{s,1-s\}}{2(p+2)}]$ such that for any $\ep\in(0,\ep_0]$, \eqref{eq:est} holds. More precisely, for any $B_{2r}\Subset\Omega$ there is a constant $c=c(n,s,p,\Lambda,\lambda)$ such that
\begin{align*}
[u]_{W^{s+\ep,p}(B_r)}\leq c\,\|f\|^{\frac{1}{p-1}}_{(W^{s-\ep(p-1),p}_0(B_{2r}))^*}+c\,r^{-\ep}[u]_{W^{s,p}(B_{2r})}+c\,r^{-s-\ep+\frac{n}{p}}\tail(u-(u)_{B_{r}};B_{3r/2}).
\end{align*}
Moreover, for any $\delta\in[0,\ep]$ with $\delta<n/p$, \eqref{eq:est2} and the following estimate hold:
\begin{equation*}
 [u]_{W^{s+\ep-\delta,\frac{np}{n-\delta p}}(B_r)}\leq c\,\|f\|^{\frac{1}{p-1}}_{(W^{s-\ep(p-1),p}(B_{2r}))^*}+c\,r^{-\ep}[u]_{W^{s,p}(B_{2r})}+c\,r^{-s-\ep+\frac{n}{p}}\tail(u-(u)_{B_{r}};B_{3r/2}).
\end{equation*}
\end{theorem}

\begin{remark}[Self-improving property for  systems] 
We emphasize that all of our arguments apply in essentially the same way to the vectorial case, i.e., the case of fractional $p$-Poisson systems.
\end{remark}

This paper is organized as follows. In Section \ref{sec:pre}, we present notation, function spaces, and necessary lemmas. In Section \ref{sec:3}, we justify the a priori regularity assumption $u\in W^{s+\ep,p}_{\loc}(\Omega)$ in \cite{Sch16} by introducing regularized problems to \eqref{eq:fracp} and \eqref{eq:regional}. Section \ref{sec:regional} is devoted to proving Theorem \ref{thm:self}. Finally, in Section \ref{sec:fracp}, we show Theorem \ref{thm:self2}.

\section{Preliminaries}\label{sec:pre}
\subsection{Notation}
We denote by $c>0$ a generic constant, whose specific value may vary from line to line. We occasionally denote its dependence using parentheses, e.g., $c=c(n,s,p)$. We denote each open ball in $\mathbb{R}^{n}$ as $B_{r}(x_{0}) \coloneqq \{x\in\mathbb{R}^{n}:|x-x_{0}|<r\}$. We omit the center of a ball if it is clear from the context.

For a function $f:U\rightarrow\er$ with $0<|U|<\infty$, the integral average over $U$ of $f$ is denoted by
\begin{equation*}
(f)_{U} \coloneqq \mean{U}f\,dx \coloneqq \frac{1}{|U|}\int_{U}f\,dx.
\end{equation*}

\subsection{Fractional Sobolev spaces and dual spaces} 
For an open set $U\subset\ern$ and $0<s<1<p<\infty$, we define the fractional Sobolev space $W^{s,p}(U)$ as the set of all measurable functions $v:U\to\er$ having finite norm
\begin{align*}
\|v\|_{W^{s,p}(U)} \coloneqq \|v\|_{L^{p}(U)}+[v]_{W^{s,p}(U)},
\end{align*}
where
\begin{align*}
[v]_{W^{s,p}(U)}^p \coloneqq \int_{U}\int_{U}\left|\dfrac{v(x)-v(y)}{|x-y|^{s}}\right|^p\dfrac{dxdy}{|x-y|^n}.
\end{align*}
Also, we define
\begin{align*}
W^{s,p}_0(U) \coloneqq \{v\in W^{s,p}(\ern):v=0 \,\,\,\,\text{a.e. in }\ern\setminus U\}
\end{align*}
and
\begin{equation*}
\|g\|_{W^{s,p}_{0}(U)} \coloneqq \|g\|_{W^{s,p}(\ern)}  \qquad \text{for} \;\; g \in W^{s,p}_{0}(U).
\end{equation*}
Note that if $U=B_r$, then 
\begin{align*}
W^{s,p}_0(B_{r})= \overline{C^{\infty}_{c}(B_{r})}^{W^{s,p}(\ern)}.
\end{align*}
With this, for $f\in (C^{\infty}_{c}(U))^*$, we denote
\begin{align*}
\|f\|_{(W^{s,p}_0(U))^*} \coloneqq \sup\left\{\langle f,g\rangle:g\in W^{s,p}_0(U),\,\,\|g\|_{W^{s,p}_0(U)}\leq 1\right\}
\end{align*}
and
\begin{align*}
(W^{s,p}_0(U))^*=\left\{f\in (C^{\infty}_{c}(U))^*:\|f\|_{(W^{s,p}_0(U))^*}<\infty\right\}.
\end{align*}

\subsection{Tail and tail spaces}
The tail space $L_{sp}^{p-1}(\ern)$ is defined as 
\begin{align*}
L_{sp}^{p-1}(\ern) \coloneqq \left\{u\in L^{1}_{\loc}(\ern):\int_{\ern}\frac{|u(x)|^{p-1}}{(1+|x|)^{n+sp}}\,dx<\infty\right\},
\end{align*}
and the tail of $u$ on $B_{r}(x_0)$ is denoted by
\begin{align*}
\tail(u;B_{r}(x_0)) \coloneqq \left(r^{sp}\int_{\ern\setminus B_{r}(x_0)}\frac{|u(x)|^{p-1}}{|x-x_0|^{n+sp}}\,dx\right)^{\frac{1}{p-1}}.
\end{align*}
Observe that $u\in L_{sp}^{p-1}(\ern)$ if and only if $\tail(u;B_r(x_0))< \infty$ for any $x_0\in\ern$ and $r>0$.

\subsection{Technical results}
We recall a few inequalities concerning fractional Sobolev spaces.
\begin{lemma}\label{lem:P}
Let $0<s<1<p<\infty$, and let $\Omega\subset\ern$ be a bounded domain. Then for any $f\in W^{s,p}(B_{r})$, we have
\begin{equation*}
\int_{\Omega}|f-(f)_{\Omega}|^p\,dx \leq \dfrac{\diam(\Omega)^{n+sp}}{|\Omega|}\int_{\Omega}\int_{\Omega}\dfrac{|f(x)-f(y)|^p}{|x-y|^{n+sp}}dxdy.
\end{equation*}
Especially, if $\Omega=B_{r}$ for some $r>0$, then for some $c=c(n)$,
\begin{align*}
\int_{B_{r}}\dfrac{|f-(f)_{B_r}|^p}{r^{sp}}\,dx\leq c\,\int_{B_{r}}\int_{B_{r}}\dfrac{|f(x)-f(y)|^p}{|x-y|^{n+sp}}dxdy.
\end{align*}
\end{lemma}
\begin{proof}
Since $t\mapsto t^{p}$ is convex, by Jensen's inequality we estimate
\begin{align*}
\int_{\Omega}|f-(f)_{\Omega}|^p\,dx&\leq \dfrac{1}{|\Omega|}\int_{\Omega}\int_{\Omega}|f(x)-f(y)|^p\,dx\,dy\\
&=\dfrac{1}{|\Omega|}\dfrac{\diam(\Omega)^{n+sp}}{\diam(\Omega)^{n+sp}}\int_{\Omega}\int_{\Omega}|f(x)-f(y)|^p\,dxdy\\
&\leq\dfrac{1}{|\Omega|}\diam(\Omega)^{n+sp}\int_{\Omega}\int_{\Omega}\dfrac{|f(x)-f(y)|^p}{|x-y|^{n+sp}}dxdy,
\end{align*}
so we get the conclusion.
\end{proof}

We also need the following lemma.
\begin{lemma}\label{lem:Pzero}
Let $0<s<1<p<\infty$, and let $\Omega_1,\Omega$ be bounded domains with $\Omega_1\Subset\Omega\subset\ern$ and  $|\Omega\setminus\Omega_1 |\geq\gamma|\Omega|$ for some $\gamma\in(0,1]$. Then there exists $c=c(n,p,\gamma)$ such that for any $f \in W^{s,p}_0(\Omega_1)$, we have
\begin{equation*}
\int_{\Omega_1}|f|^p\,dx \leq c\,\dfrac{\diam(\Omega)^{n+sp}}{|\Omega|}\int_{\Omega}\int_{\Omega}\dfrac{|f(x)-f(y)|^p}{|x-y|^{n+sp}}dxdy.
\end{equation*}
Especially, if $\Omega_1=B_{r}$ and $\Omega=B_{2r}$ for some $r>0$, then for some $c=c(n,p,\gamma)$,
\begin{equation}\label{eq:Br}
\int_{B_{r}}|f|^p\,dx \leq c\,r^{sp}\int_{B_{2r}}\int_{B_{2r}}\dfrac{|f(x)-f(y)|^p}{|x-y|^{n+sp}}dxdy.
\end{equation}
\end{lemma}
\begin{proof}
Let $B_{\rho}\subset\Omega\setminus\Omega_1$ with $\rho=\rho(\gamma)$. Then since $(f)_{B_{\rho}}=0$, for some $c=c(n,p,\gamma)$
\begin{align*}
\int_{\Omega_1}|f|^p\,dx=\int_{\Omega_1}|f-(f)_{B_{\rho}}|^p\,dx&\leq c\int_{\Omega_1}|f-(f)_{\Omega}|^p\,dx+c\,|(f)_{\Omega}-(f)_{B_{\rho}}|^p\\
&\leq c\int_{\Omega}|f-(f)_{\Omega}|^p\,dx\\
&\leq c\,\dfrac{\diam(\Omega)^{n+sp}}{|\Omega|}\int_{\Omega}\int_{\Omega}\dfrac{|f(x)-f(y)|^p}{|x-y|^{n+sp}}dxdy
\end{align*}
holds, where we also used Lemma \ref{lem:P}. Now 
\eqref{eq:Br} follows by setting $\Omega_1=B_{r}$ and $\Omega=B_{2r}$.
\end{proof}

In general, we have the following embedding.
\begin{lemma}\label{lem:SP}
Let $0<s<1\leq p<\infty$, $t\in(0,s)$ and $q\in[1,\infty)$ be such that $t-\frac{n}{q}\leq s-\frac{n}{p}$. Then for any ball $B_r\subset\ern$, the embedding $W^{s,p}(B_r)\hookrightarrow W^{t,q}(B_r)$ holds. Moreover, there exists $c= c(n,s,t,p,q)$ such that for any $f \in W^{s,p}(B_{r})$, we have
\begin{equation*}
r^{-\frac{n}{q}}\|f\|_{L^q(B_r)}+r^{-\frac{n}{q}+t}[f]_{W^{t,q}(B_r)}\leq c\,r^{-\frac{n}{p}}\|f\|_{L^p(B_r)}+c\,r^{-\frac{n}{p}+s}[f]_{W^{s,p}(B_r)}.
\end{equation*}
\end{lemma}
\begin{proof}
From \cite[Lemma 2.4]{DKLNn}, we have
\begin{equation*}
r^{-\frac{n}{q}}\|f\|_{L^q(B_r)}+r^{-\frac{n}{q}+t}[f]_{W^{t,q}(B_r)}\leq c\,r^{-n}\|f\|_{L^1(B_r)}+c\,r^{-n+\frac{s+t}{2}}[f]_{W^{\frac{s+t}{2},1}(B_r)}.
\end{equation*}
Then, using H\"{o}lder's inequality and \cite[Lemma 4.6]{coz}, we get the conclusion.
\end{proof}

The following is a potential estimate involving logarithm as in \cite[Lemma 1.2]{Sch16}. We point out a technical subtlety that in the proof of \cite[Lemma 1.2]{Sch16},  $\|I^{s}|(-\Delta)^{\frac{t}{2}}f_{j}|\|_{p}\lesssim\sum^{j+1}_{i=j-1}2^{j(t-s)}\|f_{j}\|_{p}$ (see \cite[Eq. (4.2)]{Sch16}) is used, but it is not true in general since an absolute value is applied to the term $(-\Delta)^{\frac{t}{2}}f_{j}$. Thus to use the correct inequality $\|I^{s}(-\Delta)^{\frac{t}{2}}f_{j}\|_{p}\lesssim\sum^{j+1}_{i=j-1}2^{j(t-s)}\|f_{j}\|_{p}$ and to provide a rigorous potential estimate, we employ Mikhlin multiplier theorem directly on the Littlewood--Paley pieces.
\begin{lemma}\label{lem:log}
Let $0<s<1<p<\infty$, and let $\alpha,\beta\in(0,n)$, $\gamma\in(0,1)$ be such that $\nu \coloneqq \beta-\alpha+\gamma\in(0,1)$. For $\vp\in C^{\infty}_c(\ern)$, we consider
\begin{align*}
A(\vp)\coloneqq\left(\int_{\ern}\int_{\ern}\left|\int_{\ern}k_{\alpha}(x,y,z)(-\Delta)^{\frac{\beta}{2}}\vp(z)\,dz\right|^p\dfrac{dxdy}{|x-y|^{n+\gamma p}}\right)^{\frac{1}{p}}
\end{align*}
for
\begin{align*}
k_{\alpha}(x,y,z)\coloneqq |x-z|^{\alpha-n}\log\dfrac{|x-z|}{|x-y|}-|y-z|^{\alpha-n}\log\dfrac{|y-z|}{|x-y|}.
\end{align*}
Then there exists a constant $c=c(n,p,\alpha,\beta,\gamma)$, which blows up as $\nu\nearrow 1$ or $\nu\searrow 0$, such that
\begin{align*}
A(\vp)\leq c\,[\vp]^{p}_{W^{\nu,p}(\ern)}.
\end{align*}
\end{lemma}

\begin{proof}
We use the argument in the proof of \cite[Lemma 1.2]{Sch16} via Littlewood--Paley decomposition. For a tempered distribution $g$, we define $g_j$ to be the Littlewood--Paley projections $g_{j} \coloneqq P_{j}g$, where
\begin{align*}
P_{j}g(x) \coloneqq \int_{\ern}2^{jn}p(2^{j}(x-z))g(z)\,dz.
\end{align*}
Here, $p$ is a Schwartz function with a choice satisfying
\begin{align}\label{eq:LP}
\sum_{j\in\mathbb{Z}}g_{j}=g.
\end{align}
We denote 
\begin{align*}
T\vp(x,y) \coloneqq \int_{\ern}k_{\alpha}(x,y,z)(-\Delta)^{\frac{\beta}{2}}\vp(z)\,dz.
\end{align*}
Denoting $\chi_{|y|\eqsim 2^{-k}} \coloneq \chi_{B_{2^{-k}(0)}\setminus B_{2^{-k-1}(0)}}(y)$, we employ the Littlewood--Paley decomposition \eqref{eq:LP}. Then
\begin{align*}
A(\vp)^p\leq c\sum_{k\in\mathbb{Z},j\in\mathbb{Z}}I_{j,k},
\end{align*}
where
\begin{align*}
I_{j,k} \coloneqq \int_{\ern}\int_{\ern}\chi_{|x-y|\eqsim 2^{-k}}|T\vp(x,y)|^{p-1}|T{\vp}_j(x,y)|\dfrac{dxdy}{|x-y|^{n+\gamma p}}.
\end{align*}
Let
\begin{align*}
a_k \coloneqq \left(\int_{\ern}\int_{\ern}\chi_{|x-y|\eqsim 2^{-k}}|T\vp(x,y)|^p\dfrac{dxdy}{|x-y|^{n+\gamma p}}\right)^{\frac{1}{p}}\quad\text{and}\quad b_{j} \coloneqq 2^{j(\gamma+\beta-\alpha)}\|\vp_{j}\|_{L^p(\ern)}.
\end{align*}
Note that $\left(\sum_{k\in\mathbb{Z}}a_{k}^p\right)^{1/p}\eqsim A(\vp)$. By H\"{o}lder's inequality,
\begin{align*}
I_{j,k}\leq c a_{k}^{p-1}\left(\int_{\ern}\int_{\ern}\chi_{|x-y|\eqsim 2^{-k}}|T\vp_{j}(x,y)|^p\frac{dxdy}{|x-y|^{n+\gamma p}}\right)^{\frac{1}{p}} \eqqcolon a_{k}^{p-1}\widetilde{I}_{j,k}.
\end{align*}
We now estimate $\widetilde{I}_{j,k}$. To do this, we write $h=x-y$.

\medskip

\textbf{Case 1: $|h|\leq 2^{-j}$. }By a substitution, we can write
\begin{align}\label{eq:T}
\begin{split}
T\vp_{j}(x,y)&=\int_{\ern}|z|^{\alpha-n}\log\frac{|z|}{|x-y|}\left((-\Delta)^{\frac{\beta}{2}}\vp_{j}(z+x)-(-\Delta)^{\frac{\beta}{2}}\vp_{j}(z+y)\right)\,dz\\
&=\int_{\ern}\underbrace{|z|^{\alpha-n}\log\frac{|z|}{|h|}}_{ \eqqcolon \widetilde{k}_{\alpha,|h|}(z)}\left((-\Delta)^{\frac{\beta}{2}}\vp_{j}(z+x)-(-\Delta)^{\frac{\beta}{2}}\vp_{j}(z+x-h)\right)\,dz\\
&=\left[\widetilde{k}_{\alpha,|h|}\ast(-\Delta)^{\beta/2}\vp_{j}\right](-x)-\left[\widetilde{k}_{\alpha,|h|}\ast(-\Delta)^{\beta/2}\vp_{j}\right](-x+h) \eqqcolon \widetilde{T}\vp_{j}(x,h).
\end{split}
\end{align}
We apply Fourier transform to the $x$-variable, \cite[Eq. (A.2)]{BDKL25}, and $\supp\mathcal{F}[\vp_{j}](\xi)\subset B_{2^{j}}(0)$ to obtain
\begin{align*}
\mathcal{F}[\widetilde{T}\vp_{j}](\xi)&=(1-e^{-2i\pi\xi h})\left(\mathcal{F}[\widetilde{k}_{\alpha,|h|}]\mathcal{F}[(-\Delta)^{\beta/2}\vp_{j}]\right)(\xi)\\
&=(1-e^{-2i\pi\xi h})\left(c_{1,\alpha}|\xi|^{\beta-\alpha}+c_{2,\alpha}|\xi|^{\beta-\alpha}\log|\xi|+c_{3,\alpha}|\xi|^{\beta-\alpha}\log |h|\right)\mathcal{F}[\vp_{j}](\xi)\\
&=m_{j}(\xi)\mathcal{F}[\vp_{j}](\xi)
\end{align*}
for some $c_{1,\alpha},c_{2,\alpha},c_{3,\alpha}\in\er$, where
\begin{align*}
m_{j}(\xi) \coloneqq \eta_{j}(\xi)(1-e^{-2i\pi\xi h})\left(c_{1,\alpha}|\xi|^{\beta-\alpha}+c_{2,\alpha}|\xi|^{\beta-\alpha}\log|\xi|+c_{3,\alpha}|\xi|^{\beta-\alpha}\log |h|\right)
\end{align*}
with $\eta_{j}\in C^{\infty}_{c}(\ern)$ being a standard annular cut-off function such that  $0\leq\eta_{j}\leq 1$, $\eta_{j}\equiv 1$ in $B_{2^{j}}(0)\setminus B_{2^{j-1}}(0)$, $\eta_{j}\equiv 0$ in $\ern\setminus B_{2^{j+1}}(0)\cup B_{2^{j-2}}(0)$, and $|\partial^{\gamma}\eta_{j}(\xi)|\leq c(n)|\xi|^{-|\gamma|}$ for any multi-index $\gamma$ with $|\gamma|\leq\lfloor n/2 \rfloor+1$ in $\ern$. Observing $\left|1-e^{-2i\pi\xi h}\right|\leq c\,|\xi||h|$, with $|\xi|\eqsim 2^{j}$ when $\xi\in\supp \eta_{j}$,
\begin{align}\label{eq:m}
\begin{split}
|m_{j}(\xi)|&\leq c_{\alpha}|\xi|^{\beta-\alpha}\left(1+|\log|\xi||h||\right)|\xi||h|\eta_{j}(\xi)\leq c_{\alpha}2^{j(\beta-\alpha+1)}\left(1+|\log(2^j|h|)|\right)|h|.
\end{split}
\end{align}
Similarly, to estimate $|\partial^{\gamma}m_{j}(\xi)|$, let us write multi-index $\widetilde{\gamma}$ with $1\leq |\widetilde{\gamma}|\leq |\gamma|\leq\lfloor n/2 \rfloor+1$. By $|h|\leq 2^{-j}$, since $|\xi|\eqsim 2^{j}$ when $\xi\in\supp m_{j}$, using $|h|\leq 2^{-j}\eqsim|\xi|^{-1}$ we have
\begin{align*}
\left|\partial^{\widetilde{\gamma}}_{\xi}(1-e^{-2i\pi\xi h})\right|\leq c\,|h|^{|\widetilde{\gamma}|}=c\,|h|^{|\widetilde{\gamma}|-1}|h|\leq c\,|\xi|^{1-|\widetilde{\gamma}|}|h|.
\end{align*} 
Then together with $|\partial^{\widetilde{\gamma}}\eta_{j}(\xi)|\leq c(n)|\xi|^{-|\widetilde{\gamma}|}$, for any multi-index $\gamma$ with $|\gamma|\leq\lfloor n/2 \rfloor+1$ we obtain
\begin{align}\label{eq:m2}
|\partial^{\gamma}m_{j}(\xi)|\leq c_{\alpha}2^{j(\beta-\alpha+1)}\left(1+|\log(2^j|h|)|\right)|h||\xi|^{-|\gamma|}.
\end{align}
Then by \eqref{eq:m} and \eqref{eq:m2}, Mikhlin multiplier theorem \cite[Theorem 6.2.7]{Gra} yields
\begin{align}\label{eq:1mik}
\begin{split}
\|\widetilde{T}\vp_{j}(\cdot,h)\|_{L^{p}(\ern)}&=\|\mathcal{F}^{-1}\left[m_{j}\mathcal{F}[\vp_{j}]\right]\|_{L^{p}(\ern)}\\
&\leq c(n,p,\alpha,\beta)2^{j(\beta-\alpha+1)}\left(1+|\log(2^j|h|)|\right)|h|\|\vp_{j}\|_{L^{p}(\ern)}.
\end{split}
\end{align}

Now we consider the next case.

\medskip

\textbf{Case 2: $|h|>2^{-j}$. }Same as \eqref{eq:T}, we can write
\begin{align*}
T\vp_{j}(x,y) & =\left[\widetilde{k}_{\alpha,|h|}\ast(-\Delta)^{\beta/2}\vp_{j}\right](-x)-\left[\widetilde{k}_{\alpha,|h|}\ast(-\Delta)^{\beta/2}\vp_{j}\right](-x+h) \\
& \eqqcolon \mathcal{T}\vp_{j}(-x)-\mathcal{T}\vp_{j}(-x+h).
\end{align*}
In this case, we only estimate $\mathcal{T}\vp_{j}(-x)$ and $\mathcal{T}\vp_{j}(-x+h)$ which will be considered as a translation of $\mathcal{T}\vp_{j}(-x)$ in integration over $\ern$ for $x$-variable (see \eqref{eq:2mik} below).
Now we apply Fourier transform, \cite[Eq. (A.2)]{BDKL25}, and $\supp\mathcal{F}[\vp_{j}](\xi)\subset B_{2^{j}}(0)$ to obtain
\begin{align*}
\mathcal{F}[\mathcal{T}\vp_{j}](\xi)&=\left(\mathcal{F}[\widetilde{k}_{\alpha,|h|}]\mathcal{F}[(-\Delta)^{\beta/2}\vp_{j}]\right)(\xi)\\
&=\left(c_{1,\alpha}|\xi|^{\beta-\alpha}+c_{2,\alpha}|\xi|^{\beta-\alpha}\log|\xi|+c_{3,\alpha}|\xi|^{\beta-\alpha}\log |h|\right)\mathcal{F}[\vp_{j}](\xi)=\widetilde{m}_{j}(\xi)\mathcal{F}[\vp_{j}](\xi)
\end{align*}
for some $c_{1,\alpha},c_{2,\alpha},c_{3,\alpha}\in\er$, where
\begin{align*}
\widetilde{m}_{j}(\xi) \coloneqq \eta_{j}(\xi)\left(c_{1,\alpha}|\xi|^{\beta-\alpha}+c_{2,\alpha}|\xi|^{\beta-\alpha}\log|\xi|+c_{3,\alpha}|\xi|^{\beta-\alpha}\log|h|\right)
\end{align*}
with $\eta_{j}\in C^{\infty}_{c}(\ern)$ defined in \textbf{Case 1}. Now using $|\xi|\eqsim 2^{j}$ when $\xi\in\supp \eta_{j}$, similar to above,
\begin{align}\label{eq:wtm}
\begin{split}
|\widetilde{m}_{j}(\xi)|&\leq c_{\alpha}|\xi|^{\beta-\alpha}\left(1+|\log|\xi||+|\log|h||\right)\eta_{j}(\xi)\leq c_{\alpha}2^{j(\beta-\alpha)}\left(1+|\log(2^j|h|)|\right)
\end{split}
\end{align}
holds. Similarly, for any multi-indices $\gamma$ with $|\gamma|\leq\lfloor n/2 \rfloor+1$ we can estimate
\begin{align}\label{eq:wtm2}
|\partial^{\gamma}\widetilde{m}_{j}(\xi)|\leq c_{\alpha}2^{j(\beta-\alpha)}\left(1+|\log(2^j|h|)|\right)|\xi|^{-|\gamma|}.
\end{align}
Then by \eqref{eq:wtm} and \eqref{eq:wtm2}, together with the fact that the integration over $\ern$ is invariant under translation $-x\rightarrow -x+h$, Mikhlin multiplier theorem \cite[Theorem 6.2.7]{Gra} yields
\begin{align}\label{eq:2mik}
\begin{split}
\|T\vp_{j}\|_{L^{p}(\ern)}&\leq\|\mathcal{T}\vp_{j}(\cdot)\|_{L^{p}(\ern)}+\|\mathcal{T}\vp_{j}(\cdot-h)\|_{L^{p}(\ern)}\\
&\leq\|\mathcal{T}\vp_{j}(\cdot)\|_{L^{p}(\ern)}\\
&=\|\mathcal{F}^{-1}\left[\widetilde{m}_{j}\mathcal{F}[\vp_{j}]\right]\|_{L^{p}(\ern)}\\
&\leq c(n,p,\alpha,\beta)2^{j(\beta-\alpha)}\left(1+|\log(2^j|h|)|\right)\|\vp_{j}\|_{L^{p}(\ern)}.
\end{split}
\end{align}

Now, with \textbf{Case 1} and \textbf{Case 2}, we estimate $\widetilde{I}_{j,k}$ as follows: Firstly, using \eqref{eq:T} and \eqref{eq:1mik},
\begin{align}\label{eq:Itilde}
\begin{split}
\widetilde{I}_{j,k}^p&=\int_{|h|\eqsim 2^{-k}}\int_{\ern}|\widetilde{T}\vp_{j}(x,h)|^p\frac{dxdh}{|h|^{n+\gamma p}}\\
&\leq c\,2^{j(\beta-\alpha+1)p}\|\vp_{j}\|_{L^{p}(\ern)}^p\int_{|h|\eqsim 2^{-k}}\dfrac{\left(1+|\log(2^j|h|)|\right)^p|h|^p}{|h|^{n+\gamma p}}\,dh\\
&\leq c\,\dfrac{2^{j(\beta-\alpha+1)p}\left(1+|j-k|\right)^p}{2^{(1-\gamma)pk}}\|\vp_{j}\|_{L^{p}(\ern)}^p\leq c\,b_{j}^p2^{(j-k)(1-\gamma)p}(1+|j-k|)^p
\end{split}
\end{align}
with $c=c(n,p,\alpha,\beta)$. Secondly, using the symmetry with respect to $x$ and $y$, and with \eqref{eq:2mik},
\begin{align}\label{eq:Itilde2}
\begin{split}
\widetilde{I}_{j,k}^p&=\int_{|h|\eqsim 2^{-k}}\int_{\ern}|\mathcal{T}\vp_{j}(-x)-\mathcal{T}\vp_{j}(-x+h)|^p\frac{dxdh}{|h|^{n+\gamma p}}\\
&\leq c\,2^{j(\beta-\alpha)p}\|\vp_{j}\|_{L^{p}(\ern)}^p\int_{|h|\eqsim 2^{-k}}\dfrac{\left(1+|\log(2^j|h|)|\right)^p}{|h|^{n+\gamma p}}\,dh\\
&\leq c\,2^{j(\beta-\alpha)p}\left(1+|j-k|\right)^p2^{\gamma pk}\|\vp_{j}\|_{L^{p}(\ern)}^p\leq c\,b_{j}^p2^{(k-j)\gamma p}(1+|j-k|)^p
\end{split}
\end{align}
with $c=c(n,p,\alpha,\beta)$. Define
\begin{align*}
d_{l} \coloneqq 
\begin{cases}
(1+|l|)2^{-|l|(1-\gamma)}&\,\,\,\,\text{for }\,\,l<0,\\
(1+l)2^{-l\gamma}&\,\,\,\,\text{for }\,\,l\geq 0.
\end{cases}
\end{align*}
Then we rewrite \eqref{eq:Itilde} and \eqref{eq:Itilde2} as $\widetilde{I}_{j,k}\leq c\,d_{k-j}b_{j}$. Since $0<\gamma<1$, we get $\|d\|_{l^1(\mathbb{Z})}\leq c(\gamma,p)$, so H\"{o}lder's inequality and Young's inequality for convolution yield
\begin{align*}
\sum_{k,j\in\mathbb{Z}}a_{k}^{p-1}d_{k-j}b_{j}\leq \|a\|_{l^p(\mathbb{Z})}^{p-1}\|d\ast b\|_{l^{p}(\mathbb{Z})}\leq \|d\|_{l^{1}(\mathbb{Z})}\|a\|_{l^p(\mathbb{Z})}^{p-1}\|b\|_{l^{p}(\mathbb{Z})}\leq c\,\|a\|_{l^p(\mathbb{Z})}^{p-1}\|b\|_{l^{p}(\mathbb{Z})}
\end{align*}
so that $A(\xi)\leq c\,A(\xi)^{p-1}[\xi]_{W^{\nu,p}(\ern)}$. This yields the conclusion.
\end{proof}

\section{Approximation to regularized problems}\label{sec:3}

In this section, we perform rigorous approximation schemes to drop the a priori regularity assumption $u\in W^{s+\ep,p}_{\loc}(\Omega)$ in the proof of \cite[Theorem 1.3]{Sch16}. For this, we introduce regularized problems to \eqref{eq:fracp} and \eqref{eq:regional}, whose solutions $u_{\theta}$ ($\theta\in(0,1)$)  have uniform higher differentiability bounds, i.e., they satisfy
\begin{align*}
u_{\theta}\in W^{s+\ep,p}_{\loc}(\Omega)
\end{align*}
for any $\ep\in(0,\ep_0]$ with some $\ep_0=\ep_0(n,s,p,\Lambda,\lambda)$. Define the mollifier $\mu\in C^{\infty}_{c}(\ern)$ as
\begin{align}\label{eq:sm}
\mu(x)=
\begin{cases}
c_{*}\exp\left(\frac{1}{|x|^2-1}\right)&\quad\text{if }|x|<1,\\
0&\quad\text{if }|x|\geq 1.
\end{cases}
\end{align}
Here $c_{*}$ is a normalization constant such that
\begin{align*}
\int_{\ern}\mu(x)\,dx=1.
\end{align*}
For $\theta\in(0,1)$, let us write
\begin{align*}
\mu_{\theta}(x):=\frac{1}{\theta^n}\mu\left(\frac{x}{\theta}\right).
\end{align*}
Then for the right-hand side $f$ of \eqref{eq:fracp}, we define $f_{\theta}$ as a mollification of $f$ in the distributional sense as follows: 
\begin{align}\label{eq:fepsilon}
\begin{split}
f_{\theta}(x) \coloneqq & \left(f\ast\mu_{\theta}\right)(x)=\iint_{\er^{2n}}K(y,z)\phi(u(y)-u(z))\left(\mu_{\theta}(x-y)-\mu_{\theta}(x-z)\right)\,dy\,dz.
\end{split}
\end{align}
Then since $\mu_{\theta}\in C^{\infty}_{c}(\ern)$, $f_{\theta}(x)$ is well-defined for any $x\in\ern$ and $f_{\theta}\in C^{\infty}(\ern)$. Moreover, we can show the following convergence lemma from $f_{\theta}$ to $f$.
\begin{lemma}\label{lem:conv.f}
Let $f\in (W^{t,p}_0(\Omega))^*$ for some $0<t\leq s$. Then for any $\widetilde{\Omega}_0\Subset\Omega$, $\|f_{\theta}-f\|_{(W^{t,p}_0(\widetilde{\Omega}_0))^*}\rightarrow 0$ as $\theta\rightarrow 0$.
\end{lemma}
\begin{proof}
Let us choose $\Omega_0\Subset\Omega$ such that $\widetilde{\Omega}_0\Subset\Omega_0\Subset\Omega$ and $\partial\Omega_0$ is Lipschitz. Recall that $C^{\infty}_{c}(\Omega_0)$ is dense in $(W^{t,p}_0(\Omega_0))^*$ so that for any $g\in C^{\infty}_{c}(\Omega_0)$ with $\|f-g\|_{(W^{t,p}_0(\Omega_0))^*}\leq\delta$, and for any $\theta\leq\frac{1}{2}\dist(\widetilde{\Omega}_0,\ern\setminus\Omega_0)$ with $y\in B_{\theta}(0)$, we can write
\begin{align*}
&\|f(\cdot-y)-f(\cdot)\|_{(W^{t,p}_0(\widetilde{\Omega}_0))^*}\\
&\leq \|f(\cdot-y)-g(\cdot-y)\|_{(W^{t,p}_0(\widetilde{\Omega}_0))^*}+\|g(\cdot-y)-g(\cdot)\|_{(W^{t,p}_0(\widetilde{\Omega}_0))^*}+\|g(\cdot)-f(\cdot)\|_{(W^{t,p}_0(\widetilde{\Omega}_0))^*}\\
&\leq 2\|f-g\|_{(W^{t,p}_0(\Omega_0))^*}+\|g(\cdot-y)-g(\cdot)\|_{(W^{t,p}_0(\widetilde{\Omega}_0))^*}\leq 2\delta+\|g(\cdot-y)-g(\cdot)\|_{(W^{t,p}_0(\widetilde{\Omega}_0))^*}\rightarrow 2\delta
\end{align*} 
as $y\rightarrow 0$. Since $\delta$ was arbitrary, we get $\lim_{y\rightarrow 0}\|f(\cdot-y)-f(\cdot)\|_{(W^{t,p}_0(\widetilde{\Omega}_0))^*}=0$. Using this and $\supp\mu_{\theta}\subset B_{\theta}(0)$, we have
\begin{align}\label{eq:limep}
\lim_{\theta\rightarrow 0}\int_{\ern}\mu_{\theta}(y)\|f(\cdot-y)-f(\cdot)\|_{(W^{t,p}_0(\widetilde{\Omega}_0))^*}\,dy\leq \lim_{\theta\rightarrow 0}\sup_{y\in B_{\theta}(0)}\|f(\cdot-y)-f(\cdot)\|_{(W^{t,p}_0(\widetilde{\Omega}_0))^*}=0.
\end{align}
Then for any $\psi\in W^{t,p}_0(\widetilde{\Omega}_0)$ with $\|\psi\|_{W^{t,p}(\ern)}\leq 1$, writing $\psi_{\theta}(x)=\left(\psi\ast\mu_{\theta}\right)(x)$, we observe
\begin{align*}
&\int_{\ern}f_{\theta}(\psi_{\theta}-\psi)\,dx\\
& =\int_{\er^{n}}\left(\iint_{\er^{2n}}K(y,z)\phi(u(y)-u(z))\left(\mu_{\theta}(x-y)-\mu_{\theta}(x-z)\right)\,dz\,dy\right)(\psi(x)-\psi_{\theta}(x))\,dx\\
& =\iint_{\er^{2n}}K(y,z)\phi(u(y)-u(z))\int_{\ern}\left(\mu_{\theta}(x)(\psi(x+y)-\psi_{\theta}(x+y)-(\psi(x+z)-\psi_{\theta}(x+z)))\right)\,dx\,dz\,dy\\
& =\int_{\ern}\mu_{\theta}(x)\iint_{\er^{2n}}K(y-x,z-x)\phi(u(y-x)-u(z-x))\left(\psi(y)-\psi_{\theta}(y)-(\psi(z)-\psi_{\theta}(z))\right)\,dz\,dy\,dx\\
&=\int_{\ern}\mu_{\theta}(x)\langle f(\cdot-x),\psi_{\theta}-\psi\rangle\,dx
\end{align*}
so that together with \eqref{eq:limep}, by defining $\mathcal{W}:=\{\psi\in W^{t,p}_0(\widetilde{\Omega}_0),\|\psi\|_{W^{t,p}(\ern)}\leq 1\}$, we have
\begin{align}\label{eq:ftheta-f}
\begin{split}
\|f_{\theta}-f\|_{(W^{t,p}_0(\widetilde{\Omega}_0))^*}&=\sup_{\psi\in \mathcal{W}}\left|\int_{\ern}f_{\theta}(\psi_{\theta}-\psi)\,dx-\langle f,\psi\rangle\right|\\
&=\sup_{\psi\in \mathcal{W}}\left|\int_{\ern}\mu_{\theta}(x)\langle f(\cdot-x)-f(\cdot),\psi_{\theta}-\psi\rangle\,dx\right|\\
&\leq \sup_{\psi\in \mathcal{W}}\|\psi_{\theta}-\psi\|_{W^{t,p}_0(\widetilde{\Omega}_0)}\int_{\ern}\mu_{\theta}(x)\|f(\cdot-x)-f(\cdot)\|_{(W^{t,p}_0(\widetilde{\Omega}_0))^*}\,dx\\
&\leq 2\sup_{\psi\in \mathcal{W}}\int_{\ern}\mu_{\theta}(x)\|f(\cdot-x)-f(\cdot)\|_{(W^{t,p}_0(\widetilde{\Omega}_0))^*}\,dx\rightarrow 0
\end{split}
\end{align}
as $\theta\rightarrow 0$. The conclusion follows.
\end{proof}

Define
\begin{align*}
\mathcal{A}(x,y) \coloneqq
\frac{\phi(u(x)-u(y))}{|u(x)-u(y)|^{p-2}(u(x)-u(y))}K(x,y)|x-y|^{n+sp} \in[(\Lambda\lambda)^{-1},\Lambda\lambda]
\end{align*}
so that
\begin{align*}
\mathcal{L}_{s}u(x)=\text{p.v.}\int_{\ern}\mathcal{A}(x,y)\dfrac{|u(x)-u(y)|^{p-2}(u(x)-u(y))}{|x-y|^{n+sp}}\,dy.
\end{align*}
Now for $\theta\in(0,1)$, we consider 
\begin{align*}
\mathcal{A}_{\theta}(x,y)\coloneqq \int_{\ern}\int_{\ern}\mathcal{A}(x-z,y-w)\mu_{\theta}(z)\mu_{\theta}(w)\,dz\,dw\in C^{\infty}.
\end{align*}
Note that $\mathcal{A}_{\theta}(x,y)\in[(\Lambda\lambda)^{-1},\Lambda\lambda]$ and $\mathcal{A}_{\theta}(x,y)\rightarrow \mathcal{A}(x,y)$ for a.e. $x,y\in\ern$ as $\theta\rightarrow 0$. 
Now denote by
\begin{align}\label{eq:Lepsilon}
\mathcal{L}_{s,\theta}u(x)=\text{p.v.}\int_{\ern}\mathcal{A}_{\theta}(x,y)\dfrac{|u(x)-u(y)|^{p-2}(u(x)-u(y))}{|x-y|^{n+sp}}\,dy.
\end{align}
For a weak solution $u\in W^{s,p}_{\loc}(\Omega)\cap L^{p-1}_{sp}(\ern)$ to \eqref{eq:fracp}, a domain $\widetilde{\Omega}\Subset\Omega$ with Lipschitz boundary $\partial\widetilde{\Omega}$, and $0<\theta\leq\frac{1}{2}\dist(\widetilde{\Omega},\Omega)$, let $u_{\theta}\in W^{s,p}(\widetilde{\Omega})\cap L^{p-1}_{sp}(\ern)$ be the weak solution to  
\begin{equation}\label{eq:reg1}
\left\{
\begin{aligned}
\mathcal{L}_{s,\theta}u_{\theta}&=f_{\theta}& \text{in }&\widetilde{\Omega},\\
u_{\theta}&=u& \text{in }&\ern\setminus \widetilde{\Omega}.
\end{aligned}
\right.
\end{equation}
Then we can show the following lemma.

\begin{lemma}[Higher differentiability of regularized solutions]\label{lem:hd}
There exists $\ep_0=\ep_0(s,p)\in(0,1)$ such that for any $\ep\in(0,\ep_0]$, we have $u_{\theta}\in W^{s+\ep,p}_{\loc}(\widetilde{\Omega})$ whenever $\theta>0$. 
\end{lemma}
\begin{proof}
The proof is based on \cite[Section 5]{DKLNh}, and the idea is that whenever $\theta$ remains positive, $\mathcal{A}_{\theta}\in C^{\infty}\subset C^{1}$ and $f_{\theta}\in C^{\infty}$. The main difference compared to \cite{DKLNh} is that here is a coefficient $\mathcal{A}_{\theta}\in C^{\infty}\subset C^{1}$, whereas in \cite{DKLNh} there is constant coefficient, i.e., $\mathcal{A}_{\theta}\equiv 1$ holds. Keeping this in mind, we modify the proof of \cite[Section 5]{DKLNh}. By the localization argument given in \cite{DKLNh}, we can assume $u_{\theta}\in W^{s,p}(\ern)$. Then we modify \cite[Lemmas 5.1 and 5.2]{DKLNh} as follows. Choose $B_{r}\subset \widetilde{\Omega}$ with $r\leq 1$, fix a cut-off function $\psi\in C^{\infty}_{c}(B_{3r/4})$ with $r|D\psi|+r^2|D^2\psi|\leq c$, and fix $0<|h|<r/1000$. By testing \eqref{eq:reg1} with $\delta_{-h}(\psi^{\overline{p}}\delta_{h}u_{\theta})$ for $\overline{p} \coloneqq \max\{p,2\}$, we obtain
\begin{align*}
I_{1}+I_{2} \coloneqq &\int_{\ern}\int_{\ern}\mathcal{A}_{\theta}(x+h,y+h)\delta_{h}A(\delta^{s}_{x,y}u_{\theta})\delta^{s}_{x,y}(\psi^{\overline{p}}\delta_{h}u_{\theta})\dfrac{dxdy}{|x-y|^n}\\
&+\int_{\ern}\int_{\ern}\left[\delta_{h}\mathcal{A}_{\theta}(x,y)\right]A(\delta^{s}_{x,y}u_{\theta})\delta^{s}_{x,y}(\psi^{\overline{p}}\delta_{h}u_{\theta})\dfrac{dxdy}{|x-y|^n}=\int_{B_{r}}\delta_{h}f_{\theta}\psi^{\overline{p}}\delta_{h}u_{\theta}\,dx \eqqcolon I_{3},
\end{align*}
where we have denoted $\delta_{h}g(x) \coloneqq g(x+h)-g(x)$ and $\delta^s_{x,y}g \coloneqq \frac{g(x)-g(y)}{|x-y|^s}$ for any measurable function $g$, and  $A(t) \coloneqq |t|^{p-2}t$ for any $t \in \R$. Now we divide the proof into two cases.

\medskip
\textbf{Case 1: $p\geq 2$. }Using $\mathcal{A}_{\theta}\in[(\Lambda\lambda)^{-1},\Lambda\lambda]$, $I_{1}$ and $I_{3}$ can be estimated as the same way as in the proof of \cite[Lemma 5.1]{DKLNh} so that we have
\begin{align}\label{eq:est1}
\begin{split}
&\int_{\ern}\int_{\ern}|\delta_{h}V(\delta^s_{x,y}u_{\theta})|^2\max\{\psi^p(x),\psi^p(y)\}\dfrac{dxdy}{|x-y|^n}\\
& \leq c\int_{\ern}\int_{\ern}(|\delta^s_{x,y}u_{\theta}|+(|\delta_{h}u_{\theta}(x)|+|\delta_{h}u_{\theta}(y)|)|\delta^s_{x,y}\psi|)^{p-2}(|\delta_{h}u_{\theta}(x)|+|\delta_{h}u_{\theta}(y)|)^2|\delta^s_{x,y}\psi|^2\dfrac{dxdy}{|x-y|^n}\\
& \quad+c\int_{B_{r}}|\delta_{h}u_{\theta}||\delta_{h}f_{\theta}|\,dx+c|I_2|,
\end{split}
\end{align}
where we have denoted $V(t) \coloneqq |t|^{(p-2)/2}t$ for any $t \in \R$. Now using $|\delta_{h}\mathcal{A}_{\theta}(x,y)|\leq[\mathcal{A}_{\theta}]_{C^1}|h|$ and Young's inequality, we can estimate $I_2$ as
\begin{align*}
& I_{2} \leq c|h|\int_{\ern}\int_{\ern}|\delta^{s}_{x,y}u_{\theta}|^{p-1}\left|\delta^{s}_{x,y}\psi^p\delta_{h}u_{\theta}(y)+\psi^p(x)\delta^s_{x,y}\delta_{h}u_{\theta}\right|\dfrac{dxdy}{|x-y|^n}\\
&\leq c|h|\int_{\ern}\int_{\ern}|\delta^{s}_{x,y}u_{\theta}|^{p-1}\left(|x-y|^{-s}\min\{|x-y|,1\}|\delta_{h}u_{\theta}(y)|+\psi(x)\left|\delta^s_{x,y}\delta_{h}u_{\theta}\right|\right)\dfrac{dxdy}{|x-y|^n}\\
&\leq c_{\sigma}|h|^{p'}[u_{\theta}]_{W^{s,p}(\ern)}^p+\int_{\ern}\int_{\ern}\left(|x-y|^{-sp}\min\{|x-y|,1\}^p|\delta_{h}u_{\theta}(y)|^p+\sigma\psi(x)^p\left|\delta^s_{x,y}\delta_{h}u_{\theta}\right|^p\right)\dfrac{dxdy}{|x-y|^n}\\
&\leq c_{\sigma}|h|^{p'}[u_{\theta}]_{W^{s,p}(\ern)}^p+\|\delta_{h}u_{\theta}\|^p_{L^{p}(\ern)}+\sigma\int_{\ern}\int_{\ern}\psi(x)^p\left|\delta^s_{x,y}\delta_{h}u_{\theta}\right|^p\dfrac{dxdy}{|x-y|^n}
\end{align*}
for any $\sigma\in(0,1)$. Now, in view of \cite[Lemma 2.1 (a) and (d)]{DKLNh}, we get $\left|\delta^s_{x,y}\delta_{h}u_{\theta}\right|^p\leq c|\delta_{h}V(\delta^s_{x,y}u_{\theta})|^2$. Then the last term in the right-hand side of the above estimate can be absorbed to the left-hand side of \eqref{eq:est1}. Then we arrive at
\begin{align}\label{eq:er2n}
\begin{split}
&\int_{\ern}\int_{\ern}|\delta_{h}V(\delta^s_{x,y}u_{\theta})|^2\max\{\psi^p(x),\psi^p(y)\}\dfrac{dxdy}{|x-y|^n}\\
& \leq c\int_{\ern}\int_{\ern}(|\delta^s_{x,y}u_{\theta}|+(|\delta_{h}u_{\theta}(x)|+|\delta_{h}u_{\theta}(y)|)|\delta^s_{x,y}\psi|)^{p-2}(|\delta_{h}u_{\theta}(x)|+|\delta_{h}u_{\theta}(y)|)^2|\delta^s_{x,y}\psi|^2\dfrac{dxdy}{|x-y|^n}\\
& \quad+c\int_{B_{r}}|\delta_{h}u_{\theta}||\delta_{h}f_{\theta}|\,dx+c\,|h|^{p'}[u_{\theta}]_{W^{s,p}(\ern)}^p+\|\delta_{h}u_{\theta}\|^p_{L^{p}(\ern)}
\end{split}
\end{align}
with $c$ independent of $\theta$. Applying the argument of the proof of \cite[Lemma 5.2]{DKLNh}, we obtain the estimate in \cite[Lemma 5.2]{DKLNh}, with $|h|^{p}$ in the right-hand side replaced by $|h|^{p'}$.  Observe that such an estimate implies some gain of differentiability of $u_{\theta}$ with respect to $h$, so that applying the proof of \cite[Lemma 5.3]{DKLNh} yields that there exists $\ep_0=\ep_0(s,p)\in(0,1)$ such that $u_{\theta}\in W^{s+\ep,p}(B_{r})$ for any $\ep\in(0,\ep_0]$. Now, we can conclude that $u_{\theta}\in W^{s+\ep,p}_{\loc}(\widetilde{\Omega})$ by using a covering argument (see e.g., the proof of  \eqref{eq:gamma.s.ep} below).

\medskip
\textbf{Case 2: $1<p\leq 2$. }Using $\mathcal{A}_{\theta}\in[(\Lambda\lambda)^{-1},\Lambda\lambda]$, $I_{1}$ and $I_{3}$ can be estimated as the same way as in the proof of \cite[Lemma 5.1]{DKLNh} so that we have
\begin{align}\label{eq:est1.1}
\begin{split}
&\int_{\ern}\int_{\ern}|\delta_{h}V(\delta^s_{x,y}u_{\theta})|^2\max\{\psi^2(x),\psi^2(y)\}\dfrac{dxdy}{|x-y|^n}\\
&\quad\leq c\int_{\ern}\int_{\ern}((|\delta_{h}u_{\theta}(x)|+|\delta_{h}u_{\theta}(y)|)|\delta^s_{x,y}\psi|)^{p}\dfrac{dxdy}{|x-y|^n}+c\int_{B_{r}}|\delta_{h}u_{\theta}||\delta_{h}f_{\theta}|\,dx+c|I_2|.
\end{split}
\end{align}
Now, using $|\delta_{h}\mathcal{A}_{\theta}(x,y)|\leq[\mathcal{A}_{\theta}]_{C^1}|h|$ and Young's inequality, we can estimate $I_2$ as
\begin{align*}
I_{2}&\leq c|h|\int_{\ern}\int_{\ern}|\delta^{s}_{x,y}u_{\theta}|^{p-1}\left|\delta^{s}_{x,y}\psi^p\delta_{h}u_{\theta}(y)+\psi^p(x)\delta^s_{x,y}\delta_{h}u_{\theta}\right|\dfrac{dxdy}{|x-y|^n}\\
&\leq c|h|\int_{\ern}\int_{\ern}|\delta^{s}_{x,y}u_{\theta}|^{p-1}\left(|x-y|^{-s}\min\{|x-y|,1\}|\delta_{h}u_{\theta}(y)|+\psi(x)\left|\delta^s_{x,y}\delta_{h}u_{\theta}\right|\right)\dfrac{dxdy}{|x-y|^n}\\
&\leq c|h|^{p'}[u_{\theta}]_{W^{s,p}(\ern)}^p+\int_{\ern}\int_{\ern}|x-y|^{-sp}\min\{|x-y|,1\}^p|\delta_{h}u_{\theta}(y)|^p\dfrac{dxdy}{|x-y|^n}\\
&\quad+|h|\int_{\ern}\int_{\ern}\psi(x)\underbrace{\left(|\delta^{s}_{x,y}u_{\theta}|^{\frac{p-2}{2}}\left|\delta^s_{x,y}\delta_{h}u_{\theta}\right|\right)}_{\eqsim|\delta_{h}V(\delta^s_{x,y}u_{\theta})|}|\delta^{s}_{x,y}u_{\theta}|^{\frac{p}{2}}\dfrac{dxdy}{|x-y|^n}\\
&\leq c_{\sigma}(|h|^{p'}+|h|^2)[u_{\theta}]_{W^{s,p}(\ern)}^p+\|\delta_{h}u_{\theta}\|^p_{L^{p}(\ern)}+\sigma\int_{\ern}\int_{\ern}\psi(x)^2\left|\delta_{h}V(\delta^s_{x,y}u_{\theta})\right|^2\dfrac{dxdy}{|x-y|^n}
\end{align*}
for any $\sigma\in(0,1)$. In the above display, the last term in the right-hand side can be absorbed to the left-hand side of \eqref{eq:est1.1}. Then together with $|h|\leq r\leq 1$, we arrive at
\begin{align}\label{eq:er2n2}
\begin{split}
&\int_{\ern}\int_{\ern}|\delta_{h}V(\delta^s_{x,y}u_{\theta})|^2\max\{\psi^2(x),\psi^2(y)\}\dfrac{dxdy}{|x-y|^n}\\
&\quad\leq c\int_{\ern}\int_{\ern}(|\delta_{h}u_{\theta}(x)|+|\delta_{h}u_{\theta}(y)|)|\delta^s_{x,y}\psi|)^{p}\dfrac{dxdy}{|x-y|^n}\\
&\quad\quad+c\int_{B_{r}}|\delta_{h}u_{\theta}||\delta_{h}f_{\theta}|\,dx+c\,|h|^{2}[u_{\theta}]_{W^{s,p}(\ern)}^p+\|\delta_{h}u_{\theta}\|^p_{L^{p}(\ern)}.
\end{split}
\end{align}
Applying the argument of the proof of \cite[Lemma 5.5]{DKLNh}, we obtain the estimate in \cite[Lemma 5.5]{DKLNh}. Observe that the obtained estimate implies some gain of differentiability of $u_{\theta}$ with respect to $h$, so that applying the proof of \cite[Lemma 5.6]{DKLNh} yields that $u_{\theta}\in W^{s+\ep,p}(B_{r})$. Now we deduce $u_{\theta}\in W^{s+\ep,p}_{\loc}(\widetilde{\Omega})$ via a covering argument.
\end{proof}

Note that the above argument can be applied to the vectorial case as well. Now we show the following convergence of $u_{\theta}$ to $u$ in the $W^{s,p}$-sense.

\begin{lemma}[Convergence]\label{lem:conv}
We have $\|u_{\theta}-u\|_{W^{s,p}(\widetilde{\Omega})}\rightarrow 0$ and 
\begin{align*}
\int_{\ern}\frac{|u(x)-u_{\theta}(x)|^{p-1}}{(1+|x|)^{n+sp}}\,dx\rightarrow 0\quad\text{as }\theta\rightarrow 0.
\end{align*}
\end{lemma}
\begin{proof}
Testing $u_{\theta}-u\in W^{s,p}_{0}(\widetilde{\Omega})$ to both \eqref{eq:fracp} and \eqref{eq:reg1}, we have
\begin{align*}
\langle\mathcal{L}_{s,\theta}u_{\theta}-\mathcal{L}_{s}u,u_{\theta}-u\rangle=\int_{\ern}f_{\theta}(u_{\theta}-u)\,dx-\langle f,u_{\theta}-u\rangle.
\end{align*}
Then using \cite[Eq. (3.6)]{ByuLee24} for $p<2$ and \cite[Lemma 2.1 (d)]{DKLNh} for $p\geq 2$, with any $\eta\in(0,1]$ we have
\begin{align*}
&I_1 \coloneqq \iint_{\er^{2n}}\dfrac{|u(x)-u(y)-(u_{\theta}(x)-u_{\theta}(y))|^p}{|x-y|^{n+sp}}\,dx\,dy\\
&\quad\leq c_{\eta}\,\langle\mathcal{L}_{s,\theta}u_{\theta}-\mathcal{L}_{s,\theta}u,u_{\theta}-u\rangle+\eta[u]_{W^{s,p}(\ern)}^p\chi_{\{p<2\}}\\
&\quad\leq 
c_{\eta}\left|\langle\mathcal{L}_{s,\theta}u-\mathcal{L}_{s}u,u_{\theta}-u\rangle\right|+c_{\eta}\left|\int_{\ern}f_{\theta}(u_{\theta}-u)\,dx-\langle f,u_{\theta}-u\rangle\right|+\eta[u]_{W^{s,p}(\ern)}^p\\
&\quad \eqqcolon c_{\eta}(I_2+I_3)+\eta[u]_{W^{s,p}(\ern)}^p.
\end{align*}
For $I_2$, with $\Omega_1$ satisfying $\widetilde{\Omega}\Subset\Omega_1\Subset\Omega$ and $\partial\Omega_1$ being Lipschitz (this property will be used later in \eqref{eq:limep}), using Young's inequality and $|\mathcal{A}_{\theta}(x,y)|+|\mathcal{A}(x,y)|\leq 2\Lambda\lambda$, for any $\sigma\in(0,1)$ we estimate
\begin{align*}
I_2&\leq \iint_{(\widetilde{\Omega}^c\times \widetilde{\Omega}^c)^c}|\mathcal{A}_{\theta}(x,y)-\mathcal{A}(x,y)|\dfrac{|u(x)-u(y)|^{p-1}}{|x-y|^{n+sp}}|u(x)-u(y)-(u_{\theta}(x)-u_{\theta}(y))|\,dx\,dy\\
&\leq \iint_{\Omega_1\times\Omega_1}|\mathcal{A}_{\theta}(x,y)-\mathcal{A}(x,y)|\dfrac{|u(x)-u(y)|^{p-1}}{|x-y|^{n+sp}}|u(x)-u(y)-(u_{\theta}(x)-u_{\theta}(y))|\,dx\,dy\\
&\,\,\,\,+c\int_{\ern\setminus\Omega_1}\int_{\widetilde{\Omega}}|\mathcal{A}_{\theta}(x,y)-\mathcal{A}(x,y)|\dfrac{|u(x)-(u)_{\Omega_1}|^{p-1}+|u(y)-(u)_{\Omega_1}|^{p-1}}{|x-y|^{n+sp}}|u(x)-u_{\theta}(x)|\,dx\,dy\\
&\leq c_{\sigma}\iint_{\Omega_1\times\Omega_1}|\mathcal{A}_{\theta}(x,y)-\mathcal{A}(x,y)|\dfrac{|u(x)-u(y)|^{p}}{|x-y|^{n+sp}}\,dx\,dy\\
&\quad+\sigma\iint_{\Omega_1\times\Omega_1}\dfrac{|u(x)-u(y)-(u_{\theta}(x)-u_{\theta}(y))|^{p}}{|x-y|^{n+sp}}\,dx\,dy\\
&\quad+c_{\sigma}\int_{\ern\setminus\Omega_1}\int_{\widetilde{\Omega}}|\mathcal{A}_{\theta}(x,y)-\mathcal{A}(x,y)|\dfrac{|u(x)-(u)_{\Omega_1}|^{p}}{|x-y|^{n+sp}}\,dx\,dy\\
&\quad+\sigma\int_{\widetilde{\Omega}}|u(x)-u_{\theta}(x)|^{p}\,dx\\
&\quad+c\int_{\widetilde{\Omega}}\left(\int_{\ern\setminus\Omega_1}|\mathcal{A}_{\theta}(x,y)-\mathcal{A}(x,y)|\dfrac{|u(y)-(u)_{\Omega_1}|^{p-1}}{|x_0-y|^{n+sp}}\,dy\right)|u(x)-u_{\theta}(x)|\,dx\\
& \eqqcolon I_{2,1}+I_{2,2}+I_{2,3}+I_{2,4}+I_{2,5}
\end{align*}
with $c$ independent of $\theta$. Note that for $B_{\rho}\subset\Omega_1\setminus \widetilde{\Omega}$, observing $(u-u_{\theta})_{B_{\rho}}=0$, by Lemma \ref{lem:Pzero}
\begin{align}\label{eq:utheta}
\int_{\widetilde{\Omega}}|u(x)-u_{\theta}(x)|\,dx\leq\left(\int_{\widetilde{\Omega}}|u(x)-u_{\theta}(x)|^p\,dx\right)^{\frac{1}{p}}\leq c\,[u-u_{\theta}]_{W^{s,p}(\Omega)}
\end{align}
with $c$ independent of $\theta$. Now by choosing $\sigma$ sufficiently small, $I_{2,2}$ and $I_{2,4}$ can be absorbed to $I_{1}$. Since $|\mathcal{A}_{\theta}(x,y)|+|\mathcal{A}(x,y)|\leq 2\Lambda\lambda$ and $u\in W^{s,p}_{\loc}(\Omega)$, $I_{2,1}$ and $I_{2,3}$ converge to zero as $\theta\rightarrow 0$ by the dominated convergence theorem. Now for $I_{2,5}$, by Young's inequality, for any $\sigma\in(0,1]$ we have
\begin{align*}
I_{2,5}\leq \sigma\int_{\widetilde{\Omega}}|u(x)-u_{\theta}(x)|^p\,dx+c_{\sigma}\int_{\widetilde{\Omega}}\left(\int_{\ern\setminus\Omega_1}|\mathcal{A}_{\theta}(x,y)-\mathcal{A}(x,y)|\dfrac{|u(y)-(u)_{\Omega_1}|^{p-1}}{|x_0-y|^{n+sp}}\,dy\right)^{\frac{p}{p-1}}\,dx.
\end{align*}
Then the first term of the right-hand side of the above can be absorbed to $I_{1}$ by \eqref{eq:utheta}, and the second term converges to zero as $\theta\rightarrow 0$ by $u\in L^{p-1}_{sp}(\ern)$ and the dominated convergence theorem. In summary, for any $\eta\in(0,1]$ we get
\begin{align*}
I_{1}\leq c_{\eta}o(\theta)+c_{\eta}I_{3}+\eta[u]_{W^{s,p}(\ern)}^p.
\end{align*}
Now $I_{3}$ can be estimated by the same way to \eqref{eq:ftheta-f} so that we have $I_{3}\leq o(\theta)$. In summary, we have
\begin{align}\label{eq:I1}
I_{1}\leq c_{\eta}o(\theta)+\eta[u]_{W^{s,p}(\ern)}^p.
\end{align}
On the other hand, using $\widetilde{\Omega}\Subset\Omega_1$ and $u-u_{\theta}=0$ in $\ern\setminus \widetilde{\Omega}$, we observe
\begin{align*}
I_{1}&\geq \int_{\widetilde{\Omega}}\int_{\widetilde{\Omega}}\dfrac{|u(x)-u(y)-(u_{\theta}(x)-u_{\theta}(y))|^p}{|x-y|^{n+sp}}\,dx\,dy+\int_{\ern\setminus\Omega_1}\int_{\widetilde{\Omega}}\dfrac{|u(x)-u_{\theta}(x)|^p}{|x-y|^{n+sp}}\,dx\,dy\\
&\geq c\,\|u_{\theta}-u\|_{W^{s,p}(\widetilde{\Omega})}^p+c\left(\int_{\widetilde{\Omega}}|u(x)-u_{\theta}(x)|^{p-1}\,dx\right)^{\frac{p}{p-1}}\\
&\geq c\,\|u_{\theta}-u\|_{W^{s,p}(\widetilde{\Omega})}^p+c\left(\int_{\ern}\frac{|u(x)-u_{\theta}(x)|^{p-1}}{(1+|x|)^{n+sp}}\,dx\right)^{\frac{p}{p-1}}
\end{align*}
with $c>0$ independent of $\theta$ and $\eta$. Together with \eqref{eq:I1}, Since $\eta$ and $\theta$ were arbitrary, sending $\eta\rightarrow 0$ first and then $\theta\rightarrow 0$ in \eqref{eq:I1} yield the conclusion.
\end{proof}

Similarly, we are able to obtain analogous results for the case of regional fractional $p$-Laplacian. Define
\begin{align}\label{eq:Lepsilon2}
\mathcal{L}_{s,\Omega,\theta}u(x) \coloneqq \text{p.v.}\int_{\Omega}\mathcal{A}_{\theta}(x,y)\dfrac{|u(x)-u(y)|^{p-2}(u(x)-u(y))}{|x-y|^{n+sp}}\,dy.
\end{align}
Also, for the right-hand side of \eqref{eq:regional}, we define
\begin{align}\label{eq:fepsilon2}
\begin{split}
f_{\Omega,\theta}(x) \coloneqq\left(f_{\Omega}\ast\mu_{\theta}\right)(x)=\int_{\Omega}\int_{\Omega}K(y,z)\phi(u(y)-u(z))\left(\mu_{\theta}(x-y)-\mu_{\theta}(x-z)\right)\,dy\,dz.
\end{split}
\end{align}
Then since $\mu_{\theta}\in C^{\infty}_{c}(\ern)$, $f_{\Omega,\theta}(x)$ is well-defined for any $x\in\ern$ and $f_{\Omega,\theta}\in C^{\infty}(\ern)$. Moreover, we can show the following convergence lemma from $f_{\theta}$ to $f$.
\begin{lemma}\label{lem:conv.fOmega}
Let $f_{\Omega}\in (W^{t,p}_0(\Omega))^*$ for some $0<t\leq s$. Then for any $\widetilde{\Omega}_0\Subset\Omega$, $\|f_{\Omega,\theta}-f_{\Omega}\|_{(W^{t,p}_0(\widetilde{\Omega}_0))^*}\rightarrow 0$ as $\theta\rightarrow 0$.
\end{lemma}
\begin{proof}
For any $\psi\in W^{t,p}_0(\widetilde{\Omega}_0)$ with $\|\psi\|_{W^{t,p}(\ern)}\leq 1$, and for $\theta\leq\frac{1}{2}\dist(\widetilde{\Omega}_0,\Omega^c)$, writing $\psi_{\theta}(x)=\left(\psi\ast\mu_{\theta}\right)(x)$, we have
\begin{align*}
&\int_{\ern}f_{\Omega,\theta}(\psi_{\Omega,\theta}-\psi)\,dx\\
&=\int_{\ern}\mu_{\theta}(x)\int_{\Omega}\int_{\Omega}K(y-x,z-x)\phi(u(y-x)-u(z-x))\left(\psi(y)-\psi_{\theta}(y)-(\psi(z)-\psi_{\theta}(z))\right)\,dz\,dy\,dx,
\end{align*}
where we observe that since $\supp\mu_{\theta}(x)\subset B_{\frac{1}{2}(\dist(\widetilde{\Omega}_0,\Omega^c)}(0)$ from $\theta\leq\frac{1}{2}\dist(\widetilde{\Omega}_0,\Omega^c)$, and\\ $\supp\left(\psi(y)-\psi_{\theta}(y)-(\psi(z)-\psi_{\theta}(z))\right)\subset\widetilde{\Omega}\times\widetilde{\Omega}$, so that
\begin{align*}
\supp\left(K(y-x,z-x)\phi(u(y-x)-u(z-x))\left(\psi(y)-\psi_{\theta}(y)-(\psi(z)-\psi_{\theta}(z))\right)\right)\subset\Omega\times\Omega.
\end{align*}
Then we argue as in Lemma \ref{lem:conv.f} so that we obtain the conclusion. 
\end{proof}

Now for $u\in W^{s,p}(\Omega)$ being a solution of \eqref{eq:regional}, $\widetilde{\Omega}\Subset\Omega$, $0<\theta\leq\frac{1}{2}\dist(\widetilde{\Omega},\Omega)$, let $u_{\Omega,\theta}\in W^{s,p}(\widetilde{\Omega})$ be the weak solution to  
\begin{equation*}
\left\{
\begin{aligned}
\mathcal{L}_{s,\Omega,\theta}u_{\Omega,\theta}&=f_{\Omega,\theta}& \text{in }&\widetilde{\Omega},\\
u_{\Omega,\theta}&=u&\text{in }&\Omega\setminus \widetilde{\Omega}.
\end{aligned}
\right.
\end{equation*}
Then similar to Lemmas \ref{lem:hd} and \ref{lem:conv}, one can obtain the following lemmas.

\begin{lemma}[Higher differentiability of regularized solutions]\label{lem:hd2}
There exists $\ep_0=\ep_0(s,p)\in(0,1)$ such that for any $\ep\in(0,\ep_0]$, we have $u_{\theta}\in W^{s+\ep,p}_{\loc}(\widetilde{\Omega})$ whenever $\theta>0$. 
\end{lemma}
\begin{proof}
Similar to the proof of Lemma \ref{lem:hd}, Choose $B_{r}\subset \widetilde{\Omega}$ with $r\leq 1$. Fix a cutoff function $\psi\in C^{\infty}_{c}(B_{3r/4})$ with $r|D\psi|+r^2|D^2\psi|\leq c$, and fix $0<|h|<\frac{r}{1000}$. By testing \eqref{eq:reg1} with $\delta_{-h}(\psi^{\overline{p}}\delta_{h}u_{\Omega,\theta})$ for $\overline{p} \coloneqq \max\{p,2\}$, we get
\begin{align*}
I_{1}+I_{2} \coloneqq & \int_{\widetilde{\Omega}}\int_{\widetilde{\Omega}}\mathcal{A}_{\theta}(x+h,y+h)\delta_{h}A(\delta^{s}_{x,y}u_{\Omega,\theta})\delta^{s}_{x,y}(\psi^{\overline{p}}\delta_{h}u_{\Omega,\theta})\dfrac{dxdy}{|x-y|^n}\\
&+\int_{\widetilde{\Omega}}\int_{\widetilde{\Omega}}\left[\delta_{h}\mathcal{A}_{\theta}(x,y)\right]A(\delta^{s}_{x,y}u_{\Omega,\theta})\delta^{s}_{x,y}(\psi^{\overline{p}}\delta_{h}u_{\Omega,\theta})\dfrac{dxdy}{|x-y|^n} \\
=&\int_{B_{r}}\delta_{h}f_{\Omega,\theta}\psi^{\overline{p}}\delta_{h}u_{\Omega,\theta}\,dx \eqqcolon I_{3}.
\end{align*}
In the case  $p\geq 2$, the same argument leads us to obtain \eqref{eq:er2n} for integral over $\ern$ replaced by $\widetilde{\Omega}$ with the universal constant $c$ independent of $\theta$. 
In the case $1<p\leq 2$, we can obtain \eqref{eq:er2n2} for integral over $\ern$ replaced by $\widetilde{\Omega}$. Applying the argument of the proof of \cite[Lemmas 5.2 and 5.5]{DKLNh}, we obtain the estimates given in \cite[Lemma 5.2]{DKLNh} for $p\geq 2$ and \cite[Lemma 5.5]{DKLNh} for $1<p\leq 2$, with $|h|^{p}$ in the right-hand side replaced by $|h|^{p'}$ for integral over $\ern$ replaced by $\widetilde{\Omega}$. Observe that the obtained estimate implies some gain of differentiability of $u_{\Omega,\theta}$ with respect to $h$, so that applying the proof of \cite[Lemma 5.3]{DKLNh} (when $p\geq 2$) and \cite[Lemma 5.6]{DKLNh} (when $1<p\leq 2$)  yields that there exists $\ep_0=\ep_0(s,p)\in(0,1)$ such that for any $\ep\in(0,\ep_0]$, $u_{\Omega,\theta}\in W^{s+\ep,p}(B_{r})$ holds. Now $u_{\Omega,\theta}\in W^{s+\ep,p}_{\loc}(\widetilde{\Omega})$ can be obtained by a covering argument.
\end{proof}

\begin{lemma}[Convergence]\label{lem:conv2}
We have $\|u_{\Omega,\theta}-u\|_{W^{s,p}(\Omega)}\rightarrow 0$.
\end{lemma}
\begin{proof}
We first show that $\|u_{\Omega,\theta}-u\|_{W^{s,p}(\widetilde{\Omega})}\rightarrow 0$ as $\theta\rightarrow 0$. Testing $u_{\Omega,\theta}-u\in W^{s,p}_{0}(\widetilde{\Omega})$ to both \eqref{eq:fracp} and \eqref{eq:reg1}, we have
\begin{align*}
\langle\mathcal{L}_{s,\Omega,\theta}u_{\Omega,\theta}-\mathcal{L}_{s,\Omega}u,u_{\theta}-u\rangle=\int_{\ern}f_{\Omega,\theta}(u_{\Omega,\theta}-u)\,dx-\langle f_{\Omega},u_{\Omega,\theta}-u\rangle.
\end{align*}
Then for any $\eta\in(0,1]$,
\begin{align*}
I_1 & \coloneqq \int_{\Omega}\int_{\Omega}\dfrac{|u(x)-u(y)-(u_{\theta}(x)-u_{\theta}(y))|^p}{|x-y|^{n+sp}}\,dx\,dy\\
& \leq c_{\eta}\,\langle\mathcal{L}_{s,\Omega,\theta}u_{\Omega,\theta}-\mathcal{L}_{s,\Omega,\theta}u,u_{\Omega,\theta}-u\rangle+\eta[u]^p_{W^{s,p}(\Omega)}\\
& \leq 
c_{\eta}\left|\langle\mathcal{L}_{s,\Omega,\theta}u-\mathcal{L}_{s,\Omega}u,u_{\Omega,\theta}-u\rangle\right|+c_{\eta}\left|\int_{\ern}f_{\Omega,\theta}(u_{\Omega,\theta}-u)\,dx-\langle f,u_{\theta}-u\rangle\right|+\eta[u]^p_{W^{s,p}(\Omega)}\\
& \eqqcolon c_{\eta}(I_2+I_3)+\eta[u]^p_{W^{s,p}(\Omega)}.
\end{align*}
For $I_2$, with $\Omega_1$ satisfying $\widetilde{\Omega}\Subset\Omega_1\Subset\Omega$ and $\partial\Omega_1$ being Lipschitz, using Young's inequality, similar to the proof of Lemma \ref{lem:conv}, for any $\sigma\in(0,1)$ we estimate $I_2\leq I_{2,1}+I_{2,2}+I_{2,3}+I_{2,4}+I_{2,5}$, where $I_{2,1}$, $I_{2,2}$, $I_{2,4}$ are the same as in the proof of Lemma \ref{lem:conv}, but the integral over $\ern\setminus\Omega_1$ is replaced by $\Omega\setminus\Omega_1$ for $I_{2,3}$, $I_{2,5}$. The same reasoning as in the proof of Lemma \ref{lem:conv} can be applied to $I_{2,1}$, $I_{2,2}$, $I_{2,3}$, $I_{2,4}$. For $I_{2,5}$, we have
\begin{align*}
I_{2,5}\leq \sigma\int_{\widetilde{\Omega}}|u(x)-u_{\theta}(x)|^p\,dx+c_{\sigma}\int_{\widetilde{\Omega}}\left(\int_{\Omega\setminus\Omega_1}|\mathcal{A}_{\theta}(x,y)-\mathcal{A}(x,y)|\dfrac{|u(y)-(u)_{\Omega_1}|^{p-1}}{|x_0-y|^{n+sp}}\,dy\right)^{\frac{p}{p-1}}\,dx.
\end{align*}
Then the first term of the right-hand side of the above can be absorbed to $I_{1}$, and the second term converges to zero as $\theta\rightarrow 0$ by $u\in W^{s,p}(\Omega)$ and dominated convergence theorem. In summary, for any $\eta\in(0,1]$ we get
\begin{align*}
I_{1}\leq c_{\eta}o(\theta)+c_{\eta}I_{3}+\eta[u]^p_{W^{s,p}(\Omega)}.
\end{align*}
Now $I_{3}$ can be estimate by the same way to Lemma \ref{lem:conv.fOmega} so that we have $I_{3}\leq o(\theta)$. Then we have 
$I_{1}\leq c_{\eta}o(\theta)+\eta[u]^p_{W^{s,p}(\Omega)}$ for arbitrary $\theta\in(0,1]$ and $\eta\in(0,1]$ so $\|u_{\Omega,\theta}-u\|_{W^{s,p}(\Omega)}\rightarrow 0$ holds.
\end{proof}

\section{The regional fractional $p$-Laplacian case}\label{sec:regional}
In this section, we show Theorem \ref{thm:self}. We start with an estimate for nonlinear commutators. 
Although the following lemma can be obtained similarly to \cite[Theorem 1.1]{Sch16}, we provide a full proof to explicitly verify its validity in the subquadratic case $p \in (1,2)$. Define $c_0=c_0(n,s,p)$ by the constant $c$ in \cite[Eq. (3.1)]{Sch16}.
\begin{lemma}\label{lem:comm}
Let $0<s<1<p<\infty$, $\ep\in[0,\frac{\min\{s,1-s\}}{2(p+2)}]$, and $B$ be a ball or $\ern$. Also, let $u\in W^{s+\ep,p}(B)$ and $\xi\in C^{\infty}_c(B)$. We consider the nonlinear commutator
\begin{align*}
R_{\ep}(u,\xi) \coloneqq \mathcal{L}_{s,\ep,B}u[\xi]-c_0\,\mathcal{L}_{s,B}u[(-\Delta)^{\frac{\ep p}{2}}\xi]
\end{align*}
with $\mathcal{L}_{s,B}$ defined in \eqref{eq:regional} and
\begin{align*}
\mathcal{L}_{s,\ep,B}u[\xi] \coloneqq \int_{B}\int_{B}\phi(u(x)-u(y))(\xi(x)-\xi(y))\frac{K(x,y)}{|x-y|^{\ep p}}\,dx\,dy.
\end{align*}
Then there exists $c=c(n,s,p,\Lambda,\lambda)\geq 1$ such that 
\begin{align}\label{eq:non.comm}
\begin{split}
|R_{\ep}(u,\xi)|\leq c\,\ep[u]_{W^{s+\ep,p}(B)}^{p-1}[\xi]_{W^{s+\ep,p}(\ern)} 
\end{split}
\end{align}
\end{lemma}
\begin{proof}
Similar to the proof of \cite[Theorem 1.1]{Sch16}, for any $\xi\in C^{\infty}_{c}(\ern)$, $t\in(0,1/2]$, and $\ep\in[0,\frac{\min\{s,1-s\}}{2(p+2)}]$, if we define
\begin{align*}
\delta_{x,y}u \coloneqq u(x)-u(y),
\end{align*}
and
\begin{align*}
l_{\ep}(x,y,z) \coloneqq \left(\dfrac{|x-z|^{t+\ep p-n}-|y-z|^{t+\ep p-n}}{|x-y|^{\ep p}}\right)-\left(|x-z|^{t-n}-|y-z|^{t-n}\right),
\end{align*}
we can compute
\begin{align*}
\mathcal{L}_{s,\ep,B}u[\xi]&=\int_{B}\int_{B}\left(\int_{\ern}\phi(\delta_{x,y}u)\left(|x-z|^{t-n}-|y-z|^{t-n}\right)(-\Delta)^{\frac{t+\ep p}{2}}\xi(z)\,dz\right)K(x,y)\,dx\,dy\\
&\quad+\int_{B}\int_{B}\left(\int_{\ern}\phi(\delta_{x,y}u)l_{\ep}(x,y,z)(-\Delta)^{\frac{t+\ep p}{2}}\xi(z)\,dz\right)K(x,y)\,dx\,dy.
\end{align*}

Using \cite[Eq. (3.1)]{Sch16} and $(-\Delta)^{\frac{t+\ep p}{2}}\xi=(-\Delta)^{\frac{t}{2}}(-\Delta)^{\frac{\ep p}{2}}\xi$, with $c_0=c_0(n,s,p)$ we have
\begin{align*}
R_{\ep}(u,\xi)&=\mathcal{L}_{s,\ep,B}u[\xi]-c_0\mathcal{L}_{s,B}u[(-\Delta)^{\frac{\ep p}{2}}\xi]\\
&=\int_{B}\int_{B}\left(\int_{\ern}\phi(\delta_{x,y}u)l_{\ep}(x,y,z)(-\Delta)^{\frac{t+\ep p}{2}}\xi(z)\,dz\right)K(x,y)\,dx\,dy.
\end{align*}
Since $l_{0}(x,y,z)=0$ for a.e. $x,y,z\in\ern$, one can see that
\begin{align*}
l_{\nu}(x,y,z)=\int^{\nu}_{0}\dfrac{d}{d\delta}l_{\delta}(x,y,z)\,d\delta.
\end{align*}
Let us write
\begin{align*}
k_{\delta}(x,y,z) &\coloneqq |x-y|^{\delta p}\dfrac{d}{d\delta}l_{\delta}(x,y,z)\\
&=|x-z|^{t+\delta p-n}\log\dfrac{|x-z|}{|x-y|}-|y-z|^{t+\delta p-n}\log\dfrac{|y-z|}{|x-y|}.
\end{align*}
By H\"{o}lder's inequality with $\ep\in(0,1)$, using \eqref{eq:K} and \eqref{eq:phi} we have
\begin{align*}
R_{\ep}(u,\xi)&=\int^{\ep}_{0}\int_{B}\int_{B}\dfrac{\phi(\delta_{x,y}u)}{|x-y|^{(s+\ep)(p-1)}}\left(\int_{\ern}\dfrac{k_{\delta}(x,y,z)(-\Delta)^{\frac{t+\ep p}{2}}\xi(z)\,dz}{|x-y|^{s+\ep-(\ep-\delta)p}}\right)K(x,y)|x-y|^{sp}dxdyd\delta\\
&\leq \int^{\ep}_{0}\int_{B}\int_{B}\dfrac{|\delta_{x,y}u|^{p-1}}{|x-y|^{(s+\ep)(p-1)}}\left|\int_{\ern}\dfrac{k_{\delta}(x,y,z)(-\Delta)^{\frac{t+\ep p}{2}}\xi(z)}{|x-y|^{s+\ep-(\ep-\delta)p}}\,dz\right|\dfrac{dxdyd\delta}{|x-y|^n}\\
&\leq\ep\,[u]^{p-1}_{W^{s+\ep,p}(B)}\sup_{\delta\in(0,\ep)}\left(\int_{B}\int_{B}\left|\int_{\ern}\dfrac{k_{\delta}(x,y,z)(-\Delta)^{\frac{t+\ep p}{2}}\xi(z)}{|x-y|^{s+\ep-(\ep-\delta)p}}\,dz\right|^p\dfrac{dxdy}{|x-y|^n}\right)^{\frac{1}{p}}.
\end{align*}
Choosing $t=1/2$ and recalling $\delta\leq\ep\leq \frac{\min\{s,1-s\}}{2(p+2)}$, we apply Lemma \ref{lem:log} with the choices 
\begin{align*}
\alpha \coloneqq t+\delta p\in[\tfrac{1}{2},\tfrac{1}{2}+1-s]\subset(0,1),
\end{align*}
\begin{align*}
\beta \coloneqq t+\ep p\in[\tfrac{1}{2},\tfrac{1}{2}+1-s]\subset(0,1),
\end{align*}
\begin{align*}
\gamma \coloneqq s+\ep-(\ep-\delta)p\leq s+\frac{1-s}{2}(=c(s))<1,\quad \gamma\geq s-\ep p\geq \frac{s}{2}>0,
\end{align*}
\begin{align*}
\nu=-\alpha+\beta+\gamma=s+\ep\leq s+\frac{1-s}{2}(=c(s))<1,\quad\nu\geq s>0,
\end{align*}
we get \eqref{eq:non.comm} with $c=c(n,s,p,\Lambda,\lambda)$.
\end{proof}

Now we prove a localization lemma for the subquadratic case.

\begin{lemma}[Localization lemma for $1<p< 2$]\label{lem:local}
Let $\Omega_1\Subset\Omega_2\Subset\Omega_3\Subset\Omega\subset\ern$ be open sets with $\Omega_1,\Omega_2,\Omega_3$ being bounded, so that $\dist(\Omega_1,\Omega_2^c),\dist(\Omega_2,\Omega_3^c),\dist(\Omega_3,\Omega^c)>0$. Let $0<s<1<p\leq 2$. For any $u\in W^{s,p}(\Omega)$ there exists $\widetilde{u}\in W^{s,p}(\ern)$ so that
\begin{itemize}
\item[(1)] $\widetilde{u}-u\equiv\text{const}$ in $\Omega_1$,
\item[(2)] $\supp\widetilde{u}\Subset\Omega_2$,
\item[(3)] $[\widetilde{u}]_{W^{s,p}(\ern)}\leq c\,[u]_{W^{s,p}(\Omega)}$,
\item[(4)] for any $t>0$ with $t\in (1-p(1-s),s)$,
\begin{align*}
\|\widetilde{\mathcal{L}}_{s,\Omega_3}\widetilde{u}\|_{(W^{t,p}_0(\Omega_3))^*}\leq c\,\|\mathcal{L}_{s,\Omega}u\|_{(W^{t,p}_0(\Omega))^*}+c\,[u]^{p-1}_{W^{s,p}(\Omega)},
\end{align*}
where
\begin{align*}
\widetilde{\mathcal{L}}_{s,\Omega_{3}}\widetilde{u}[\xi]&=\int_{\Omega_{3}}\int_{\Omega_{3}}\phi(\delta_{x,y}\widetilde{u})\delta_{x,y}\xi \widetilde{K}(x,y)\,dx\,dy
\end{align*}
for some measurable, nonnegative, symmetric kernel $\widetilde{K}:\ern\times\ern\rightarrow\er$ satisfying
\begin{align*}
\dfrac{1}{\Lambda\lambda^2|x-y|^{n+sp}}\leq\widetilde{K}(x,y)\leq \dfrac{\Lambda\lambda^2}{|x-y|^{n+sp}}.
\end{align*}
\end{itemize}
The constant $c$ depends only on 
\begin{align*}
n,s,t,p,\Lambda,\lambda,\dist(\Omega_1,(\Omega_2)^c),\dist(\Omega_2,(\Omega_3)^c),\min\{\dist(\Omega_3,\Omega^c),\diam(\Omega_3)\},\Omega.
\end{align*}
\end{lemma}
\begin{proof}
Fix a cut-off function $\eta\in C^{\infty}_c(\Omega_2)$ satisfying $0\leq \eta\leq 1$, $\eta \equiv 1$ in $\Omega_1$ and $|D \eta|\leq c(n)/\dist(\Omega_1,(\Omega_2)^c)$. Then we set $\widetilde{u} \coloneqq \eta^{1/(p-1)}(u-(u)_{\Omega_1})$. Clearly $\widetilde{u}$ satisfies (1) and (2). We have (3) from Lemma \ref{lem:P} and simple manipulations. We write 
\begin{align*}
\widetilde{u}(x)-\widetilde{u}(y)=\underbrace{\eta(x)^{\frac{1}{p-1}}(u(x)-u(y))}_{ \eqqcolon a(x,y)}+\underbrace{(\eta(x)^{\frac{1}{p-1}}-\eta(y)^{\frac{1}{p-1}})(u(y)-(u)_{\Omega_1})}_{ \eqqcolon b(x,y)}.
\end{align*}
We observe that for $T(a) \coloneqq |a|^{p-2}a$, $|T(a+b)-T(a)|\eqsim |b|(|a|+|b|)^{p-2}\lesssim |b|^{p-1}$ holds for $1<p<2$. Also, recalling the notation
\begin{align*}
\delta_{x,y}g=g(x)-g(y)
\end{align*}
for measurable $g$, for any $x,y\in\ern$, we see that
\begin{align*}
\widetilde{K}(x,y) \coloneqq \dfrac{|\delta_{x,y}\widetilde{u}|^{p-2}\delta_{x,y}\widetilde{u}}{\phi(\delta_{x,y}\widetilde{u})}\dfrac{\phi(\delta_{x,y}u)}{|\delta_{x,y}u|^{p-2}\delta_{x,y}u}K(x,y)
\end{align*}
satisfies $(\lambda^2\Lambda)^{-1}|x-y|^{-n-sp}\leq\widetilde{K}(x,y)\leq \lambda^2\Lambda|x-y|^{-n-sp}$ from \eqref{eq:K} and \eqref{eq:phi}. Using this and $\left(\eta^{1/(p-1)}
\right)^{p-1}=\eta$, as in the proof of \cite[Lemma 8.1]{Sch16}, for any $\xi\in C^{\infty}_{c}(\Omega_3)$ we have
\begin{equation*}
\begin{aligned}
\widetilde{\mathcal{L}}_{s,\Omega}\widetilde{u}[\xi]&=\int_{\Omega}\int_{\Omega}\phi(\delta_{x,y}\widetilde{u})\delta_{x,y}\xi \widetilde{K}(x,y)\,dx\,dy\\
&=\int_{\Omega}\int_{\Omega}|\delta_{x,y}\widetilde{u}|^{p-2}\delta_{x,y}\widetilde{u}\delta_{x,y}\xi\dfrac{\phi(\delta_{x,y}u)}{|\delta_{x,y}u|^{p-2}\delta_{x,y}u}K(x,y)\,dx\,dy\\
&=\int_{\Omega}\int_{\Omega}|\delta_{x,y}u|^{p-2}(\delta_{x,y}u)(\delta_{x,y}[\eta\xi])\dfrac{\phi(\delta_{x,y}u)}{|\delta_{x,y}u|^{p-2}\delta_{x,y}u}K(x,y)\,dx\,dy\\
&\quad-\int_{\Omega}\int_{\Omega}|\delta_{x,y}u|^{p-2}(\delta_{x,y}u)(\delta_{x,y}\eta)\xi(y)\dfrac{\phi(\delta_{x,y}u)}{|\delta_{x,y}u|^{p-2}\delta_{x,y}u}K(x,y)\,dx\,dy\\
&\quad+\int_{\Omega}\int_{\Omega}(T(a+b)-T(a))\delta_{x,y}\xi\dfrac{\phi(\delta_{x,y}u)}{|\delta_{x,y}u|^{p-2}\delta_{x,y}u}K(x,y)\,dx\,dy\\
&=\mathcal{L}_{s,\Omega}u[\eta\xi]-\int_{\Omega}\int_{\Omega}|\delta_{x,y}u|^{p-2}(\delta_{x,y}u)(\delta_{x,y}\eta)\xi(y)\dfrac{\phi(\delta_{x,y}u)}{|\delta_{x,y}u|^{p-2}\delta_{x,y}u}K(x,y)\,dx\,dy\\
&\quad+\int_{\Omega}\int_{\Omega}(T(a+b)-T(a))\delta_{x,y}\xi\dfrac{\phi(\delta_{x,y}u)}{|\delta_{x,y}u|^{p-2}\delta_{x,y}u}K(x,y)\,dx\,dy.
\end{aligned}
\end{equation*}
Note that 
\begin{align*}
\frac{1}{\lambda\Lambda|x-y|^{n+sp}}\leq\dfrac{\phi(\delta_{x,y}u)}{|\delta_{x,y}u|^{p-2}\delta_{x,y}u}K(x,y)\leq \frac{\lambda\Lambda}{|x-y|^{n+sp}}.
\end{align*}
Then by H\"{o}lder inequality, for any $t<s$,
\begin{align*}
|\widetilde{\mathcal{L}}_{s,\Omega}\widetilde{u}[\xi]|&\leq c\,\|\mathcal{L}_{s,\Omega}u\|_{(W^{t,p}_0(\Omega))^*}[\eta\xi]_{W^{t,p}(\Omega)}\\
&\quad+c\,[u]^{p-1}_{W^{s,p}(\Omega)}\left(\int_{\Omega}\int_{\Omega}\dfrac{|\delta_{x,y}\eta|^{p}|\xi(y)|^p}{|x-y|^{n+sp}}\,dx\,dy\right)^{\frac{1}{p}}\\
&\quad+c\,[\xi]_{W^{t,p}(\Omega)}\left(\int_{\Omega}\int_{\Omega_2}\dfrac{|\delta_{x,y}\eta^{\frac{1}{p-1}}|^p|u(y)-(u)_{\Omega_1}|^p}{|x-y|^{n+\frac{sp-t}{p-1}p}}\,dx\,dy\right)^{\frac{p-1}{p}}
\end{align*}
with $c=c(\Lambda,\lambda)$. For the first term in the right-hand side of the above inequality, we use $\supp\eta\subset\Omega_2$, $0\leq\eta\leq 1$, and the triangle inequality to estimate
\begin{equation*}
\begin{aligned}
[\eta\xi]^{p}_{W^{t,p}(\Omega)}&\leq\int_{\Omega_3}\int_{\Omega_3}\dfrac{|\eta\xi(x)-\eta\xi(y)|^p}{|x-y|^{n+tp}}\,dx\,dy+\int_{\Omega_2}\int_{\Omega\setminus\Omega_3}\dfrac{|\eta\xi(y)|^p}{|x-y|^{n+tp}}\,dx\,dy\\
&\leq c\int_{\Omega_3}\int_{\Omega_3}\dfrac{|\delta_{x,y}\eta|^p|\xi(x)|^p}{|x-y|^{n+tp}}\,dx\,dy\\
&\quad+c\int_{\Omega_3}\int_{\Omega_3}\dfrac{|\eta(y)||\delta_{x,y}\xi|^p}{|x-y|^{n+tp}}\,dx\,dy+c\int_{\Omega_2}\int_{\Omega\setminus\Omega_3}\dfrac{|\xi(y)|^p}{|x-y|^{n+tp}}\,dx\,dy.
\end{aligned}
\end{equation*}
Here, since $\frac{|\delta_{x,y}\eta|^p}{|x-y|^{tp}}\lesssim|x-y|^{p-tp}$ from $|D\eta|\leq c(n)/\dist(\Omega_1,(\Omega_2)^c)$, and since $\xi\in C^{\infty}_{c}(\Omega_3)$ with $\Omega_3\Subset(\Omega\cap B_{10\diam(\Omega_3)}(x_1))$ for some $x_1\in\Omega$, together with Lemma \ref{lem:Pzero}, we have
\begin{align*}
\int_{\Omega_3}\int_{\Omega_3}\dfrac{|\delta_{x,y}\eta|^p|\xi(x)|^p}{|x-y|^{n+tp}}\,dx\,dy&\lesssim \int_{\Omega_3}\int_{\Omega_3}|\xi(x)|^p|x-y|^{-n+p-tp}\,dx\,dy\\
&\lesssim \int_{\Omega_3}|\xi(x)|^p\,dx\lesssim [\xi]^p_{W^{t,p}(\Omega\cap B_{\diam(\Omega_3)}(x_1))}\leq c\,[\xi]^p_{W^{t,p}(\ern)}
\end{align*}
with $c=c(n,t,p,\dist(\Omega_1,(\Omega_2)^c),\min\{\dist(\Omega_3,\Omega^c),\diam(\Omega_3)\},\Omega)$.

Also, since $0\leq\eta\leq 1$, we estimate
\begin{align*}
\int_{\Omega_3}\int_{\Omega_3}\dfrac{|\eta(y)||\delta_{x,y}\xi|^p}{|x-y|^{n+tp}}\,dx\,dy\leq \int_{\Omega_3}\int_{\Omega_3}\dfrac{|\delta_{x,y}\xi|^p}{|x-y|^{n+tp}}\,dx\,dy\leq [\xi]^p_{W^{t,p}(\ern)}.
\end{align*}
Moreover, since $\Omega_2\Subset\Omega_3\Subset\Omega$ and $\xi\in C^{\infty}_{c}(\Omega_3)$, we again use Lemma \ref{lem:Pzero} to have
\begin{align*}
\int_{\Omega_2}\int_{\Omega\setminus\Omega_3}\dfrac{|\xi(y)|^p}{|x-y|^{n+tp}}\,dx\,dy&\lesssim \int_{\Omega_2}|\xi(y)|^p\,dy\lesssim\|\xi\|^p_{L^{p}(\Omega_3)}\lesssim [\xi]^p_{W^{t,p}(\Omega)}\leq c\,[\xi]^p_{W^{t,p}(\ern)}
\end{align*}
with $c=c(n,t,p,\min\{\dist(\Omega_3,\Omega^c),\diam(\Omega_3)\},\Omega)$. Combining the above four displays, we get
\begin{align*}
[\eta\xi]^{p}_{W^{t,p}(\Omega)}\leq c\,[\xi]^p_{W^{t,p}(\ern)}
\end{align*}
and so 
\begin{align*}
\|\mathcal{L}_{s,\Omega}u\|_{(W^{t,p}_0(\Omega))^*}[\eta\xi]_{W^{t,p}(\Omega)}\leq c\,\|\mathcal{L}_{s,\Omega}u\|_{(W^{t,p}_0(\Omega))^*}[\xi]_{W^{t,p}(\ern)}
\end{align*}
with $c=c(n,t,p,\dist(\Omega_1,(\Omega_2)^c),\min\{\dist(\Omega_3,\Omega^c),\diam(\Omega_3)\},\Omega)$. Also, by choosing a bounded subset $\Omega_4$ such that $\Omega_2\Subset\Omega_4\Subset\Omega_3$ and then proceeding as in the proof of \cite[Lemma 8.1]{Sch16}, we have
\begin{align*}
\int_{\Omega}\int_{\Omega}\dfrac{|\delta_{x,y}\eta|^{p}|\xi(y)|^p}{|x-y|^{n+sp}}\,dx\,dy\leq c\,[\xi]^p_{W^{t,p}(\ern)},
\end{align*}
where $c=c(n,s,t,p,\dist(\Omega_1,(\Omega_2)^c),\dist(\Omega_2,(\Omega_3)^c),\min\{\dist(\Omega_3,\Omega^c),\diam(\Omega_3)\},\Omega)$. Finally, using 
\begin{align*}
|\delta_{x,y}\eta^{1/(p-1)}| & \leq \|D(\eta^{1/(p-1)})\|_{L^{\infty}}|x-y| \\
& \leq c\|\eta^{1/(p-1)-1}D\eta\|_{L^{\infty}}|x-y|\leq c\|D\eta\|_{L^{\infty}}|x-y|\leq c|x-y|
\end{align*}
for $1<p<2$, together with the fact that $1-p(1-s)<t<s$, we get
\begin{align*}
&\int_{\Omega}\int_{\Omega_2}\dfrac{|\delta_{x,y}\eta^{1/(p-1)}|^p|u(y)-(u)_{\Omega_1}|^p}{|x-y|^{n+\frac{p}{p-1}(sp-t)}}\,dx\,dy\\
& \leq c\int_{\Omega_3}|u(y)-(u)_{\Omega_1}|^p\int_{\Omega_2}|x-y|^{-n-\frac{p}{p-1}(sp-t)+p}\,dx\,dy+c\int_{\Omega\setminus\Omega_{3}}|u(y)-(u)_{\Omega_1}|^p\int_{\Omega_2}\dfrac{dx\,dy}{|x-y|^{n+sp}}\\
& \leq c\int_{\Omega_1}\int_{\Omega_3}|u(y)-u(z)|^p\,dy\,dz + c\int_{\Omega_1}\int_{\Omega\setminus\Omega_3}|u(y)-u(z)|^p\,dy\,dz\leq c\,[u]_{W^{s,p}(\Omega)}
\end{align*}
with the fact that for $x,z\in\Omega_2$ and $y\in\Omega^{c}_{3}$, $|x-z|\eqsim|y-z|$ holds and $\Omega_1,\Omega_2,\Omega_3$ are bounded. Here, we have also used the fact that $\Omega_1\Subset\Omega_2\Subset\Omega_3$ and so for $y\in\Omega^{c}_{3}$, $\int_{\Omega_2}\frac{1}{|x-y|^{n+sp}}\,dx\leq c$.

Thus, for any $\xi\in C^{\infty}_{c}(\Omega_3)$ we get
\begin{align*}
|\widetilde{\mathcal{L}}_{s,\Omega}\widetilde{u}[\xi]|\leq c\,\left(\|\mathcal{L}_{s,\Omega}u\|_{(W^{t,p}_{0}(\Omega))^*}+[u]_{W^{s,p}(\Omega)}^{p-1}\right)[\xi]_{W^{t,p}(\ern)}
\end{align*}
with $c=c(n,s,t,p,\Lambda,\lambda,\dist(\Omega_1,(\Omega_2)^c),\dist(\Omega_2,(\Omega_3)^c),\min\{\dist(\Omega_3,\Omega^c),\diam(\Omega_3)\},\Omega)$. The remainder of the proof is completely similar to that of \cite[Lemma 8.1]{Sch16}, so we omit the details. 
\end{proof}

It is worth stating the localization lemma for $1<p<\infty$. Combining \cite[Lemma 8.1]{Sch16} and Lemma \ref{lem:local}, we have the following:

\begin{corollary}[Localization lemma for $1<p<\infty$]\label{cor:local}
Let $\Omega_1\Subset\Omega_2\Subset\Omega_3\Subset\Omega\subset\ern$ be open sets with $\Omega_1,\Omega_2,\Omega_3$ being bounded, and
$\dist(\Omega_1,\Omega_2^c),\dist(\Omega_2,\Omega_3^c),\dist(\Omega_3,\Omega^c)>0$. Let $0<s<1<p<\infty$. For any $u\in W^{s,p}(\Omega)$ there exists $\widetilde{u}\in W^{s,p}(\ern)$ so that
\begin{itemize}
\item[(1)] $\widetilde{u}-u\equiv\text{const}$ in $\Omega_1$,
\item[(2)] $\supp\widetilde{u}\Subset\Omega_2$,
\item[(3)] $[\widetilde{u}]_{W^{s,p}(\ern)}\leq c\,[u]_{W^{s,p}(\Omega)}$,
\item[(4)] for any $t>0$ with $t\in (1-(1-s)\min\{p,2\},s)$,
\begin{align*}
\|\widetilde{\mathcal{L}}_{s,\Omega_3}\widetilde{u}\|_{(W^{t,p}_0(\Omega_3))^*}\leq c\,\|\mathcal{L}_{s,\Omega}u\|_{(W^{t,p}_0(\Omega))^*}+c\,[u]^{p-1}_{W^{s,p}(\Omega)},
\end{align*}
where $\mathcal{L}_{s,\Omega}$ is defined in \eqref{eq:regional} satisfying \eqref{eq:K} and \eqref{eq:phi}, and
\begin{align*}
\widetilde{\mathcal{L}}_{s,\Omega_3}\widetilde{u}[\xi]&=\int_{\Omega_3}\int_{\Omega_3}\phi(\delta_{x,y}\widetilde{u})\delta_{x,y}\xi \widetilde{K}(x,y)\,dx\,dy
\end{align*}
with some measurable, nonnegative, symmetric kernel $\widetilde{K}:\ern\times\ern\rightarrow\er$ satisfying
\begin{align*}
\dfrac{1}{\Lambda\lambda^2|x-y|^{n+sp}}\leq\widetilde{K}(x,y)\leq \dfrac{\Lambda\lambda^2}{|x-y|^{n+sp}}.
\end{align*}
\end{itemize}
The constant $c$ depends only on
\begin{align*}
n,s,t,p,\Lambda,\lambda,\dist(\Omega_1,(\Omega_2)^c),\dist(\Omega_2,(\Omega_3)^c),\min\{\dist(\Omega_3,\Omega^c),\diam(\Omega_3)\},\Omega.
\end{align*} 
\end{corollary}
\begin{remark}
Note that, in this general case, $\widetilde{u}$ is given by $\widetilde{u} = \eta^{\max\{1/(p-1),1\}}(u-(u)_{\Omega_1})$ for the cut-off function $\eta$ chosen in the proof of Lemma \ref{lem:local}.
\end{remark}

The next lemma estimates the $W^{s,p}$-norm in terms of $\mathcal{L}_{s,\ep,B_{4r}}$, where $\mathcal{L}_{s,\ep,B_{4r}}$ is defined in Lemma \ref{lem:comm} for $\ep\in[0,\frac{\min\{s,1-s\}}{2(p+2)}]$. Let $B_{r},B_{2r}, B_{4r}$ be concentric balls in $\Omega\subset\ern$. 
Then we have the following lemma.

\begin{lemma}\label{lem:wsp}
Let $s,p$ be such that $0<s<1<p<\infty$, $\ep>0$ be such that $\ep\leq\frac{1-s}{2}$, and let $B_{r},B_{2r}, B_{4r}$ be concentric balls in $\Omega\subset\ern$. Also, let $v\in W^{s+\ep,p}(\Omega)$. Then for any $\delta\in(0,1]$,
\begin{align}\label{eq:est.v}
\begin{split}
[v]_{W^{s+\ep,p}(B_{r})}^p&\leq\delta^p[v]_{W^{s+\ep,p}(B_{4r})}^p+c_{\delta}\sup_{\psi}\left(\mathcal{L}_{s,\ep,B_{4r}}v[\psi]\right)^{\frac{p}{p-1}}+c_{\delta}\int_{B_{4r}}\dfrac{|u(x)-(u)_{B_{r}}|^p}{r^{(s+\ep)p}}\,dx
\end{split}
\end{align}
holds with $c=c(n,s,p)$, where the supremum is taken over all $\psi\in C^{\infty}_{c}(2B)$ with $[\psi]_{W^{s+\ep,p}(\ern)}\leq 1$.
\end{lemma}
\begin{proof}
Let $\eta\in C^{\infty}_{c}(B_{3r/2})$ be a cut-off function satisfying $0\leq\eta\leq 1$, $\eta\equiv 1$ in $B_{r}$ and $|D\eta|\leq c(n)r^{-1}$, for $\zeta\in(0,r/2]$ and the standard mollifier $\mu$ defined in \eqref{eq:sm}, define
\begin{align}\label{eq:psi.zeta}
\phi_{\zeta}(x) \coloneqq \eta(x)\underbrace{\left[(v-(v)_{B_{r}})\ast\mu_{\zeta}\right](x)}_{ \eqqcolon v_{\zeta}(x)}\in C^{\infty}_{c}(B_{2r}).
\end{align}
Then we claim
\begin{align}\label{eq:psi.p}
\int_{\ern}\int_{\ern}\dfrac{|\delta_{x,y}\phi_{\zeta}|^p}{|x-y|^{n+(s+\ep)p}}\,dx\,dy\leq c\,\int_{B_{4r}}\int_{B_{4r}}\dfrac{|\delta_{x,y}v|^p}{|x-y|^{n+(s+\ep)p}}\,dx\,dy.
\end{align}
Indeed, using $\phi_{\zeta}\in C^{\infty}_{c}(B_{2r})$, we have
\begin{align*}
& \int_{\ern}\int_{\ern}\dfrac{|\delta_{x,y}\phi_{\zeta}|^p}{|x-y|^{n+(s+\ep)p}}\,dx\,dy=\int_{B_{2r}}\int_{\ern}\dfrac{|\delta_{x,y}\phi_{\zeta}|^p}{|x-y|^{n+(s+\ep)p}}\,dx\,dy\\
&\leq\int_{B_{3r}}\int_{B_{3r}}\dfrac{|\delta_{x,y}\phi_{\zeta}|^p}{|x-y|^{n+(s+\ep)p}}\,dx\,dy+\int_{B_{2r}}\int_{\ern\setminus B_{3r}}\dfrac{|\phi_{\zeta}(x)|^p}{|x-y|^{n+(s+\ep)p}}\,dx\,dy\\
&\eqqcolon I_1+I_2.
\end{align*}
For $I_1$, using
\begin{align}\label{eq:Leib}
\delta_{x,y}\phi_{\zeta}=\delta_{x,y}v_{\zeta}\sigma_{x,y}\eta+\sigma_{x,y}v_{\zeta}\delta_{x,y}\eta\qquad(x,y\in B_{4r}),
\end{align} 
where $\sigma_{x,y}g=\frac{1}{2}(g(x)+g(y))$ for any measurable function $g$, we have
\begin{equation*}
\begin{aligned}
I_1&\leq\int_{B_{3r}}\int_{B_{3r}}|\delta_{x,y}v_{\zeta}\sigma_{x,y}\eta|^p\dfrac{dxdy}{|x-y|^{n+(s+\ep)p}}+\int_{B_{3r}}\int_{B_{3r}}|\sigma_{x,y}v_{\zeta}\delta_{x,y}\eta|^p\dfrac{dxdy}{|x-y|^{n+(s+\ep)p}}\\
&\leq\int_{B_{4r}}\int_{B_{4r}}|\delta_{x,y}v\sigma_{x,y}\eta|^p\dfrac{dxdy}{|x-y|^{n+(s+\ep)p}}+\int_{B_{4r}}\int_{B_{4r}}|\sigma_{x,y}(v-(v)_{B_r})\delta_{x,y}\eta|^p\dfrac{dxdy}{|x-y|^{n+(s+\ep)p}}.
\end{aligned}
\end{equation*}
Here, for the last line we have argued in the same way to obtain \cite[Eq. (5) in Page 715]{Evans} using $\zeta\in(0,r/2]$. Then we further estimate
\begin{align*}
\int_{B_{4r}}\int_{B_{4r}}|\delta_{x,y}v\sigma_{x,y}\eta|^p\dfrac{dxdy}{|x-y|^{n+(s+\ep)p}}&\leq \int_{B_{4r}}\int_{B_{4r}}|\delta_{x,y}v|^p\dfrac{dxdy}{|x-y|^{n+(s+\ep)p}}
\end{align*}
and, by using Lemma \ref{lem:P} with $|\delta_{x,y}\eta|\leq c\,r^{-1}|x-y|$ from $|D\eta|\leq cr^{-1}$,
\begin{align*}
&\int_{B_{4r}}\int_{B_{4r}}|\sigma_{x,y}(v-(v)_{B_r})\delta_{x,y}\eta|^p\dfrac{dxdy}{|x-y|^{n+(s+\ep)p}}\\
&\quad\leq c\,r^{-p}\int_{B_{4r}}\int_{B_{4r}}|v-(v)_{B_r}|^p|x-y|^{-n+(1-s-\ep)p}\,dx\,dy\\
&\quad\leq c\,r^{-(s+\ep)p}\int_{B_{4r}}|v-(v)_{B_r}|^p\,dx\leq c\,\int_{B_{4r}}\int_{B_{4r}}|\delta_{x,y}v|^p\dfrac{dxdy}{|x-y|^{n+(s+\ep)p}}
\end{align*}
so that
\begin{equation*}
\begin{aligned}
I_1\leq c\int_{B_{4r}}\int_{B_{4r}}|\delta_{x,y}v|^p\dfrac{dxdy}{|x-y|^{n+(s+\ep)p}}.
\end{aligned}
\end{equation*}
Also, since $(1-s-\ep)p\geq (1-s)p/2>0$ from $\ep\leq (1-s)/2$, $c$ in the above only depends on $n,s$, and $p$. For $I_2$, arguing in the same way to obtain \cite[Eq. (5) in Page 715]{Evans} with $0\leq\eta\leq 1$, Lemma \ref{lem:P} and $\zeta\leq r/2$, we have
\begin{align*}
I_2&\leq c\,\int_{B_{3r}}\dfrac{|v(x)-(v)_{B_r}|^p}{r^{(s+\ep)p}}\,dx\leq c\,\int_{B_{4r}}\int_{B_{4r}}|\delta_{x,y}v|^p\dfrac{dxdy}{|x-y|^{n+(s+\ep)p}}.
\end{align*}
Thus \eqref{eq:psi.p} is proved.

To show \eqref{eq:est.v}, together with $\eta\equiv 1$ on $B_{r}$, \eqref{eq:Leib} for $\zeta=0$, \eqref{eq:phi}, and \eqref{eq:K} we have
\begin{equation*}
\begin{aligned}
[v]^p_{W^{s+\ep,p}(B_{r})}&=\int_{B_{r}}\int_{B_{r}}\dfrac{|\delta_{x,y}v|^{p-2}(\delta_{x,y}v)^2\sigma_{x,y}\eta}{|x-y|^{(s+\ep)p}}\dfrac{dx\,dy}{|x-y|^n}\\
&\leq c\int_{B_{r}}\int_{B_{r}}\phi(\delta_{x,y}v)\delta_{x,y}v\sigma_{x,y}\eta K(x,y)|x-y|^{-\ep p}\,dx\,dy\\
&\leq c\int_{B_{4r}}\int_{B_{4r}}\phi(\delta_{x,y}v)\delta_{x,y}v\sigma_{x,y}\eta K(x,y)|x-y|^{-\ep p}\,dx\,dy\\
&\leq c\int_{B_{4r}}\int_{B_{4r}}\phi(\delta_{x,y}v)\delta_{x,y}\psi K(x,y)|x-y|^{-\ep p}\,dx\,dy\\
&\quad+c\int_{B_{4r}}\int_{B_{4r}}\dfrac{|\delta_{x,y}v|^{p-1}|\sigma_{x,y}(v-(v)_{B_r})\delta_{x,y}\eta|}{|x-y|^{(s+\ep) p}}\frac{dx\,dy}{|x-y|^{n}} \eqqcolon I+II
\end{aligned}
\end{equation*}
for some $c=c(\Lambda,\lambda)$. With H\"{o}lder inequality, \eqref{eq:K} and \eqref{eq:phi}, we have
\begin{align}\label{eq:I}
\begin{split}
I&\leq c\int_{B_{4r}}\int_{B_{4r}}\phi(\delta_{x,y}v)\delta_{x,y}\phi_{\zeta}K(x,y)|x-y|^{-\ep p}\,dx\,dy\\
&\quad+c\int_{B_{4r}}\int_{B_{4r}}\phi(\delta_{x,y}v)|\delta_{x,y}\psi-\delta_{x,y}\phi_{\zeta}|K(x,y)|x-y|^{-\ep p}\,dx\,dy\\
&\leq c\,\mathcal{L}_{s,\ep,B_{4r}}v[\phi_{\zeta}] \\
& \quad +c\,\underbrace{\left(\int_{B_{4r}}\int_{B_{4r}}\dfrac{|\delta_{x,y}v|^p}{|x-y|^{(s+\ep)p}}\dfrac{dx\,dy}{|x-y|^n}\right)^{\frac{p-1}{p}}}_{=:\mathcal{V}}\left(\int_{B_{4r}}\int_{B_{4r}}\dfrac{|\delta_{x,y}\psi-\delta_{x,y}\phi_{\zeta}|^p}{|x-y|^{(s+\ep)p}}\dfrac{dx\,dy}{|x-y|^n}\right)^{\frac{1}{p}}.
\end{split}
\end{align}
Further, using \eqref{eq:psi.p} and Young's inequality, for any $\delta\in(0,1]$ we estimate
\begin{align}\label{eq:mathcalL}
\begin{split}
\mathcal{L}_{s,\ep,B_{4r}}v[\phi_{\zeta}]&\leq c\,[v]^p_{W^{s+\ep,p}(B_{4r})}\sup_{[\psi]_{W^{s,p}(\ern)}\leq 1}\mathcal{L}_{s,\ep,B_{4r}}v[\psi]\\
&\leq \delta^{p}[v]^p_{W^{s+\ep,p}(B_{4r})}+c_{\delta}\sup_{[\psi]_{W^{s,p}(\ern)}\leq 1}\left(\mathcal{L}_{s,\ep,B_{4r}}v[\psi]\right)^{\frac{p}{p-1}}.
\end{split}
\end{align}
Also, by the basic property of mollifier, we see that
\begin{align}\label{eq:int}
\int_{B_{4r}}\int_{B_{4r}}\dfrac{|\delta_{x,y}\psi-\delta_{x,y}\phi_{\zeta}|^p}{|x-y|^{(s+\ep)p}}\dfrac{dx\,dy}{|x-y|^n}\rightarrow 0
\end{align}
as $\zeta\rightarrow 0$. Then applying \eqref{eq:mathcalL} first to \eqref{eq:I}, and then considering \eqref{eq:int} by sending $\zeta\rightarrow 0$ since $I$, $\mathcal{V}$, and the right-hand side of \eqref{eq:mathcalL} are independent of $\zeta$, \eqref{eq:I} becomes
\begin{align*}
I\leq \delta^{p}[v]^p_{W^{s+\ep,p}(B_{4r})}+c_{\delta}\sup_{[\psi]_{W^{s,p}(\ern)}\leq 1}\left(\mathcal{L}_{s,\ep,B_{4r}}v[\psi]\right)^{\frac{p}{p-1}}.
\end{align*}

On the other hand, using H\"{o}lder's inequality and $(1-s-\ep)p\geq (1-s)p/2>0$ from $\ep\leq(1-s)/2$, with $\|D\eta\|_{L^{\infty}}\leq c\,r^{-1}$ we obtain
\begin{align*}
II&\leq c\,\|D\eta\|_{L^{\infty}}\int_{B_{4r}}\int_{B_{4r}}\dfrac{|\delta_{x,y}v|^{p-1}|v(x)-(v)_{B_{r}}|}{|x-y|^{n+(s+\ep)p-1}}\,dx\,dy\\
&\leq c\,\|D\eta\|_{L^{\infty}}\left(\int_{B_{4r}}\int_{B_{4r}}\dfrac{|\delta_{x,y}v|^p}{|x-y|^{(s+\ep)p+n}}\,dx\,dy\right)^{\frac{p-1}{p}}\left(\int_{B_{4r}}\int_{B_{4r}}\dfrac{|v(x)-(v)_{B_r}|^p}{|x-y|^{n+(s+\ep)p-p}}\,dx\,dy\right)^{\frac{1}{p}}\\
&\leq c\,r^{-1}\left(\int_{B_{4r}}\int_{B_{4r}}\dfrac{|\delta_{x,y}v|^p}{|x-y|^{(s+\ep)p+n}}\,dx\,dy\right)^{\frac{p-1}{p}}r^{1-s-\ep}\left(\int_{B_{4r}}|v(x)-(v)_{B_r}|^p\,dx\right)^{\frac{1}{p}} \\
&= c\,r^{-s-\ep}\left(\int_{B_{4r}}\int_{B_{4r}}\dfrac{|\delta_{x,y}v|^p}{|x-y|^{(s+\ep)p}}\frac{dx\,dy}{|x-y|^n}\right)^{\frac{p-1}{p}}\left(\int_{B_{4r}}|v(x)-(v)_{B_{r}}|^p\,dx\right)^{\frac{1}{p}}
\end{align*}
with $c=c(n,s,p)$. Then Young's inequality with $\delta\in(0,1]$ gives the conclusion.
\end{proof}

Also, we need the following.
\begin{lemma}\label{lem:disj}
Let $\ep\in(0,\frac{1-s}{2}]$, $\eta\in C^{\infty}_{c}(B_{6r})$ with $0\leq\eta\leq 1$, $\eta\equiv 1$ in $B_{5r}$, $|D\eta|\leq c(n)/r$, and $\psi\in C^{\infty}_{c}(B_{4r})$. Then for any $\nu \in (0,1/2]$,
\begin{align*}
\left[(1-\eta)(-\Delta)^{\nu}\psi\right]_{\mathrm{Lip}(B_{6r})}\leq c\,r^{-\frac{n}{p}-1+s+\ep-2\nu}[\psi]_{W^{s+\ep,p}(\ern)}
\end{align*}
holds with $c=c(n,s,p)$ independent of $\ep$ and $\nu$.
\end{lemma}
\begin{proof}
The proof is based on that of \cite[Lemma A.1]{BRS}. We observe
\begin{align*}
D((1-\eta)(-\Delta)^{\nu}\psi)=(-D\eta)\chi_{\ern\setminus B_{5r}}(-\Delta)^{\nu}\psi+(1-\eta)\chi_{\ern\setminus B_{5r}}(-\Delta)^{\nu}D\psi.
\end{align*}
Note that $|D\eta|\leq c/r$ and $|1-\eta|\leq 1$. Also, with $k_{\nu}(z)=|z|^{-n-2\nu}\chi_{\ern\setminus B_{r}}(z)$ and $\supp\psi\subset B_{4r}$,
\begin{align*}
(-\Delta)^{\nu}\psi(x)=c_{\nu}\,(k_{\nu}\ast \psi)(x)
\end{align*}
with 
\begin{align*}
c_{\nu}=\frac{\nu 2^{2\nu}\Gamma\left(\frac{n+2\nu}{2}\right)}{\pi^{\frac{n}{2}}\Gamma(1-\nu)}\leq c(n)\quad\text{for any }\nu\in(0,\tfrac{1}{2}].
\end{align*}
for $x\in\ern\setminus B_{5r}$. Thus
\begin{align*}
|\chi_{\ern\setminus B_{5r}}(-\Delta)^{\nu}\psi(x)|\leq c\,\|k_{\nu}\|_{L^{\infty}(\ern)}\int_{B_{4r}}|\psi|\,dx\leq c\,r^{-n-2\nu}\int_{B_{4r}}|\psi|\,dx
\end{align*}
with $c=c(n)$. Moreover, since $\supp\psi\subset B_{4r}$, the divergence theorem implies
\begin{align*}
|\chi_{\ern\setminus B_{5r}}(-\Delta)^{\nu}D\psi|&=c\,|(k_{\nu}\ast D\psi)(x)|\\
&=c\left|\int_{\ern\setminus B_r}|z|^{-n-2\nu}D\psi(x-z)\,dz\right|\\
&\leq c\,\int_{\ern\setminus B_r}|z|^{-n-2\nu-1}|\psi(x-z)|\,dz\leq c\,r^{-n-2\nu-1}\int_{B_{4r}}|\psi|\,dx
\end{align*}
with $c=c(n)$. Then using H\"{o}lder's inequality we get
\begin{align*}
&r^{1-s-\ep+2\nu}\left[(1-\eta)(-\Delta)^{\nu}\psi\right]_{\mathrm{Lip}(B_{6r})}\\
&\quad\leq c\,r^{-s-\ep}\mean{B_{4r}}|\psi|\,dx\leq c\,\left(\mean{B_{4r}}\dfrac{|\psi|^p}{r^{(s+\ep)p}}\,dx\right)^{\frac{1}{p}}\leq c\,r^{-\frac{n}{p}}\left(\int_{B_{4r}}\frac{|\psi|^p}{r^{(s+\ep)p}}\,dx\right)^{\frac{1}{p}}
\end{align*}
with $c=c(n,s)$. Here, using $\psi\in C^{\infty}_c(B_{4r})$, by Lemma \ref{lem:Pzero}, we have
\begin{align*}
\int_{B_{4r}}\dfrac{|\psi|^p}{r^{(s+\ep)p}}\,dx&\leq c\,[\psi]^p_{W^{s+\ep,p}(B_{5r})}
\end{align*}
with $c=c(n,s,p)$. This implies the conclusion.
\end{proof}

\begin{proof}[Proof of Theorem \ref{thm:self}]
The proof is essentially similar to that of \cite[Lemma 5.1]{Sch16}. However, for the sake of completeness, we give a full proof here.

Let $0<s<1<p\leq 2$. For $f\in (W^{s,p}_0(\Omega))^*$, let $u\in W^{s,p}(\Omega)$ be a weak solution to \eqref{eq:regional} under assumptions \eqref{eq:K} and \eqref{eq:phi}. Fix bounded set $\Omega_1\Subset\Omega$ (so $\dist(\Omega_1,\Omega)>0$), and choose bounded $\Omega_2,\Omega_3\subset\Omega$ such that
\begin{enumerate}
\item[(a)] $\Omega_3\Subset \Omega\cap B_{10\diam(\Omega_1)}(x_1)$ for some $x_1\in\Omega_1$ with $\dist(\Omega_1,\Omega_3^c)\leq\dist(\Omega_3,\Omega^c)$, and \item[(b)] $\dist(\Omega_1,\Omega_2^c),\dist(\Omega_2,\Omega_3^c),\dist(\Omega_3,\Omega^c)>0$ with $\dist(\Omega_1,\Omega_2^c)\approx\dist(\Omega_2,\Omega_3^c)$.
\end{enumerate}
Then $\Omega_1\Subset\Omega_2\Subset\Omega_3\Subset\Omega$ and one can see that
\begin{align*}
c(\Omega_1,\Omega)=\min\{\dist(\Omega_1,\Omega^c),\diam(\Omega_1)\}\lesssim \min\{\dist(\Omega_3,\Omega^c),\diam(\Omega_3)\}\leq 10\diam(\Omega_1)=c(\Omega_1).
\end{align*}
Also, we consider $\Omega_4\Subset\Omega$ such that $\dist(\Omega_3,\Omega_4^c)\eqsim\dist(\Omega_2,\Omega_3^c)$ and $\partial\Omega_4$ is Lipschitz. For $0<\rho\leq\frac{1}{2}\dist(\Omega_4,\Omega)$, let $u_{\rho}\in W^{s,p}(\Omega_4)$ be the weak solution to  
\begin{align}\label{eq:reg2}
\begin{split}
\mathcal{L}_{s,\Omega,\rho}u_{\rho}&=f_{\Omega,\rho}\quad\text{in }\Omega_4,\\
u_{\rho}&=u\quad\quad\text{in }\ern\setminus \Omega_4,
\end{split}
\end{align}
where $\mathcal{L}_{s,\Omega,\rho}$ and $f_{\Omega,\rho}$ are defined in \eqref{eq:Lepsilon2} and \eqref{eq:fepsilon2}, respectively.

Now by Lemma \ref{lem:local}, there exists $w_{\rho}\in W^{s,p}(\ern)$ so that
\begin{itemize}
\item[(1)] $w_{\rho}\equiv u_{\rho}$ in $\Omega_1$,
\item[(2)] $\supp w_{\rho}\Subset\Omega_2$,
\item[(3)] $[w_{\rho}]_{W^{s,p}(\ern)}\leq c\,[u_{\rho}]_{W^{s,p}(\Omega)}$,
\item[(4)] for any $t>0$ with $t\in (1-(1-s)\min\{p,2\},s)$,
\begin{align*}
\|\widetilde{\mathcal{L}}_{s,\Omega_3,\rho}w_{\rho}\|_{(W^{t,p}_0(\Omega_3))^*}\leq c\,\|\mathcal{L}_{s,\Omega,\rho}u_{\rho}\|_{(W^{t,p}_0(\Omega))^*}+c\,[u_{\rho}]^{p-1}_{W^{s,p}(\Omega)}
\end{align*}
with $c$ only depend on $n,s,t,p,\Lambda,\lambda,\Omega_1,\Omega$, where $\mathcal{L}_{s,\Omega,\rho}$ is defined in \eqref{eq:Lepsilon2}, and
\begin{align*}
\widetilde{\mathcal{L}}_{s,\Omega_3,\rho}w_{\rho}[\xi]&=\int_{\Omega_3}\int_{\Omega_3}\phi(\delta_{x,y}w_{\rho})\delta_{x,y}\xi \widetilde{K}(x,y)\,dx\,dy
\end{align*}
with some measurable, nonnegative, symmetric kernel $\widetilde{K}:\ern\times\ern\rightarrow\er$ satisfying
\begin{align*}
\dfrac{1}{\Lambda\lambda^2|x-y|^{n+sp}}\leq\widetilde{K}(x,y)\leq \dfrac{\Lambda\lambda^2}{|x-y|^{n+sp}}.
\end{align*}
\end{itemize}

There are finitely many balls $(B_{k})^{N}_{k=1}\subset\Omega_3$ with $\Omega_2\subset\bigcup^{N}_{k=1}B_{k}$. We denote by $10B_k$ the concentric balls to $B_{k}$ with radius $5\diam(B_{k})$; we can assume $\bigcup^{N}_{k=1}10B_k\subset\Omega_3$. 
For $\ep\in[0,\frac{\min\{s,1-s\}}{2(p+2)}]$,
\begin{align*}
\Gamma_s \coloneqq [w_{\rho}]^p_{W^{s,p}(\Omega_3)}\qquad\text{and}\qquad\Gamma_{s+\ep} \coloneqq[w_{\rho}]^p_{W^{s+\ep,p}(\Omega_3)}.
\end{align*}
Note that $\Gamma_{s+\ep}<\infty$ by Lemma \ref{lem:hd2}.
Then, since $\supp w_{\rho}\subset\Omega_2$ and $\Omega_2\subset\bigcup^{N}_{k=1}B_{k}$, we have
\begin{align*}
\Gamma_{s+\ep}\leq \sum^{N}_{k=1}\int_{2B_{k}}\int_{2B_{k}}\dfrac{|\delta_{x,y}w_{\rho}|^p}{|x-y|^{n+(s+\ep)p}}\,dx\,dy+\sum^{N}_{k=1}\int_{\Omega_3\setminus 2B_{k}}\int_{B_{k}}\dfrac{|w_{\rho}(x)-w_{\rho}(y)|^p}{|x-y|^{n+(s+\ep)p}}\,dx\,dy.
\end{align*}
Here, since the supports of the following two integrals are disjoint, one can see that
\begin{align*}
\int_{\Omega_3\setminus 2B_{k}}\int_{B_{k}}\dfrac{|w_{\rho}(x)-w_{\rho}(y)|^p}{|x-y|^{(s+\ep)p}}\dfrac{dx\,dy}{|x-y|^n}\leq c(n,s,p)\,(\diam B_k)^{-\ep p}\Gamma_s
\end{align*}
and so with $c=c(n,s,p,\Omega_1,\Omega)$,
\begin{align}\label{eq:gamma.s.ep}
\Gamma_{s+\ep}\leq \sum^{N}_{k=1}[w_{\rho}]^p_{W^{s+\ep,p}(2B_k)}+c\,\Gamma_s.
\end{align}

Recalling $\bigcup^{N}_{k=1}8B_k\subset\Omega_3$, we use Lemma \ref{lem:wsp}, and Lemma \ref{lem:P} to get
\begin{align*}
\begin{split}
\Gamma_{s+\ep}\leq c^*\delta^{p}\Gamma_{s+\ep}+c_{\delta}\Gamma_{s}+c_{\delta}^*\sum^{N}_{k=1}\sup_{\xi}(\widetilde{\mathcal{L}}_{s,\ep,8B_{k},\rho}w_{\rho}[\xi])^{\frac{p}{p-1}},
\end{split}
\end{align*}
where the supremum is taken over all $\xi\in C^{\infty}_{c}(4B_{k})$ and $[\xi]_{W^{s+\ep,p}(\ern)}\leq 1$, $c^*=c(n,s,p,\Lambda,\lambda)$, $c^*_{\delta}=c(n,s,p,\Lambda,\lambda,\delta)$, and $c_{\delta}=c(n,s,p,\Lambda,\lambda,\Omega_1,\Omega,\delta)$. Now choosing $\delta=\delta(n,s,p,\Lambda,\lambda)$ small,
\begin{align*}
\begin{split}
\Gamma_{s+\ep}\leq c\Gamma_{s}+c^*\sum^{N}_{k=1}\sup_{\xi}(\widetilde{\mathcal{L}}_{s,\ep,8B_{k},\rho}w_{\rho}[\xi])^{\frac{p}{p-1}}.
\end{split}
\end{align*}

From Lemma \ref{lem:comm}, for any $\ep\in(0,\frac{\min\{s,1-s\}}{2(p+2)}]$, it follows that
\begin{align*}
\Gamma_{s+\ep}\leq c^*\sum^{N}_{k=1}\left(\sup_{\xi}\widetilde{\mathcal{L}}_{s,8B_{k},\rho}w_{\rho}[(-\Delta)^{\frac{\ep p}{2}}\xi]\right)^{\frac{p}{p-1}}+c^*\,\ep\Gamma_{s+\ep}+c\,\Gamma_{s}.
\end{align*}
with $c^*=c(n,s,p,\Lambda,\lambda)$ and $c=c(n,s,p,\Lambda,\lambda,\Omega_1,\Omega)$. By choosing $\ep_0\leq \frac{\min\{s,1-s\}}{2(p+2)}$ small enough depending on $n,s,p,\Lambda,\lambda$, for any $\ep\in(0,\ep_0]$ we arrive at
\begin{align}\label{eq:pfthm1}
\Gamma_{s+\ep}\leq c^*\sum^{N}_{k=1}\left(\sup_{\xi}\widetilde{\mathcal{L}}_{s,8B_{k},\rho}w_{\rho}[(-\Delta)^{\frac{\ep p}{2}}\xi]\right)^{\frac{p}{p-1}}+c\,\Gamma_{s}
\end{align}
with $c^*=c(n,s,p,\Lambda,\lambda)$ and $c=c(n,s,p,\Lambda,\lambda,\Omega_1,\Omega)$.

We now replace the term $(-\Delta)^{\frac{\ep p}{2}}\xi$ with an appropriate test function. To do this, we choose a cut-off function with $\eta_{6B_k}\in C^{\infty}_{c}(6B_{k})$, $\eta_{6 B_{k}}\equiv 1$ in $5B_{k}$, and $|D\eta_{6 B_{k}}|\leq c(n)/(\diam(6 B_{k}))$. Set
\begin{align*}
\psi \coloneqq \eta_{6B_k}(-\Delta)^{\frac{\ep p}{2}}\xi \;\; \Longleftrightarrow \;\;
(-\Delta)^{\frac{\ep p}{2}}\xi = \psi+(1-\eta_{6B_{k}})(-\Delta)^{\frac{\ep p}{2}}\xi.
\end{align*}
Then $\psi\in C^{\infty}_{c}(6B_k)$. Moreover, using properties of Triebel--Lizorkin spaces and Littlewood--Paley decomposition (see \cite[Eqs. (4.2)--(4.4)]{Sch16}), we have
\begin{equation*}
\begin{aligned}
[\psi]_{W^{s-\ep(p-1),p}(\ern)} 
&\leq c_{k}\,\|\psi\|_{F^{s-\ep(p-1)}_{p,p}(\ern)}\leq c_{k}\,\|(-\Delta)^{\frac{\ep p}{2}}\xi\|_{F^{s-\ep(p-1)}_{p,p}(\ern)}\leq c_{k}\,\|\xi\|_{F^{s+\ep}_{p,p}(\ern)} \\ 
& \leq c_{k}[\xi]_{W^{s+\ep,p}(\ern)}\leq c_k.
\end{aligned}
\end{equation*}
Also, since $\ep\leq\frac{1-s}{2(p+2)}\leq\frac{1-s}{2}\leq\frac{1}{2}$, using Lemma \ref{lem:disj}, we get
\begin{align*}
\left[(1-\eta_{6B_{k}})(-\Delta)^{\frac{\ep p}{2}}\xi\right]_{\text{Lip}}\leq c_{k}[\xi]_{W^{s+\ep,p}(\ern)}
\end{align*}
with $c_k=c(n,s,p,\diam(6B_{k}))=c(n,s,p,\Omega_1,\Omega)$ (since $\diam(6B_{k})$ may depend on $\Omega_1$ and $\Omega$), which in turn implies
\begin{align*}
\left|\widetilde{\mathcal{L}}_{s,8B_k,\rho}w_{\rho}[(-\Delta)^{\frac{\ep p}{2}}\xi-\psi]\right| = \left|\widetilde{\mathcal{L}}_{s,8B_k,\rho}w_{\rho}[(1-\eta_{6B_k})(-\Delta)^{\frac{\ep p}{2}}\xi]\right| \leq c\,[w_{\rho}]^{p-1}_{W^{s,p}(\Omega)}.
\end{align*}

Hence, the estimate in \eqref{eq:pfthm1} becomes 
\begin{align*}
\begin{split}
\Gamma_{s+\ep}\leq c\,\Gamma_{s}+c\sum^{N}_{k=1}\left(\sup\left\{|\widetilde{\mathcal{L}}_{s,8B_{k},\rho}w_{\rho}[\psi]|:\psi\in C^{\infty}_{c}(6B_{k}),[\psi]_{W^{s-\ep(p-1),p}(\ern)}\leq 1\right\}\right)^{\frac{p}{p-1}}
\end{split}
\end{align*}
with $c=c(n,s,p,\Lambda,\lambda,\Omega_1,\Omega)$. Finally, we restrict the domain of integration in the operator $\widetilde{\mathcal{L}}_{s,8B_{k},\rho}$ from $8B_k$ to $\Omega_3$. Indeed, since $\supp \psi\subset 6B_{k}$, we have
\begin{align*}
\left|\widetilde{\mathcal{L}}_{s,8B_{k},\rho}w_{\rho}[\psi]-\widetilde{\mathcal{L}}_{s,\Omega_3,\rho}w_{\rho}[\psi]\right|&\leq c\,\int_{\Omega_3\setminus 8B_k}\int_{7B_k}\dfrac{|\delta_{x,y}w_{\rho}|^{p-1}|\delta_{x,y}\psi|}{|x-y|^{n+sp}}\,dx\,dy\\
&\leq c_{k}[w_{\rho}]^{p-1}_{W^{s,p}(\Omega_3)}[\psi]_{W^{s-\ep(p-1),p}(\ern)}.
\end{align*}
This implies that
\begin{align*}
\Gamma_{s+\ep}\leq c\,\Gamma_{s}+c\,\left(\sup\left\{\left|\widetilde{\mathcal{L}}_{s,\Omega_3,\rho}w_{\rho}[\psi]\right|:\psi\in C^{\infty}_{c}(\Omega),[\psi]_{W^{s-\ep(p-1),p}(\ern)}\leq 1\right\}\right)^{\frac{p}{p-1}}
\end{align*}
with $c=c(n,s,p,\Lambda,\lambda,\Omega_1,\Omega)$. Therefore, with the previous choice of $\ep\leq\frac{1-s}{2(p+2)}$ small enough depending on $n,s,p,\Lambda,\lambda$, together with \eqref{eq:regional} and (1)--(4) above, we obtain
\begin{align*}
\begin{split}
& [u_{\rho}]_{W^{s+\ep,p}(\Omega_1)}\leq [w_{\rho}]^p_{W^{s+\ep,p}(\Omega_3)}\\
&\leq c\,[w_{\rho}]^p_{W^{s,p}(\Omega_3)}+c\,\left(\sup\left\{\left|\widetilde{\mathcal{L}}_{s,\Omega_3,\rho}w_{\rho}[\psi]\right|:\psi\in C^{\infty}_{c}(\Omega),[\psi]_{W^{s-\ep(p-1),p}(\ern)}\leq 1\right\}\right)^{\frac{p}{p-1}}\\
&\leq c\,[u_{\rho}]_{W^{s,p}(\Omega)}^p+c\,\|\widetilde{\mathcal{L}}_{s,\Omega_3,\rho}w_{\rho}\|_{(W^{s-\ep(p-1),p}_0(\Omega_3))^*}^{\frac{p}{p-1}}\\
&\leq c\,[u_{\rho}]_{W^{s,p}(\Omega)}^p+c\,\|\mathcal{L}_{s,\Omega,\rho}u_{\rho}\|_{(W^{s-\ep(p-1),p}_0(\Omega))^*}^{\frac{p}{p-1}}\\
&= c\,[u_{\rho}]_{W^{s,p}(\Omega)}^p+c\,\|f_{\Omega,\rho}\|_{(W^{s-\ep(p-1),p}_0(\Omega))^*}^{\frac{p}{p-1}}
\end{split}
\end{align*}
with $c=c(n,s,p,\Lambda,\lambda,\Omega_1,\Omega)$, where for the last line \eqref{eq:reg2} is used. Now we let $\rho\rightarrow 0$ to get the desired estimate. To justify this, recalling $f\equiv f_{\Omega}\in (W^{s-\ep(p-1),p}_{0}(\Omega))^*$ from the assumption of Theorem \ref{thm:self}, observe that by Lemma \ref{lem:conv2}, terms in the right-hand side of the above estimate is a convergent sequence with respect to $\rho$, which implies that $$[u_{\rho}]_{W^{s+\ep,p}(\Omega_1)}\leq C$$
with $C$ independent of $\rho$. Now since $W^{s+\ep,p}(\Omega_1)$ is a reflexive Banach space, we see that there is a subsequence $\{\rho_l\}_{l=1}^{\infty}$ such that $\rho_{l}\rightarrow 0$ as $l\rightarrow\infty$ and $\lim_{l\rightarrow\infty}[u_{\rho_{l}}-u_0]_{W^{s+\ep,p}(\Omega_1)}=0$ for some $u_0\in W^{s+\ep,p}(\Omega_1)$. Here,  Lemma \ref{lem:conv} implies that $u_{\rho}\rightarrow u$ a.e. in $\Omega_1$ and in $W^{s,p}(\Omega)$. Thus, by the uniqueness of the limit, we see that $u_0\equiv u$ a.e. in $\Omega_1$ and so $u\in W^{s+\ep,p}(\Omega_1)$. Therefore,
\begin{align*}
[u]_{W^{s+\ep,p}(\Omega_1)}&=\liminf_{\rho\rightarrow 0}[u_{\rho}]_{W^{s+\ep,p}(\Omega_1)}\\
&\leq \liminf_{\rho\rightarrow 0}\left(c\,[u_{\rho}]_{W^{s,p}(\Omega)}^p+c\,\|\mathcal{L}_{s,\Omega,\rho}u_{\rho}\|_{(W^{s-\ep(p-1),p}_0(\Omega))^*}^{\frac{p}{p-1}}\right)\\
&\leq c\,[u]_{W^{s,p}(\Omega)}^p+c\,\|\mathcal{L}_{s,\Omega}u\|_{(W^{s-\ep(p-1),p}_0(\Omega))^*}^{\frac{p}{p-1}}.
\end{align*}
This is the first estimate together with \eqref{eq:est} in Theorem \ref{thm:self}. The second one together with \eqref{eq:est2} follows from Lemma \ref{lem:SP} with Lemma \ref{lem:P}.
\end{proof}

\section{The fractional $p$-Laplacian case}\label{sec:fracp}
In this section, we prove Theorem \ref{thm:self2}. First, we note the following scaling properties of \eqref{eq:fracp} with \eqref{eq:K} and \eqref{eq:phi}.
\begin{lemma}\label{lem:scaling}
Let $f\in (W^{s,p}_0(B_{R}(x_0)))^*$ be given and let $u \in W^{s,p}_{\loc}(B_{R}(x_0))\cap L^{p-1}_{sp}(\ern)$ be a weak solution to
\begin{align*}
\mathcal{L}_{s}u=f\quad\text{in }B_{R}(x_0),
\end{align*}
where $\mathcal{L}_s$ is defined in \eqref{eq:fracp} with \eqref{eq:K} and \eqref{eq:phi}. Then
\begin{align*}
u_{R}(x) \coloneqq \dfrac{u(Rx+x_0)}{R^s}\in W^{s,p}_{\loc}(B_1)\cap L^{p-1}_{sp}(\ern)
\end{align*}
is a weak solution to
\begin{align*}
\text{p.v.}\int_{\ern}K_{R}(x,y)\phi_{R}(u_{R}(x)-u_{R}(y))\,dy=f_{R}\quad\text{in }B_{1}
\end{align*}
where 
\begin{itemize}
\item $f_{R}(x) \coloneqq R^{s}f(Rx+x_0)\in (W^{s,p}_0(B_{1}))^*$,
\item $K_{R}:\ern\times\ern\rightarrow\er$ is a measurable, symmetric kernel defined by \[K_{R}(x,y) \coloneqq R^{n+sp}K(Rx,Ry)\] and satisfies
\begin{align*}
\dfrac{1}{\Lambda\,|x-y|^{n+sp}}\leq K_{R}(x,y)\leq\dfrac{\Lambda}{|x-y|^{n+sp}},
\end{align*}
\item $\phi_{R}:\er\rightarrow\er$ defined by  $\phi_{R}(t) \coloneqq R^{-s(p-1)}\phi(R^st)$ satisfies
\begin{align*}
|\phi_{R}(t)|\leq\lambda|t|^{p-1},\quad\phi_{R}(t)t\geq \lambda^{-1}|t|^p\quad(t\in\er).
\end{align*}
\end{itemize}
\end{lemma}
The following is a localization lemma for $1<p<\infty$, which is an analog of Corollary  \ref{cor:local}.
\begin{lemma}[Localization lemma]\label{lem:local2}
Let $0<s<1<p<\infty$ and $B_{5r}\subset\Omega$. For any $u\in W^{s,p}_{\loc}(\Omega) \cap L^{p-1}_{sp}(\ern)$ there exists $\widetilde{u}\in W^{s,p}(\ern)$ so that
\begin{itemize}
\item[(1)] $\widetilde{u}-u\equiv\text{const}$ in $B_{r}$,
\item[(2)] $\supp\widetilde{u}\Subset B_{2r}$,
\item[(3)] $[\widetilde{u}]_{W^{s,p}(\ern)}\leq c\,[u]_{W^{s,p}(B_{3r})}$,
\item[(4)] for any $t>0$ with
\begin{align}\label{eq:t}
t\in (1-(1-s)\min\{p,2\},s),
\end{align}
we have 
\begin{align*}
& \|\widetilde{\mathcal{L}}_{s}\widetilde{u}\|_{(W^{t,p}_0(B_{3r}))^*} \\
& \leq c\,\left(\|\mathcal{L}_{s}u\|_{(W^{t,p}_0(B_{4r}))^*}+r^{-s+t}[u]^{p-1}_{W^{s,p}(B_{4r})}+r^{-sp+t+\frac{n}{p'}}\tail(u-(u)_{B_{r}};B_{3r})^{p-1}\right),
\end{align*}
where
\begin{align*}
\widetilde{\mathcal{L}}_{s}\widetilde{u}[\xi]&=\int_{\ern}\int_{\ern}\phi(\delta_{x,y}\widetilde{u})\delta_{x,y}\xi \widetilde{K}(x,y)\,dx\,dy
\end{align*}
with some measurable, nonnegative, symmetric kernel $\widetilde{K}:\ern\times\ern\rightarrow\er$ satisfying
\begin{align*}
\dfrac{1}{\Lambda\lambda^2|x-y|^{n+sp}}\leq\widetilde{K}(x,y)\leq \dfrac{\Lambda\lambda^2}{|x-y|^{n+sp}}.
\end{align*}
\end{itemize}
The constant $c$ depends only on $n,s,p,\lambda,\Lambda,t$. 
\end{lemma}
\begin{proof}
By the scaling argument given in Lemma \ref{lem:scaling}, we may assume that $r=1$. Let $B_{5}\subset\Omega$, and let $\eta\equiv\eta_{B_{1}}\in C^{\infty}_c(B_{2})$ be a cut-off function satisfying $\eta_{B_{1}}\equiv 1$ in $B_1$. We set
\begin{align*}
\widetilde{u} \coloneqq \eta^{\overline{p}}(u-(u)_{B_1})\quad\text{with}\quad\overline{p}=\max\left\{\tfrac{1}{p-1},1\right\}.
\end{align*}
Clearly $\widetilde{u}$ satisfies (1) and (2). We have (3) with Lemma \ref{lem:P}. We write 
\begin{align*}
\widetilde{u}(x)-\widetilde{u}(y)=\underbrace{\eta(x)^{\overline{p}}(u(x)-u(y))}_{ \eqqcolon a(x,y)}+\underbrace{(\eta(x)^{\overline{p}}-\eta(y)^{\overline{p}})(u(y)-(u)_{B_1})}_{ \eqqcolon b(x,y)}.
\end{align*}
For any $x,y\in\ern$, we see from \eqref{eq:K} and \eqref{eq:phi} that
\begin{align*}
\widetilde{K}(x,y) \coloneqq \dfrac{|\delta_{x,y}\widetilde{u}|^{p-2}\delta_{x,y}\widetilde{u}}{\phi(\delta_{x,y}\widetilde{u})}\dfrac{\phi(\delta_{x,y}u)}{|\delta_{x,y}u|^{p-2}\delta_{x,y}u}K(x,y)
\end{align*}
satisfies $(\lambda^2\Lambda)^{-1}|x-y|^{-n-sp}\leq\widetilde{K}(x,y)\leq \lambda^2\Lambda|x-y|^{-n-sp}$. Using this, similarly to the proof of Lemma \ref{lem:local}, for any $\xi\in C^{\infty}_{c}(B_3)$ we have
\begin{align*}
\widetilde{\mathcal{L}}_{s}\widetilde{u}[\xi]&=\int_{\ern}\int_{\ern}\phi(\delta_{x,y}\widetilde{u})\delta_{x,y}\xi \widetilde{K}(x,y)\,dx\,dy\\
&=\int_{\ern}\int_{\ern}|\delta_{x,y}\widetilde{u}|^{p-2}\delta_{x,y}\widetilde{u}\delta_{x,y}\xi\dfrac{\phi(\delta_{x,y}u)}{|\delta_{x,y}u|^{p-2}\delta_{x,y}u}K(x,y)\,dx\,dy\\
&=\int_{\ern}\int_{\ern}|\delta_{x,y}u|^{p-2}(\delta_{x,y}u)(\delta_{x,y}[\eta^{\overline{p}(p-1)}\xi])\dfrac{\phi(\delta_{x,y}u)}{|\delta_{x,y}u|^{p-2}\delta_{x,y}u}K(x,y)\,dx\,dy\\
&\quad-\int_{\ern}\int_{\ern}|\delta_{x,y}u|^{p-2}(\delta_{x,y}u)(\delta_{x,y}\eta^{\overline{p}(p-1)})\xi(y)\dfrac{\phi(\delta_{x,y}u)}{|\delta_{x,y}u|^{p-2}\delta_{x,y}u}K(x,y)\,dx\,dy\\
&\quad+\int_{\ern}\int_{\ern}(T(a+b)-T(a))\delta_{x,y}\xi\dfrac{\phi(\delta_{x,y}u)}{|\delta_{x,y}u|^{p-2}\delta_{x,y}u}K(x,y)\,dx\,dy\\
&=\mathcal{L}_{s}u[\eta^{\overline{p}(p-1)}\xi]-\int_{\ern}\int_{\ern}|\delta_{x,y}u|^{p-2}(\delta_{x,y}u)(\delta_{x,y}\eta^{\overline{p}(p-1)})\xi(y)\dfrac{\phi(\delta_{x,y}u)}{|\delta_{x,y}u|^{p-2}\delta_{x,y}u}K(x,y)\,dx\,dy\\
&\quad+\int_{\ern}\int_{\ern}(T(a+b)-T(a))\delta_{x,y}\xi\dfrac{\phi(\delta_{x,y}u)}{|\delta_{x,y}u|^{p-2}\delta_{x,y}u}K(x,y)\,dx\,dy \eqqcolon J_1+J_2+J_3,
\end{align*}
where $T(a) \coloneqq |a|^{p-2}a$ for $a\in\er$. Note that 
\begin{align*}
(\lambda\Lambda)^{-1}|x-y|^{-n-sp}\leq\dfrac{\phi(\delta_{x,y}u)}{|\delta_{x,y}u|^{p-2}\delta_{x,y}u}K(x,y)\leq(\lambda\Lambda)|x-y|^{-n-sp}.
\end{align*}

For $J_1$, using the fact that $\supp\eta\subset B_2$, $0\leq\eta\leq 1$, and the triangle inequality, we estimate
\begin{align*}
[\eta^{\overline{p}(p-1)}\xi]^{p}_{W^{t,p}(\Omega)}&\leq\int_{B_3}\int_{B_3}\dfrac{|\eta^{\overline{p}(p-1)}\xi(x)-\eta^{\overline{p}(p-1)}\xi(y)|^p}{|x-y|^{n+tp}}\,dx\,dy+\int_{B_2}\int_{\Omega\setminus B_3}\dfrac{|\eta^{\overline{p}(p-1)}\xi(y)|^p}{|x-y|^{n+tp}}\,dx\,dy\\
&\leq c\int_{B_3}\int_{B_3}\dfrac{|\eta^{\overline{p}(p-1)}(x)-\eta^{\overline{p}(p-1)}(y)|^p|\xi(x)|^p}{|x-y|^{n+tp}}\,dx\,dy\\
&\quad+c\int_{B_3}\int_{B_3}\dfrac{|\eta^{\overline{p}(p-1)}(y)||\xi(x)-\xi(y)|^p}{|x-y|^{n+tp}}\,dx\,dy+c\int_{B_2}\int_{\Omega\setminus B_3}\dfrac{|\xi(y)|^p}{|x-y|^{n+tp}}\,dx\,dy.
\end{align*}
Here, since 
\begin{align}\label{eq:eta}
\frac{|\delta_{x,y}\eta^{\overline{p}(p-1)}|^p}{|x-y|^{tp}}\lesssim |\eta|^{(\overline{p}(p-1)-1)p}\|D\eta\|_{L^{\infty}}^p|x-y|^{p-tp}\lesssim |x-y|^{p-tp}
\end{align}
from $0\leq\eta\leq 1$, $|D\eta|\leq c(n)$, and $\overline{p}(p-1)-1\geq 0$, and since $\xi\in C^{\infty}_{c}(B_3)$, we have
\begin{align*}
\int_{B_3}\int_{B_3}\dfrac{|\delta_{x,y}\eta^{\overline{p}(p-1)}|^p|\xi(x)|^p}{|x-y|^{n+tp}}\,dx\,dy&\lesssim \int_{B_3}\int_{B_3}|\xi(x)|^p|x-y|^{-n+p-tp}\,dx\,dy\\
&\lesssim \int_{B_3}|\xi(x)|^p\,dx\lesssim [\xi]^p_{W^{t,p}(\Omega)}\leq c\,[\xi]^p_{W^{t,p}(\ern)}
\end{align*}
with $c=c(n,s,p)$, where for the last inequality we have used Lemma \ref{lem:Pzero} with $B_3\Subset\Omega$. Also, since $0\leq\eta\leq 1$, we estimate
\begin{align*}
\int_{B_3}\int_{B_3}\dfrac{|\eta(y)^{\overline{p}(p-1)}||\delta_{x,y}\xi|^p}{|x-y|^{n+tp}}\,dx\,dy\leq \int_{B_3}\int_{B_3}\dfrac{|\delta_{x,y}\xi|^p}{|x-y|^{n+tp}}\,dx\,dy\leq [\xi]^p_{W^{t,p}(B_3)}.
\end{align*}
Moreover, since $B_2\Subset B_3$ and $\xi\in C^{\infty}_{c}(B_3)$, using Lemma \ref{lem:Pzero} we have
\begin{align*}
\int_{B_2}\int_{\Omega\setminus B_3}\dfrac{|\xi(y)|^p}{|x-y|^{n+tp}}\,dx\,dy&\lesssim \int_{B_2}|\xi(y)|^p\,dy\lesssim\|\xi\|^p_{L^{p}(B_3)}\lesssim [\xi]^p_{W^{t,p}(\Omega)}\leq c\,[\xi]^p_{W^{t,p}(\ern)}
\end{align*}
with $c=c(n,t,p)$. Then we get
\begin{align*}
[\eta^{\overline{p}(p-1)}\xi]^{p}_{W^{t,p}(\Omega)}\leq c\,[\xi]^p_{W^{t,p}(\ern)}
\end{align*}
so that 
\begin{align*}
|J_1|=\left|\mathcal{L}_{s}u[\eta^{\overline{p}(p-1)}\xi]\right|\leq \|\mathcal{L}_{s}u\|_{(W^{t,p}_0(\Omega))^*}[\eta^{\overline{p}(p-1)}\xi]_{W^{t,p}(\Omega)}\leq c\,\|\mathcal{L}_{s}u\|_{(W^{t,p}_0(\Omega))^*}[\xi]_{W^{t,p}(\ern)}.
\end{align*}

As for $J_2$, using the fact that $\supp\xi\subset B_3$, we estimate
\begin{align*}
&|J_2|=\left|\int_{\ern}\int_{\ern}|\delta_{x,y}u|^{p-2}\delta_{x,y}u(\delta_{x,y}\eta^{\overline{p}(p-1)})\xi(y)\dfrac{\phi(\delta_{x,y}u)}{|\delta_{x,y}u|^{p-2}\delta_{x,y}u}K(x,y)\,dx\,dy\right|\\
&\leq \int_{\ern}\int_{\ern}\dfrac{|\delta_{x,y}u|^{p-1}|(\delta_{x,y}\eta^{\overline{p}(p-1)})\xi(y)|}{|x-y|^{n+sp}}\,dx\,dy\\
&\leq \int_{B_3}\int_{\ern}\dfrac{|\delta_{x,y}u|^{p-1}|(\delta_{x,y}\eta^{\overline{p}(p-1)})\xi(y)|}{|x-y|^{n+sp}}\,dx\,dy\\
&=\int_{B_3}\int_{B_4}\dfrac{|\delta_{x,y}u|^{p-1}|(\delta_{x,y}\eta^{\overline{p}(p-1)})\xi(y)|}{|x-y|^{n+sp}}\,dx\,dy+\int_{B_3}\int_{\ern\setminus B_4}\dfrac{|\delta_{x,y}u|^{p-1}|(\delta_{x,y}\eta^{\overline{p}(p-1)})\xi(y)|}{|x-y|^{n+sp}}\,dx\,dy\\
& \eqqcolon I_1+I_2,
\end{align*}
where in the third line, we have used the fact that $\supp\xi\subset B_3$.

For $I_1$, using H\"{o}lder's inequality, we have
\begin{align*}
I_1&\leq \left(\int_{B_4}\int_{B_3}\dfrac{|\delta_{x,y}u|^p}{|x-y|^{n+sp}}\,dy\,dx\right)^{\frac{p-1}{p}}\left(\int_{B_4}\int_{B_3}\dfrac{|(\delta_{x,y}\eta^{\overline{p}(p-1)})\xi(y)|^p}{|x-y|^{n+sp}}\,dy\,dx\right)^{\frac{1}{p}}.
\end{align*}
A direct computation yields
\begin{align*}
\int_{B_4}\int_{B_3}\dfrac{|\delta_{x,y}u|^p}{|x-y|^{n+sp}}\,dy\,dx\leq [u]^{p}_{W^{s,p}(B_{4})}.
\end{align*}
Also, using \eqref{eq:eta} and Lemma \ref{lem:Pzero} with $\xi\in C^{\infty}_{c}(B_3)$, we estimate
\begin{align*}
\int_{B_4}\int_{B_3}\dfrac{|(\delta_{x,y}\eta^{\overline{p}(p-1)})\xi(y)|^p}{|x-y|^{n+sp}}\,dy\,dx & \leq c\int_{B_4}\int_{B_3}|x-y|^{p-sp-n}|\xi(y)|^p\,dy\,dx \\
& \leq c\int_{B_3}|\xi(y)|^p\,dy\leq c\,t^{(t-s)p}[\xi]^{p}_{W^{t,p}(\ern)}
\end{align*}
with some $c=c(n,s,t,p)$. Combining the above three displays, we get
\begin{align*}
I_1\leq c\,[u]^{p-1}_{W^{s,p}(\Omega)}[\xi]_{W^{t,p}(\ern)}.
\end{align*}

For $I_2$, using the triangle inequality and the fact that $|\delta_{x,y}(\eta^{\overline{p}(p-1)})|\leq 2$ for $y\in B_3$ and $x\in\ern\setminus B_4$ with $B_3\Subset B_4$, we estimate
\begin{align*}
I_2&\leq\int_{B_3}\left(\int_{\ern\setminus B_4}|(\delta_{x,y}\eta^{\overline{p}(p-1)})\xi(y)|\dfrac{(|u(x)-(u)_{B_1}|+|u(y)-(u)_{B_1}|)^{p-1}}{|x-y|^{n+sp}}\,dx\right)\,dy\\
&\leq c\int_{B_3}\left(\int_{\ern\setminus B_4}|(\delta_{x,y}\eta^{\overline{p}(p-1)})\xi(y)|\dfrac{|u(x)-(u)_{B_1}|^{p-1}}{|x-y|^{n+sp}}\,dx\right)\,dy\\
&\quad+c\int_{B_3}\left(\int_{\ern\setminus B_4}|(\delta_{x,y}\eta^{\overline{p}(p-1)})\xi(y)|\dfrac{|u(y)-(u)_{B_1}|^{p-1}}{|x-y|^{n+sp}}\,dx\right)\,dy\\
&\leq \left(\int_{B_3}|\xi(y)|\,dy\right)\tail(u-(u)_{B_1};B_4)^{p-1}+\int_{B_3}|\xi(y)||u(y)-(u)_{B_1}|^{p-1}\,dy.
\end{align*}
Here, using H\"{o}lder's inequality, Lemma \ref{lem:Pzero} and the fact that $\supp\xi\Subset B_3 \Subset 2B_3 \Subset \Omega$, we estimate
\begin{align*}
\int_{B_3}|\xi(y)|\,dy \le c\left(\int_{B_3}|\xi(y)|^p\,dy\right)^{\frac{1}{p}}\leq c\,[\xi]_{W^{t,p}(\ern)}.
\end{align*}
In a similar way, using also Lemma \ref{lem:P}, we have
\begin{align*}
\int_{B_3}|\xi(y)||u(y)-(u)_{B_1}|^{p-1}\,dy&\leq\left(\int_{B_3}|\xi(y)|^p\,dy\right)^{\frac{1}{p}}\left(\int_{B_3}|u(y)-(u)_{B_1}|^{p}\,dy\right)^{\frac{p-1}{p}}\\
&\leq [\xi]_{W^{t,p}(\ern)}[u]^{p-1}_{W^{s,p}(B_{4})}
\end{align*}
with $c=c(n,s,p,t)$. Then we get
\begin{align*}
|J_2|\leq c\,[\xi]_{W^{t,p}(\ern)}\left([u]^{p-1}_{W^{s,p}(B_{4})}+\tail(u-(u)_{B_1};B_4)^{p-1}\right).
\end{align*}

For $J_3$, we first consider the case $p\in(1,2]$, in which there holds
\begin{align*}
|T(a+b)-T(a)|\eqsim |b|(|a|+|b|)^{p-2}\leq |b|^{p-1}.
\end{align*}
Using this together with $\supp\eta\Subset B_2\Subset B_3$, we estimate
\begin{align*}
|J_3|&\leq c\int_{\ern}\int_{\ern}\left(|\delta_{x,y}\eta^{\overline{p}}||u(y)-(u)_{B_1}|\right)^{p-1}|\delta_{x,y}\xi||x-y|^{-n-sp}\,dx\,dy\\
&=c\int_{B_3}\int_{\ern}\left(|\delta_{x,y}\eta^{\overline{p}}||u(y)-(u)_{B_1}|\right)^{p-1}|\delta_{x,y}\xi||x-y|^{-n-sp}\,dx\,dy\\
&\quad+c\int_{\ern\setminus B_3}\int_{\ern}\left(|\eta(x)^{\overline{p}}||u(y)-(u)_{B_1}|\right)^{p-1}|\delta_{x,y}\xi||x-y|^{-n-sp}\,dx\,dy \eqqcolon I_3+I_4
\end{align*}
for some $c=c(\Lambda,\lambda,p)$. For $I_3$, observe that \eqref{eq:t} implies $t>\max\{0,sp-p+1\}$ when $p\in(1,2]$ and that
\begin{align*}
|\delta_{x,y}\eta^{\overline{p}}|\lesssim |\eta|^{\overline{p}-1}\|D\eta\|_{L^{\infty}}|x-y|\lesssim |x-y|.
\end{align*}
Note also that $\supp\xi\Subset B_3\Subset B_4$.
Using these facts together with H\"{o}lder's inequality, we have
\begin{align*}
|I_3|&\leq c\int_{B_3}\int_{B_4}|\delta_{x,y}\eta^{\overline{p}}|^{p-1}|u(y)-(u)_{B_1}|^{p-1}|\delta_{x,y}^t\xi||x-y|^{-n-sp+t}\,dx\,dy\\
&\quad+c\int_{B_3}\int_{\ern\setminus B_4}|\eta(y)^{\overline{p}}|^{p-1}|u(y)-(u)_{B_1}|^{p-1}|\delta_{x,y}^t\xi||x-y|^{-n-sp+t}\,dx\,dy\\
&\leq c\left(\int_{B_4}\int_{B_4}|\delta^t_{x,y}\xi|^p\dfrac{dxdy}{|x-y|^n}\right)^{\frac{1}{p}}\left(\int_{B_4}\int_{B_4}|x-y|^{p'(t-sp+p-1)-n}|u(y)-(u)_{B_1}|^p\,dx\,dy\right)^{\frac{p-1}{p}}\\
&\quad+c\int_{B_3}|\eta(y)^{\overline{p}}|^{p-1}|u(y)-(u)_{B_1}|^{p-1}|\xi(y)|\left(\int_{\ern\setminus B_4}|x-y|^{-n-sp}\,dx\right)\,dy,
\end{align*}
where $\delta^{t}_{x,y}\xi \coloneqq \frac{\xi(x)-\xi(y)}{|x-y|^{t}}$. 
Here, using Lemma \ref{lem:P}, we have
\begin{align*}
&\left(\int_{B_4}\int_{B_4}|\delta^t_{x,y}\xi|^p\dfrac{dxdy}{|x-y|^n}\right)^{\frac{1}{p}}\left(\int_{B_4}\int_{B_4}|x-y|^{p'(t-sp+p-1)-n}|u(y)-(u)_{B_1}|^p\,dx\,dy\right)^{\frac{p-1}{p}}\\
&\quad \leq c\left(\int_{B_4}\int_{B_4}|\delta^t_{x,y}\xi|^p\dfrac{dxdy}{|x-y|^n}\right)^{\frac{1}{p}}\left(\int_{B_4}|u(y)-(u)_{B_1}|^p\,dy\right)^{\frac{p-1}{p}}\\
&\quad\leq c\,[u]^{p-1}_{W^{s,p}(B_{4})}[\xi]_{W^{t,p}(\ern)}.
\end{align*}
Also, using $B_3\Subset B_4$, H\"{o}lder's inequality, $0\leq\eta\leq 1$, Lemmas \ref{lem:P} and \ref{lem:Pzero}, we estimate
\begin{align*}
&\int_{B_3}|\eta(y)^{\overline{p}}|^{p-1}|u(y)-(u)_{B_1}|^{p-1}|\xi(y)|\left(\int_{\ern\setminus B_4}|x-y|^{-n-sp}\,dx\right)\,dy\\
&\quad\leq c\int_{B_3}|u(y)-(u)_{B_1}|^{p-1}|\xi(y)|\,dy\\
&\quad\leq c\left(\int_{B_3}|u(y)-(u)_{B_1}|^{p}\,dy\right)^{\frac{p-1}{p}}\left(\int_{B_3}|\xi(y)|^p\,dy\right)^{\frac{1}{p}}\leq c\,[u]^{p-1}_{W^{s,p}(B_{4})}[\xi]_{W^{t,p}(\ern)}
\end{align*}
so that, with $c=c(n,s,t,p)$, 
\begin{align*}
|I_3|\leq c\,[u]^{p-1}_{W^{s,p}(B_{4})}[\xi]_{W^{t,p}(\ern)}.
\end{align*}

For $I_4$, using the facts that $\supp\xi\Subset B_3\Subset B_4$ and that $\supp\eta\Subset B_2$, along with Lemma \ref{lem:Pzero},
\begin{align*}
& |I_4| \leq \int_{B_2}|\eta(x)^{\overline{p}}|^{p-1}|\xi(x)|\left(\int_{\ern\setminus B_3}\dfrac{|u(y)-(u)_{B_1}|^{p-1}}{|x_0-y|^{n+sp}}\,dy\right)\,dx\\
&\leq \int_{B_4}|\eta(x)^{\overline{p}}|^{p-1}|\xi(x)|\left(\int_{\ern\setminus B_3}\dfrac{|u(y)-(u)_{B_1}|^{p-1}}{|x_0-y|^{n+sp}}\,dy\right)\,dx\leq\tail(u-(u)_{B_1};B_3)^{p-1}[\xi]_{W^{t,p}(\ern)},
\end{align*}
where $x_0\in B_3$ is the center of $B_3$. Then for $J_3$, we obtain
\begin{align*}
|J_3|\leq c\,[\xi]_{W^{t,p}(\ern)}\left([u]^{p-1}_{W^{s,p}(B_{4})}+\tail(u-(u)_{B_1};B_3)^{p-1}\right).
\end{align*}

Thus we get
\begin{align*}
|\widetilde{\mathcal{L}}_{s}\widetilde{u}[\xi]|\leq c\,\left(\|\mathcal{L}_{s}u\|_{(W^{t,p}_{0}(B_{4}))^*}+[u]_{W^{s,p}(B_{4})}^{p-1}+\tail(u-(u)_{B_1};B_3)^{p-1}\right)[\xi]_{W^{t,p}(\ern)}
\end{align*}
with $c=c(n,s,p,t)$, which leads to the conclusion for $p\in(1,2)$.

When $p>2$, we only have to estimate $J_3$. In this case observe that, by using Young's inequality, 
\begin{align*}
|T(a+b)-T(a)|\eqsim |b|(|a|+|b|)^{p-2}\lesssim  |b||a|^{p-2}+|b|^{p-1}.
\end{align*}
Then we estimate
\begin{align*}
|J_3|&\leq c\int_{\ern}\int_{\ern}\left(|\delta_{x,y}\eta||u(y)-(u)_{B_1}|\right)|\eta(x)\delta_{x,y}u|^{p-2}|\delta_{x,y}\xi||x-y|^{-n-sp}\,dx\,dy\\
&\quad+c\int_{\ern}\int_{\ern}\left(|\delta_{x,y}\eta||u(y)-(u)_{B_1}|\right)^{p-1}|\delta_{x,y}\xi||x-y|^{-n-sp}\,dx\,dy\\
&\leq c\int_{B_3}\int_{\ern}\left(|\delta_{x,y}\eta||u(y)-(u)_{B_1}|\right)|\eta(x)\delta_{x,y}u|^{p-2}|\delta_{x,y}\xi||x-y|^{-n-sp}\,dx\,dy\\
&\quad+c\int_{B_3}\int_{\ern}\left(|\delta_{x,y}\eta||u(y)-(u)_{B_1}|\right)^{p-1}|\delta_{x,y}\xi||x-y|^{-n-sp}\,dx\,dy\\
&\quad+c\int_{\ern\setminus B_3}\int_{\ern}\left(|\delta_{x,y}\eta||u(y)-(u)_{B_1}|\right)|\eta(x)\delta_{x,y}u|^{p-2}|\delta_{x,y}\xi||x-y|^{-n-sp}\,dx\,dy\\
&\quad+c\int_{\ern\setminus B_3}\int_{\ern}\left(|\delta_{x,y}\eta||u(y)-(u)_{B_1}|\right)^{p-1}|\delta_{x,y}\xi||x-y|^{-n-sp}\,dx\,dy\eqqcolon I_5+I_6+I_7+I_8
\end{align*}
with $c=c(\Lambda,\lambda,p)$. For $I_5$, using the Lipschitz continuity of $\eta$, $2s-1<t<s$ from \eqref{eq:t}, $0\leq\eta\leq 1$, H\"{o}lder's inequality with exponents $\left(p,p,p/(p-2)\right)$, and Lemma \ref{lem:P}, we have
\begin{align*}
|I_5|&\leq \int_{B_3}\int_{\ern}|\delta_{x,y}^{1}\eta||u(y)-(u)_{B_1}||\eta(x)\delta_{x,y}^su|^{p-2}|\delta_{x,y}^{t}\xi||x-y|^{-n-2s+t+1}\,dx\,dy\\
&\leq \left(\int_{B_4}\int_{B_4}|\delta^t_{x,y}\xi|^p\dfrac{dxdy}{|x-y|^n}\right)^{\frac{1}{p}}\left(\int_{B_4}\int_{B_4}|x-y|^{(t-2s+1)p-n}|u(y)-(u)_{B_1}|^p\,dx\,dy\right)^{\frac{1}{p}}\\
&\quad\times\left(\int_{B_{4}}\int_{B_{4}}|\delta^s_{x,y}u|^p\dfrac{dxdy}{|x-y|^n}\right)^{\frac{p-2}{p}}\\
&\leq [u]^{p-1}_{W^{s,p}(B_{4})}[\xi]_{W^{t,p}(\ern)}
\end{align*}
for $c=c(n,s,p,\Lambda,\lambda,t)$. For $I_6$, if $\max\{0,2s-1\}<t$ then since $p\geq 2$, $t>\max\{0,sp-p+1\}$ holds. Using this, together with $|\delta_{x,y}\eta|\lesssim |x-y|$, H\"{o}lder's inequality, $\supp\eta\Subset B_2\Subset B_4$, $\supp\xi\Subset B_3\Subset B_4$, $0\leq\eta\leq 1$, and Lemmas \ref{lem:P} and \ref{lem:Pzero}, we have
\begin{align*}
|I_6|&\leq c\int_{B_3}\int_{B_4}|\delta_{x,y}\eta|^{p-1}|u(y)-(u)_{B_1}|^{p-1}|\delta_{x,y}^t\xi||x-y|^{-n-sp+t}\,dx\,dy\\
&\quad+c\int_{B_3}\int_{\ern\setminus B_4}|\eta(y)|^{p-1}|u(y)-(u)_{B_1}|^{p-1}|\delta_{x,y}^t\xi||x-y|^{-n-sp+t}\,dx\,dy\\
&\leq c\left(\int_{B_4}\int_{B_4}|\delta^t_{x,y}\xi|^p\dfrac{dxdy}{|x-y|^n}\right)^{\frac{1}{p}}\left(\int_{B_4}\int_{B_4}|x-y|^{(t-sp+p-1)-n}|u(y)-(u)_{B_1}|^p\,dx\,dy\right)^{\frac{p-1}{p}}\\
&\quad+c\int_{B_3}|\eta(y)|^{p-1}|u(y)-(u)_{B_1}|^{p-1}|\xi(y)|\left(\int_{\ern\setminus B_4}|x-y|^{-n-sp}\,dx\right)\,dy\\
&\leq [u]^{p-1}_{W^{s,p}(B_{4})}[\xi]_{W^{t,p}(\ern)}
\end{align*}
with $c=c(n,s,p,\Lambda,\lambda,t)$. For $I_7$, from $\supp\eta\Subset B_2$, $\supp\xi\Subset B_3$, $\delta_{x,y}u=\delta_{x,y}(u-(u)_{B_1})$, $0\leq\eta\leq 1$, triangle and Young's inequality, $B_2\Subset B_3$, Lemmas \ref{lem:P} and \ref{lem:Pzero}, we have
\begin{align*}
|I_7|&\leq c\int_{\ern\setminus B_3}\int_{B_2}\left(|\eta(x)||u(y)-(u)_{B_1}|\right)|\eta(x)\delta_{x,y}u|^{p-2}|\xi(x)||x-y|^{-n-sp}\,dx\,dy\\
&\leq c\int_{B_2}|\xi(x)|\left(\int_{\ern\setminus B_3}\dfrac{|u(y)-(u)_{B_1}|^{p-1}}{|x_0-y|^{n+sp}}\,dy+|u(x)-(u)_{B_1}|^{p-1}\right)\,dx\\
&\leq c\,([u]_{W^{s,p}(B_{4})}^{p-1}+\tail(u-(u)_{B_1};B_3)^{p-1})[\xi]_{W^{t,p}(\ern)}
\end{align*}
with $c=c(n,s,p,\Lambda,\lambda,t)$. For $I_8$, $\supp\eta\Subset B_2$, $\supp\xi\Subset B_3$, $0\leq\eta\leq 1$, and Lemma \ref{lem:Pzero}, we get
\begin{align*}
|I_8|&=\int_{\ern\setminus B_3}\int_{\ern}\left(|\eta(x)||u(y)-(u)_{B_1}|\right)^{p-1}|\delta_{x,y}\xi||x-y|^{-n-sp}\,dx\,dy\\
&\leq \int_{B_2}|\eta(x)|^{p-1}|\xi(x)|\left(\int_{\ern\setminus B_3}\dfrac{|u(y)-(u)_{B_1}|^{p-1}}{|x_0-y|^{n+sp}}\,dy\right)\,dx\\
&\leq \int_{B_4}|\eta(x)|^{p-1}|\xi(x)|\left(\int_{\ern\setminus B_3}\dfrac{|u(y)-(u)_{B_1}|^{p-1}}{|x_0-y|^{n+sp}}\,dy\right)\,dx\leq\tail(u-(u)_{B_1};B_3)^{p-1}[\xi]_{W^{t,p}(\ern)}
\end{align*}
$c=c(n,s,p,\Lambda,\lambda,t)$. From the last five displays, when $p>2$ we estimate $J_3$ as
\begin{align*}
|J_3|\leq c\,[\xi]_{W^{t,p}(\ern)}\left([u]^{p-1}_{W^{s,p}(B_{4})}+\tail(u-(u)_{B_1};B_3)^{p-1}\right).
\end{align*}
Combining the estimates found for $J_1$, $J_2$ and $J_3$, we get
\begin{align*}
|\widetilde{\mathcal{L}}_{s}\widetilde{u}[\xi]|\leq c\,\left(\|\mathcal{L}_{s}u\|_{(W^{t,p}_{0}(B_{4}))^*}+[u]_{W^{s,p}(B_{4})}^{p-1}+\tail(u-(u)_{B_1};B_3)^{p-1}\right)[\xi]_{W^{t,p}(\ern)},
\end{align*}
which leads to the conclusion for $p\geq 2$.
\end{proof}

\begin{proof}[Proof of Theorem \ref{thm:self2}]
We can assume $r=1$ by scaling argument Lemma \ref{lem:scaling}.
Let $0<s<1<p<\infty$. For $f\in (W^{s,p}_0(\Omega))^*$, let $u\in W^{s,p}_{\loc}(\Omega)\cap L^{p-1}_{sp}(\ern)$ be a weak solution to \eqref{eq:fracp} in $\Omega$ under assumptions \eqref{eq:K} and \eqref{eq:phi}. Without loss of generality, we assume $B_{4}\Subset\Omega$. For $0<\rho\leq\frac{1}{2}\dist(B_4,\Omega)$, let $u_{\rho}\in W^{s,p}(B_4)$ be the weak solution to  
\begin{equation}\label{eq:reg3}
\left\{
\begin{aligned}
\mathcal{L}_{s,\rho}u_{\rho}&=f_{\rho}&\text{in }&B_4,\\
u_{\rho}&=u&\text{in }&\ern\setminus B_4,
\end{aligned}
\right.
\end{equation}
where $\mathcal{L}_{s,\rho}$ and $f_{\rho}$ are defined in \eqref{eq:Lepsilon} and \eqref{eq:fepsilon}, respectively. Then by Lemma \ref{lem:local2}, there exists $w_{\rho}\in W^{s,p}(\ern)$ so that
\begin{itemize}
\item[(1)] $w_{\rho}\equiv u_{\rho}$ in $B_1$,
\item[(2)] $\supp w_{\rho}\Subset B_2$,
\item[(3)] $[w_{\rho}]_{W^{s,p}(\ern)}\leq c\,[u_{\rho}]_{W^{s,p}(B_{3})}$,
\item[(4)] for any $t>0$ with $t\in (1-(1-s)\min\{p,2\},s)$,
\begin{align*}
& \|\widetilde{\mathcal{L}}_{s,\rho}w_{\rho}\|_{(W^{t,p}_0(B_{3}))^*} \\
& \leq c\,\left(\|\mathcal{L}_{s,\rho}u_{\rho}\|_{(W^{t,p}_0(B_{4}))^*}+[u_{\rho}]^{p-1}_{W^{s,p}(B_{4})}+\tail(u_{\rho}-(u_{\rho})_{B_{1}};B_{3})^{p-1}\right)
\end{align*}
with $c=c(n,s,t,p,\Lambda,\lambda)$, where $\mathcal{L}_{s,\rho}$ is defined in \eqref{eq:Lepsilon}, and
\begin{align*}
\widetilde{\mathcal{L}}_{s,\rho}w_{\rho}[\xi]&=\int_{\ern}\int_{\ern}\phi(\delta_{x,y}w_{\rho})\delta_{x,y}\xi \widetilde{K}(x,y)\,dx\,dy
\end{align*}
with some measurable, nonnegative, symmetric kernel $\widetilde{K}:\ern\times\ern\rightarrow\er$ satisfying
\begin{align*}
\dfrac{1}{\Lambda\lambda^2|x-y|^{n+sp}}\leq\widetilde{K}(x,y)\leq \dfrac{\Lambda\lambda^2}{|x-y|^{n+sp}}.
\end{align*}
\end{itemize}

We write $\Gamma_s \coloneqq [w_{\rho}]^p_{W^{s,p}(B_{3})}$ and $\Gamma_{s+\ep} \coloneqq [w_{\rho}]^p_{W^{s+\ep,p}(B_{3})}$. Note that $\Gamma_{s+\ep}<\infty$ from Lemma \ref{lem:hd}. Then since $\supp w_{\rho}\subset B_2$, similarly to \eqref{eq:gamma.s.ep} in the proof of Theorem \ref{thm:self}, we arrive at
\begin{align*}
\Gamma_{s+\ep}\leq \,[w_{\rho}]^p_{W^{s+\ep,p}(B_{2})}+c\,\Gamma_s
\end{align*}
with $c=c(n,s,p)$. We use Lemma \ref{lem:wsp} and Lemma \ref{lem:P} to get
\begin{align*}
\begin{split}
\Gamma_{s+\ep}\leq c\delta^{p}\Gamma_{s+\ep}+c_{\delta}\Gamma_{s}+c_{\delta}\sup_{\xi}\left(\widetilde{\mathcal{L}}_{s,\ep,B_{3},\rho}w_{\rho}[\xi]\right)^{\frac{p}{p-1}}+c\,\tail(w_{\rho}-(w_{\rho})_{B_{1}};B_{3})^p,
\end{split}
\end{align*}
where the supremum is taken over all $\xi\in C^{\infty}_{c}(B_{5/2})$ and $[\xi]_{W^{s+\ep,p}(\ern)}\leq 1$ with $c=c(n,s,p,\Lambda,\lambda)$, $c_{\delta}=c(n,s,p,\Lambda,\lambda,\delta)$. Now choosing $\delta=\delta(n,s,p,\Lambda,\lambda)$ sufficiently small, we have
\begin{align*}
\begin{split}
\Gamma_{s+\ep}\leq c\Gamma_{s}+c\sup_{\xi}\left(\widetilde{\mathcal{L}}_{s,\ep,B_{3},\rho}w_{\rho}[\xi]\right)^{\frac{p}{p-1}}+c\,\tail(w_{\rho}-(w_{\rho})_{B_{1}};B_{3})^p.
\end{split}
\end{align*}

From Lemma \ref{lem:comm}, for any $\ep\in(0,\frac{\min\{s,1-s\}}{2(p+2)}]$, it follows that
\begin{align*}
\Gamma_{s+\ep}\leq c\sup_{\xi}\left(\widetilde{\mathcal{L}}_{s,B_{3},\rho}w_{\rho}[(-\Delta)^{\frac{\ep p}{2}}\xi]\right)^{\frac{p}{p-1}}+c\,\ep\Gamma_{s+\ep}+c\,\Gamma_{s}+c\,\tail(w_{\rho}-(w_{\rho})_{B_{1}};B_{3})^p
\end{align*}
with $c=c(n,s,p,\Lambda,\lambda)$. Choosing $\ep_0=\ep_0(n,s,p,\Lambda,\lambda)\leq \frac{\min\{s,1-s\}}{2(p+2)}$ small, for any $\ep\in(0,\ep_0]$
\begin{align}\label{eq:pfthm2}
\Gamma_{s+\ep}\leq c\sup_{\xi}\left(\widetilde{\mathcal{L}}_{s,B_{3},\rho}w_{\rho}[(-\Delta)^{\frac{\ep p}{2}}\xi]\right)^{\frac{p}{p-1}}+c\,\Gamma_{s}+c\,\tail(w_{\rho}-(w_{\rho})_{B_{1}};B_{3})^p
\end{align}
holds with $c=c(n,s,p,\Lambda,\lambda)$. Now changing the term $(-\Delta)^{\frac{\ep p}{2}}\xi$ to an appropriate test function as we did in the proof of Theorem \ref{thm:self}, we get
\begin{align*}
\left|\widetilde{\mathcal{L}}_{s,B_3,\rho}w_{\rho}[(-\Delta)^{\frac{\ep p}{2}}\xi-\psi]\right| = \left|\widetilde{\mathcal{L}}_{s,B_3,\rho}w_{\rho}[(1-\eta_{B_{5/2}})(-\Delta)^{\frac{\ep p}{2}}\xi]\right| \leq c\,[w_{\rho}]^{p-1}_{W^{s,p}(B_{3})}.
\end{align*}
Hence, the estimate in \eqref{eq:pfthm2} becomes 
\begin{align*}
\begin{split}
\Gamma_{s+\ep}&\leq c\,\Gamma_{s}+c\left(\sup\left\{|\widetilde{\mathcal{L}}_{s,B_{3},\rho}w_{\rho}[\psi]|:\psi\in C^{\infty}_{c}(B_{5/2}),[\psi]_{W^{s-\ep(p-1),p}(\ern)}\leq 1\right\}\right)^{\frac{p}{p-1}}\\
&\quad+c\,\tail(w_{\rho}-(w_{\rho})_{B_{1}};B_{3})^p\\
&\leq c\,\Gamma_{s}+c\left(\sup\left\{|\widetilde{\mathcal{L}}_{s,\rho}w_{\rho}[\psi]|:\psi\in C^{\infty}_{c}(B_{5/2}),[\psi]_{W^{s-\ep(p-1),p}(\ern)}\leq 1\right\}\right)^{\frac{p}{p-1}}\\
&\quad+c\,\tail(w_{\rho}-(w_{\rho})_{B_{1}};B_{3})^p
\end{split}
\end{align*}
with $c=c(n,s,p,\Lambda,\lambda)$, where for the last line we estimate as follows. Define $\phi_{\zeta}$ as in \eqref{eq:psi.zeta} with $v$ replaced by $w_{\rho}$. With H\"{o}lder inequality, we have
\begin{align*}
\begin{split}
&\left|\widetilde{\mathcal{L}}_{s,B_{3},\rho}w_{\rho}[\phi_{\zeta}]-\widetilde{\mathcal{L}}_{s,\rho}w_{\rho}[\phi_{\zeta}]\right|\\
&\leq c\,\int_{\ern\setminus B_{4r}}\int_{B_{2}}\dfrac{|\delta_{x,y}w_{\rho}|^{p-1}|\delta_{x,y}\phi_{\zeta}|}{|x-y|^{n+(s+\ep)p}}\,dx\,dy\\
&\leq c\,\left(\int_{\ern\setminus B_{4}}\int_{B_{2}}\dfrac{|w_{\rho}(x)-(w_{\rho})_{B_{r}}|^{p}}{|x-y|^{n+(s+\ep)p}}\,dx\,dy\right)^{\frac{p-1}{p}}\left(\int_{\ern\setminus B_{4}}\int_{B_{2}}\dfrac{|\phi_{\zeta}(x)|^p}{|x-y|^{n+(s+\ep)p}}\,dx\,dy\right)^{\frac{1}{p}}\\
&\quad+c\,\int_{B_{2}}|\phi_{\zeta}(x)|\int_{\ern\setminus B_{4}}\dfrac{|w_{\rho}(y)-(w_{\rho})_{B_{r}}|^{p-1}}{|x-y|^{n+(s+\ep)p}}\,dy\,dx\\
&\leq c\,\left(\int_{B_{2}}|w_{\rho}(x)-(w_{\rho})_{B_{1}}|^{p}\,dx\right)^{\frac{p-1}{p}}\left(\int_{B_{2}}|\phi_{\zeta}(x)|^p\,dx\right)^{\frac{1}{p}}\\
&\quad+c\,\tail(w_{\rho}-(w_{\rho})_{B_{1}};B_{4})^{p-1}\int_{B_{2}}|\phi_{\zeta}(x)|\,dx.
\end{split}
\end{align*}
Here, 
\begin{align*}
\int_{B_{2r}}|\phi_{\zeta}(x)|^p\,dx\leq c\,\int_{B_{4r}}|w_{\rho}-(w_{\rho})_{B_{r}}|^p\,dx,
\end{align*}
where we have used the proof of \cite[Eq. (5) in Page 715]{Evans}. Similarly, with Young's inequality,
\begin{align*}
\tail(w_{\rho}-(w_{\rho})_{B_{1}};B_{4})^{p-1}\int_{B_{2}}|\phi_{\zeta}(x)|\,dx&=\tail(w_{\rho}-(w_{\rho})_{B_{1}};B_{4})^{p-1}\left(\mean{B_{2}}|\phi_{\zeta}(x)|^p\,dx\right)^{\frac{1}{p}}\\
&\leq c\,\tail(w_{\rho}-(w_{\rho})_{B_{1}};B_{4})^{p-1}\left(\mean{B_{2}}|w_{\rho}-(w_{\rho})_{B_{1}}|^p\,dx\right)^{\frac{1}{p}}\\
&\leq c\,\tail(w_{\rho}-(w_{\rho})_{B_{1}};B_{4})^p+c\,\int_{B_{2}}|w_{\rho}-(w_{\rho})_{B_{1}}|^p\,dx.
\end{align*}
Then with \eqref{lem:wsp}, we get the conclusion.

Therefore, together with (1)--(4) above, we obtain 
\begin{align*}
\begin{split}
&[u_{\rho}]^p_{W^{s+\ep,p}(B_1)} \leq [w_{\rho}]^p_{W^{s+\ep,p}(B_3)}\\
&\leq c\,[w_{\rho}]^p_{W^{s,p}(B_3)}+c\,\left(\sup\left\{\left|\widetilde{\mathcal{L}}_{s\rho}w_{\rho}[\psi]\right|:\psi\in C^{\infty}_{c}(B_{5/2}),[\psi]_{W^{s-\ep(p-1),p}(\ern)}\leq 1\right\}\right)^{\frac{p}{p-1}}\\
&\quad\quad+c\,\tail(w_{\rho}-(w_{\rho})_{B_{1}};B_{3})^p\\
&\leq c\,[u_{\rho}]_{W^{s,p}(B_4)}^p+c\,\|\mathcal{L}_{s,\rho}u_{\rho}\|_{(W^{s-\ep(p-1),p}_0(B_4))^*}^{\frac{p}{p-1}}+c\,\tail(w_{\rho}-(w_{\rho})_{B_{1}};B_{3})^p\\
&= c\,[u_{\rho}]_{W^{s,p}(B_{4})}^p+c\,\|f_{\rho}\|_{(W^{s-\ep(p-1),p}_0(B_{4}))^*}^{\frac{p}{p-1}}+c\,\tail(w_{\rho}-(w_{\rho})_{B_{1}};B_{3})^p
\end{split}
\end{align*}
with $c=c(n,s,p,\Lambda,\lambda)$, where for the last line \eqref{eq:reg3} is used. Now similar to the limiting process for $\rho$ used in the proof of Theorem \ref{thm:self}, employing Lemma \ref{lem:conv} instead of Lemma \ref{lem:conv2} and observing
\begin{align*}
&\tail(w-(w)_{B_{1}}-(w_{\rho}-(w_{\rho})_{B_{1}});B_{3}))^{p-1}\\
&\quad\leq c\int_{\ern}\frac{|w-(w)_{B_{1}}-(w_{\rho}-(w_{\rho})_{B_{1}})|^{p-1}}{(1+|x|)^{n+sp}}\,dx\leq c\int_{\ern}\frac{|w-w_{\rho}|^{p-1}}{(1+|x|)^{n+sp}}\,dx\rightarrow 0
\end{align*}
as $\rho\rightarrow 0$, $u\in W^{s+\ep,p}(\Omega_1)$ holds. Therefore, we get
\begin{align*}
[u]_{W^{s+\ep,p}(B_1)}&=\liminf_{\rho\rightarrow 0}[u_{\rho}]_{W^{s+\ep,p}(B_1)}\\
&\leq \liminf_{\rho\rightarrow 0}\left(c\,[u_{\rho}]_{W^{s,p}(B_{4})}^p+c\,\|\mathcal{L}_{s,\rho}u_{\rho}\|_{(W^{s-\ep(p-1),p}_0(B_{4}))^*}^{\frac{p}{p-1}}+c\,\tail(w_{\rho}-(w_{\rho})_{B_{1}};B_{3})\right)\\
&\leq c\,[u]_{W^{s,p}(B_{4})}^p+c\,\|\mathcal{L}_{s}u\|_{(W^{s-\ep(p-1),p}_0(B_{4}))^*}^{\frac{p}{p-1}}+c\,\tail(w-(w)_{B_{1}};B_{3}).
\end{align*}
This is the first estimate together with \eqref{eq:est} in Theorem \ref{thm:self2}. The second one together with \eqref{eq:est2} follows from Lemma \ref{lem:SP} with Lemma \ref{lem:P}.
\end{proof}

\subsection*{Conflict of interest} The authors declare that they have no conflict of interest.

\subsection*{Data availability} Data sharing not applicable to this article as no datasets were generated or analyzed during the current study.

\subsection*{Acknowledgements} We are grateful to Harsh Prasad for suggesting the idea of proof using Mikhlin multiplier theorem in Lemma \ref{lem:log}.

\providecommand{\bysame}{\leavevmode\hbox to3em{\hrulefill}\thinspace}
\providecommand{\MR}{\relax\ifhmode\unskip\space\fi MR }
\providecommand{\MRhref}[2]{%
  \href{http://www.ams.org/mathscinet-getitem?mr=#1}{#2}
}
\providecommand{\href}[2]{#2}

\end{document}